\documentclass[10pt]{amsart}

\usepackage{amsmath, amsthm, amssymb, xcolor, graphicx, enumitem, 
everypage, datetime, array} 
\usepackage{wrapfig}

\newcounter{citedtheorems}

\newtheorem{defn}{Definition}[section]
\newtheorem{defn-star}[defn]{$^\star$Definition}

\newtheorem{theorem}[defn]{Theorem}
\newtheorem*{thm-nn}{Theorem}
\newtheorem*{theorem-m}{Theorem \ref{main-theorem}}
\newtheorem*{thm-p2a}{Theorem \ref{t:p2a}}
\newtheorem{prob}[defn]{Problem}

\newtheorem*{thm-seq}{Theorem \ref{t:seq}}
\newtheorem*{thm-e}{Theorem}
\newtheorem*{thm-m}{Main Theorem}
\newtheorem*{theorem-abs1}{Theorem \ref{ind-theorem}}
\newtheorem*{theorem-abs2}{Theorem \ref{a23}}
\newtheorem*{theorem-abs3}{Theorem \ref{ind-new}}
\newtheorem*{theorem-abs4}{Theorem \ref{m1}}
\newtheorem*{thm-x}{Theorem}
\newtheorem{thm-lit}[citedtheorems]{Theorem}
\newtheorem{defn-lit}[citedtheorems]{Definition}
\newtheorem{fact-lit}[citedtheorems]{Fact}
\newtheorem{fact}[defn]{Fact}
\newtheorem{cor}[defn]{Corollary}

\newtheorem{defn-claim}[defn]{Definition/Claim}

\newtheorem*{defn-in}{Definition \arabic{section}.\arabic{equation}}

\newtheorem*{claim-in}{Claim \arabic{section}.\arabic{equation}}

\newtheorem{concl}[defn]{Conclusion}
\newtheorem{conv}[defn]{Convention}
\newtheorem{claim}[defn]{Claim}
\newtheorem{claim-star}[defn]{$^\star$Claim}

\newtheorem{lemma}[defn]{Lemma}
\newtheorem{obs}[defn]{Observation}
\newtheorem{obs-star}[defn]{$^\star$Observation}
\newtheorem{rmk}[defn]{Remark}

\newtheorem{ntn}[defn]{Notation}
\newtheorem{disc}[defn]{Discussion}
\newtheorem{disc-star}[defn]{$^\star$Discussion}

\newtheorem{expl}[defn]{Example}

\newtheorem{qst}[defn]{Question}

\newtheorem*{old-prob}{Earlier Problem}

\newcommand{\dual}{\operatorname{dual}}
\newcommand{\lost}{\L os' }
\newcommand{\los}{\L o\'s }
\newcommand{\br}{\vspace{2mm}}
\newcommand{\sbr}{\vspace{1mm}}
\newcommand{\mfa}{\mathfrak{a}}
\newcommand{\mfb}{\mathfrak{b}}
\newcommand{\mce}{\mathcal{E}}

\newcommand{\kim}{(\kappa, \mci, \bar{m})}

\newcommand{\mcq}{\cQ}
\newcommand{\mcp}{\cP}

\newcommand{\prm}{\operatorname{par}}

\newcommand{\bp}{\leq_{\operatorname{proj}}}
\newcommand{\proj}{\operatorname{proj}}
\newcommand{\omc}{\bar{\mc}}

\newcommand{\ml}{\mathcal{L}}
\newcommand{\tlf}{\trianglelefteq}

\newcommand{\rn}{\operatorname{range}}
\newcommand{\cf}{\operatorname{cof}}

\newcommand{\dom}{\operatorname{dom}}

\newcommand{\uu}{\mathcal{U}}

\newcommand{{\xw}}{\mathbf{w}}

\newcommand{\vv}{\mathcal{V}}

\newcommand{\mcx}{\mathcal{X}}

\newcommand{\cP}{\mathcal{P}}
\newcommand{\cQ}{\mathcal{Q}}
\newcommand{\rem}{\operatorname{rm}}

\newcommand{\sM}{\mathbf{M}}

\newcolumntype{L}{>{\centering\arraybackslash}m{4cm}}

\newcommand{\AP}{\operatorname{AP}}
\newcommand{\xn}{\mathfrak{n}}

\newcommand{\mcf}{\mathcal{F}}
\newcommand{\ii}{\mathbf{i}}

\newcommand{\tp}{\operatorname{tp}}

\newcommand{\mct}{\mathcal{T}}

\newcommand{\lgn}{\operatorname{lgn}}

\newcommand{\mcim}{\mci_{\mcm}}
\newcommand{\mcm}{\mathcal{M}}
\newcommand{\mcn}{\mathcal{N}}
\newcommand{\mcy}{\mathcal{Y}}

\newcommand{\de}{\mathcal{D}}

\newcommand{\Los}{\L o\'s }

\newcommand{\ts}{\mathbf{S}}

\newcommand{\trv}{\mathbf{t}} 

\newcommand{\jj}{\mathbf{j}}

\newcommand{\ba}{\mathfrak{B}}

\newcommand{\mch}{\mathcal{H}}

\newcommand{\rstr}{\upharpoonright}
\newcommand{\mci}{\mathcal{I}}

\newcommand{\mcr}{\mathcal{R}}

\newcommand{\mcs}{\mathcal{S}}

\newcommand{\suc}{\operatorname{ims}}

\newcommand{\vp}{\varphi}
\newcommand{\lcf}{\operatorname{lcf}}

\newcommand{\ma}{\mathbf{a}}
\newcommand{\mb}{\mathbf{b}}
\newcommand{\mc}{\mathbf{c}}
\newcommand{\md}{\mathbf{d}}

\newcommand{\mx}{\mathbf{x}}
\newcommand{\my}{\mathbf{y}}
\newcommand{\mz}{\mathbf{z}}

\newcommand{\fin}{\operatorname{FI}}

\newcommand{\trg}{T_{\mathbf{rg}}}

\newcommand{\mca}{\mathcal{A}}
\newcommand{\mcb}{\mathcal{B}}

\newcommand{\xm}{\mathfrak{m}}

\newcommand{\leaves}{\operatorname{lim}}

\title[Keisler's order is not simple (and simple theories may not be either)]{Keisler's order is not simple \\ (and simple theories may not be either)}

\author{M. Malliaris and S. Shelah}
\thanks{ \emph{Thanks}. Partially supported by 
 NSF CAREER award 1553653 
 and a Minerva Research Foundation membership at IAS (Malliaris), 
 and European Research Council grant 338821 (Shelah), which with NSF 1362974 supported  
visits of the authors.  
This is paper 1167 in Shelah's list.}

\address{Department of Mathematics, University of Chicago, 5734 S. University, Chicago, IL 60637, USA} 
\email{mem@math.uchicago.edu}

\address{Einstein Institute of Mathematics, Edmond J. Safra Campus, Givat Ram, The Hebrew
University of Jerusalem, Jerusalem, 91904, Israel, and Department of Mathematics,
Hill Center - Busch Campus, Rutgers, The State University of New Jersey, 110
Frelinghuysen Road, Piscataway, NJ 08854-8019 USA}
\email{shelah@math.huji.ac.il}
\urladdr{http://shelah.logic.at}

\begin{document}

\begin{abstract}  
Solving a decades-old problem we show that 
Keisler's 1967 order on theories has the maximum number of classes. In fact, it embeds $\mcp(\omega)/\operatorname{fin}$.  
The theories we build are simple unstable with no nontrivial forking, 
and reflect growth rates of sequences 
which may be thought of as densities of certain regular pairs, in the sense of Szemer\'edi's regularity lemma. The proof involves ideas from model theory, set theory, and finite combinatorics. 
\end{abstract}

\maketitle

Keisler's order is a longstanding classification problem in model theory,
introduced in 1967 \cite{keisler} as a possible way of comparing the complexity of theories. Informally, say  
$T_1 \tlf T_2$ if the regular ultrapowers of models of $T_1$ are more likely to be saturated than those of $T_2$.  
Keisler's paper established that there was a minimum class, containing algebraically closed fields of fixed characteristic, and a 
maximum class, containing Peano arithmetic. 
By work of the second author in the seventies, see \cite{Sh:c} Chapter VI, the union of the first two classes in Keisler's order gives 
an independent characterization of the \emph{stable theories}, which are fundamental to modern model theory.  
Recently there has been much progress; 
for an account of work in the last decade, and some applications, see e.g. \cite{MiSh:998} or \cite{mm-icm}. 

Among the questions raised by Keisler (see e.g. \cite{keisler2} p. 13) 
were how many classes the order had, whether it was linear, and what were syntactic characterizations of the minimum and maximum classes. 

As of 1978 \cite{Sh:c}, the number was at least four, linearly ordered. Several years ago we discovered 
infinitely many classes, in fact an infinite descending 
chain \cite{MiSh:1050}, using certain hypergraphs first studied by Hrushovski \cite{h:letter}.  Building on that construction,  
one can find conditional instances of nonlinearity (i.e. assuming a supercompact cardinal), 
as observed independently by Ulrich \cite{ulrich} and the authors \cite{MiSh:F1530}. 
Recently, we found unconditional (ZFC) instances of nonlinearity \cite{MiSh:1140}.  It would be consistent with these papers  
to conjecture that instances of nonlinearity were few, and that the number of equivalence classes was countable.  

In the present paper we prove, in ZFC, that Keisler's order has the maximum number of classes (continuum many), 
by constructing a new family of  
simple unstable theories with no nontrivial forking which reflect growth rates of 
certain sequences of densities of finite graphs, and  by developing new methods for building ultrafilters on Boolean algebras 
which carefully reflect these theories. (Both constructions seem quite flexible.  
Perhaps one reason some major structural conjectures about simple theories 
have remained stalled for decades is that simple unstable theories may have a much richer structure than previous examples suggest.)

The rough idea of our construction is as follows. 

We first build pairs of infinite, finitely branching trees with an edge relation between nodes of corresponding height which 
thins out in an appropriate way as the height grows. In our main case\footnote{This sketch describes theories 
with additional input from \S \ref{s:sizes}; the frame in \S \ref{s:nt} is more basic.}, whenever two nodes of height $k$ connect, 
their sets of immediate successors form a bipartite graph which, depending on the level, is either complete, or sparse and random (with a size and 
edge probability which is a function of the height).  Associated to each height is a notion of 
``small'' and ``large'' and the sparse graphs in question have the property that every small set of vertices 
has a common neighbor and no large set of vertices does.   These structures, called parameters, can be thought of as encoding reduced graphs for the 
models of simple theories we then construct. The \emph{level function} of a parameter is the 
set of heights at which we use sparse random (as opposed to complete) graphs; these are our choice of 
a way to track growth rates. 
To any such parameter we then associate a simple theory, essentially a kind of 
bipartite random graph filtered through unary predicates, which is simple unstable with no nontrivial forking. 
We prove that 
as the sequences of finite densities in the parameters vary sufficiently, as measured by the level functions, 
the associated theories have wildly different saturation behavior (a fundamentally infinitary phenomenon). 

What happens on the ultrafilter side?  For $\kappa$ regular and uncountable, 
we define a new chain condition to match the simple theories and which says, 
very roughly speaking, if we are given $\kappa$ positive elements of our Boolean algebra, then after moving to 
a subset $\uu$ also of size $\kappa$, for any finite $n$ not in some ideal (of which more soon),  
any finite $u \subseteq \uu$  which is ``large'' in the sense of level $n$ has a 
subset $v$ which is still ``large'' and  whose elements are all compatible.  The precise sense in which 
we choose our family of theories to be orthogonal to each other has to do with the fact that for any partition of 
our final set of parameters into $\mcm$ and $\mcn$,  the ideal of subsets of $\omega$ 
generated by the subsets where the level functions of elements of $\mcm$ are $1$, does not contain (mod finite)  
the set where the level function of $\xn$ is 1, for any $\xn \in \mcn$.  
(The idea is that if we preserve the chain condition using the ideal coming from $\mcm$, any 
future ultrafilter will omit at least one type in any theory from $\mcn$, since given any purported solution, the ultrafilter 
can concentrate too many of its conditions at points where they cannot all be satisfied.)
Essentially this allows us to 
build by induction a (non-free!) Boolean algebra and a filter (eventually an ultrafilter) on it, 
adding formal solutions to problems coming from theories from $\mcm$ at suitable inductive stages, 
while preserving the chain condition using the ideal coming from $\mcm$ 
which ensures omission of a type for any theory with a parameter from $\mcn$.  Previous model-theoretic constructions of 
ultrafilters were focused exclusively on free Boolean algebras; for details, see \S \ref{s:cc}.  

\S \ref{s:main} contains the statements of the main theorems. 

These results suggest that not only do model theoretic dividing lines predict  
jumps in the complexity of theorems in finite combinatorics (as e.g. in stable regularity \cite{MiSh:978}, or stable 
arithmetic regularity \cite{TW}), 
but also densities in the sense of finite combinatorics can control behavior of infinite models tightly enough 
that the resulting changes in complexity are detected by ultrafilters, so are candidates for model theoretic dividing lines.

The model theoretic, set theoretic, and combinatorial aspects of these constructions admit natural variations and 
raise interesting open questions, see \S \ref{s:open}.  

Is this the end of a certain line of work on this problem? We think inversely: it tells us where to look. 

We are grateful to very helpful questions and discussions after talks on a first version of this manuscript 
in summer and fall 2019 
which improved the presentation and inspired us to prove some substantial new theorems in \S \ref{s:p-omega}. In particular, 
we thank 
 M. Goldstern, I. Kaplan, M. Magidor, F. Parente, T. Scanlon, 
 C. Terry, and M. Viale. 

We are especially grateful to the anonymous referees for numerous, detailed, and extremely helpful comments on the manuscript.  
Thank you.

\setcounter{tocdepth}{1}

\tableofcontents

\section{Notation and conventions}

\begin{conv}
Unless otherwise stated, all graphs are simple graphs: no loops and no multiple edges.
\end{conv}

\begin{conv}
We will often write bipartite graphs as triples $(V,W,E)$, where $V$, $W$ are the sets of vertices and the $E \subseteq V \times W$ is the edges.  
We will call a bipartite graph \emph{complete} if $E = V \times W$, so in this case $E$ is asymmetric. 
\end{conv}

In this paper we will have both finite and infinite (possibly uncountable) random graphs; the infinite ones are random in the sense of model theory, which should not cause confusion.  The next two definitions, ordinary \ref{c:trg} and bipartite \ref{c:tbrg}, explain what this means.

\begin{conv}[The model-theoretic random graph] \label{c:trg}
``The theory of the random graph'' means the set of first-order axioms in the language with a binary relation symbol $E$, and equality, 
which say that $E$ is symmetric irreflexive, that there are infinitely many elements, and for any two finite disjoint sets $v,w$, there is a vertex $a$ such that 
$E(a,b)$ for all $b \in v$ and $\neg E(a,c)$ for all $c \in w$. 
\end{conv}

\begin{conv}[Infinite \emph{bipartite} random graphs] \label{c:tbrg}
If $(A,B,E)$ is a bipartite graph and $A, B$ are infinite, 
we may call it a \emph{bipartite random graph} to mean that the following two conditions hold: for any two finite disjoint $u,v \subseteq B$, there is $a \in A$ such that 
$\bigwedge_{b \in u}  E(a,b) \land \bigwedge_{b \in v} \neg E(a,b)$, and conversely, for any two finite disjoint $u, v \subseteq A$, 
there is $b \in B$ such that $\bigwedge_{a \in u}  E(a,b) \land \bigwedge_{a \in v} \neg E(a,b)$.  
\end{conv}

\begin{conv}[Trees] \label{c:trees}
 Recall that a tree is a partially ordered set such that 
the set of predecessors of any given node is well ordered, so in particular linearly ordered. In this paper, the partial order will always be given by 
initial segment, denoted $\tlf$. In this paper, all trees will be finitely branching, and 
all nodes of all trees will have finite height, so any 
tree will be either of finite or countable height. 
\end{conv}

\begin{rmk}
The symbol $\tlf$ is used in this paper to denote two unrelated kinds of partial orders: to denote the partial order on elements of a given tree 
and to denote Keisler's ordering on theories.  Since these two contexts never overlap, this should not cause any confusion. 
\end{rmk}

\begin{defn} \label{fact-iffa}  As usual, we denote by $\ba^0_{\alpha, {{\mu}}, \aleph_0}$ the free Boolean algebra generated by  
$\alpha$ independent partitions each of size ${{\mu}}$, and $\ba^1_{\alpha, {{\mu}}, \aleph_0}$ is its completion. 
\end{defn}

The last subscript, $\aleph_0$, in \ref{fact-iffa} refers to the fact that 
any intersection of $<\aleph_0$ elements from distinct partitions is nonempty.  Often it is understood and so not written.  
When it takes values other than $\aleph_0$, we usually refer to it as $\theta$ -- this was used, for instance, in  
\cite{MiSh:1140}. 
On the existence of $\ba^0_{2^\lambda, {{\mu}}}$, i.e. $\ba^0_{2^\lambda, {{\mu}}, \aleph_0}$,  
when $\lambda \geq \mu$ 
see e.g. Fichtenholz-Kantorovich, or Hausdorff, or \cite{Sh:c} Appendix Theorem 1.5.  When $\theta > \aleph_0$, the existence 
theorem requires $\lambda = \lambda^{<\theta} \geq \mu$. 
In this paper,  to find $D_0, \de_*, \jj$ as in \ref{d:built} below, we use the completion. 

\begin{defn} \label{d:fin} Let 
\[ \fin_{\mu,\theta}(\alpha) = \{  h : h \mbox{ is a function, } \dom(h) \subseteq \alpha, \rn(h) \subseteq \mu, |\dom(h)| < \theta. \} \]
When $\theta = \aleph_0$, as is often our case in this paper, we usually omit it. 
\end{defn}

This notation recalls e.g. \cite{Sh:c} Definition 3.6 p. 358; 
for a more detailed explanation, see \cite{MiSh:1140} \S 1.  In our context, the simple idea is as follows. 
We are dealing with, say, a Boolean algebra $\ba^1_{\alpha, \mu, \aleph_0}$ which is the completion of a Boolean algebra 
generated freely by $\alpha$ many independent partitions (=maximal antichains) each of size ${{\mu}}$.  
Notice that if we enumerate the generators as $\langle \mx_{\beta, \epsilon} : \beta < \alpha, \epsilon < \mu \rangle$ 
then the intersection of 
finitely many [or in general, fewer than $\theta$] elements $\mx_{\beta_i, \epsilon_i}$ $(i < i_* < \omega)$ will be nonzero if and only if 
we haven't chosen two distinct elements from the same antichain, that is, if $\{ (\beta_i, \epsilon_i) : i < i_* \}$ is 
a function. 
Elements arising from such intersections will be dense in the completion and merit compact notation. 
So we shall specify elements arising as such intersections 
by functions $f$ such that 
$\dom(f) \subseteq \alpha$, $|\dom(f)| < \aleph_0$, and $\rn(\alpha) \subseteq \mu$, where again, we think of such a function as 
selecting an element from each of a small number of antichains which are then intersected to form the nonzero element ``$\mx_f$''.  
Indeed, we could make ``$\mx_{\beta, \epsilon}$'' a special case of this notation by specifying that when 
the domain of the function has size $1$ we may drop the parentheses in $\mx_{\{ (\beta, \epsilon) \}}$. 
Thus \ref{d:fin}, and thus: 

\begin{conv} 
For $g \in \fin_{\mu,\theta}(\alpha)$, let $\mx_g$ denote the corresponding nonzero 
element of $\ba$. 
\end{conv}
 
As the generators are dense in the completion, we have:   

\begin{rmk} \label{d:1a}
Let $\ba = \ba^1_{\alpha, {{\mu}}, \aleph_0}$.   Then, in our notation, 
the elements of the form $\mx_f$  for  $f \in \fin_{\aleph_0}({\alpha})$  are dense in $\ba$. 
\end{rmk}

\begin{fact}[$\Delta$-system lemma, see e.g. Kunen \cite{kunen} III.6.15] \label{fact-d}
Let $\nu$ and $\kappa$ be regular cardinals such that $\aleph_0 \leq \nu < \kappa$. Assume that 
$(\forall \alpha < \kappa)(\alpha^{<\nu} < \kappa)$.  Let $\mca$ be a family of sets with $|\mca| = \kappa$, 
such that $|A| < \nu$ for all $A \in \mca$. Then there is a $\mcb \subseteq \mca$ of size $\kappa$ such that 
$\mcb$ forms a $\Delta$-system. 
\end{fact}

Note that the family of sets need not be subsets of $\kappa$; we place no restriction on their provenance, only 
restrictions on size of the family and size of the sets. 
We will mostly use the case $\nu = \aleph_0$:

\begin{cor} \label{delta-system} If $\kappa$ is an uncountable regular cardinal and $\mca$ is a family of $\kappa$ sets, 
all of them finite, there is $\mcb \subseteq \mca$ of size $\kappa$ which forms a $\Delta$-system. 
\end{cor}

A central definition in this paper will be \emph{Keisler's order}.    For more on the order, see \cite{keisler}, or for example the extended 
introduction to \cite{MiSh:1030}.  Some key points: 

\begin{defn}[Keisler's order, \cite{keisler}] \label{d:keisler}
Let $T_1, T_2$ be complete countable first-order theories. We say $T_1 \tlf T_2$ if for every infinite $\lambda$, every regular ultrafilter 
$\de$ on $\lambda$, every model $M_1 \models T_1$, and every model $M_2 \models T_2$, if 
$(M_2)^\lambda/\de$ is $\lambda^+$-saturated, then $(M_1)^\lambda/\de$ is $\lambda^+$-saturated. 
\end{defn}

Recall that the ultrafilter $\de$ on $\lambda$ is \emph{regular} if there exists a regularizing family, meaning 
$X = \{ X_\alpha : \alpha < \lambda \} \subseteq \de$ such that the intersection of any infinitely many elements of $X$ is empty. 
By a lemma of Keisler \cite[2.1]{keisler}, if $\de$ is regular, then the choice of $M_1$, $M_2$ in \ref{d:keisler} does not matter, 
up to elementary equivalence.  For more on regular ultrafilters, see \cite{ck} \S 4.3 and \S 6.1.  

Recall that a regular ultrafilter $\de$ on $\lambda$ is \emph{$\lambda^+$-good} if every $f: [\lambda]^{<\aleph_0} \rightarrow \de$ 
which is monotonic has a multiplicative refinement, that is, if $u \subseteq v$ implies $f(u) \supseteq f(v)$ for all 
$u, v \in [\lambda]^{<\aleph_0}$, then 
there exists $g: [\lambda]^{<\aleph_0} \rightarrow \de$ such that $g(u) \subseteq f(u)$ for all $u \in [\lambda]^{<\aleph_0}$ and 
$g(u) \cap g(v) = g(u \cup v)$ for all $u,v \in [\lambda]^{<\aleph_0}$. 

Keisler \cite{keisler} proved that good ultrafilters characterize the maximum class in Keisler's order: if $\de$ is a regular ultrafilter on $\lambda$, then 
$\de$ is $\lambda^+$-good if and only if $M^\lambda/\de$ is $\lambda^+$-saturated for some, equivalently every, model of every 
complete countable theory $T$.   By extension, 

\begin{conv} \label{c:112}
If $\de$ is a regular ultrafilter on $\lambda$ and $\kappa \leq \lambda$ and $T$ is a complete countable theory, we may say 
\[ \de \mbox{ is $(\kappa^+, T)$-good } \]
if for some, equivalently every, $M \models T$ we have that $M^\lambda/\de$ is $\kappa^+$-saturated. When $\kappa = \lambda$, 
we may just say ``\emph{$\de$ is good for $T$.}''   Note that the negation, ``$\de$ is not $(\kappa^+, T)$-good'' means that for 
some, equivalently every, $M \models T$, the ultrapower $M^\lambda/\de$ is not $\kappa^+$-saturated. 
\end{conv}

\vspace{5mm}


\section{New theories} \label{s:nt}

This section defines a new family of simple theories. 

Here is a brief overview. First we define \emph{parameters}, and from these, we define our new \emph{theories}. 
A parameter is a kind of template.  It is essentially 
a pair of finitely branching trees, of countable height, along with 
data about whether nodes in the left tree do or do not connect to nodes of the corresponding height (=level) in the 
right tree.  There are a series of coherence conditions which say things like: the branching gradually increases; 
in order for two nodes to connect, all of their initial segments must have been connected, but this isn't sufficient: two nodes at 
level $n$ which connect may have many immediate successors which don't connect, but at least a few must.  
There is also a \emph{level function}, which is $0$ or $1$ at each level, and $1$ infinitely often. 

In our main case in the present paper, the connection patterns in our trees will have a particularly elegant form.  
We will first specify that all nodes at level $n$, in both trees, have the same number of successors. 
We shall then choose in advance a distinguished, growing sequence of bipartite (fairly random) graphs
with the correct number of vertices, and recall we have the 
given level function. At each level $n$, if the level function is $1$ (an ``active level''), 
whenever two nodes are connected, we let the pattern of connections among their immediate successors 
be given by the distinguished bipartite graph for level $n+1$. If the level function is $0$ (a ``lazy level,'' the 
idea being there are no additional constraints introduced at this level), 
we let the pattern of connections among their immediate successors be given by a complete bipartite graph. 
Either way, if two nodes are not connected, there are no connections between their immediate successors. 
In such cases, we can reasonably say that the tree is \emph{built over a sequence of graphs}. 
However, the construction is more flexible, and doesn't require patterns of connections to be essentially invariants 
of the level; so in this section and the next, we work out the construction at a somewhat greater level of generality. 

After defining parameters, we will continue these remarks (before \ref{rmk:212}) to motivate the \emph{theories} built from 
given parameters.  Recall convention \ref{c:trees} on trees. 

\begin{ntn}[Notation for trees] \emph{ }
\begin{enumerate}
\item In this section, a tree will always denote a subset of ${^{\omega>}\omega}$, closed under initial segments, and partially ordered by 
initial segment, denoted $\tlf$.
\item For $\mct_i$ a tree and $k<\omega$, let $\mct_{i,k}$ denote the $k$th level of $\mct_i$, i.e 
\[ \mct_{i,k} = \mct_i \cap {^k \omega }. \]
That is, that any $\eta \in \mct_{i,k}$ has length $k$ and\footnote{Note that under this setup the ``0th level'' is a singleton, i.e 
$\mct_{1,0} = \mct_{2,0} = \{ \emptyset \}$.} is a function from $\{ 0, \dots, k-1 \}$ to $\omega$ $($so if $k=0$, then 
$\dom(\eta) = \emptyset$$)$. We may write 
$\eta(t)$ for the value of $\eta$ at $t \in\dom(\eta)$. 

\item Let $\mct_{i, \leq k}$ denote $\bigcup_{\ell \leq k} \mct_{i,\ell}$. 
\item For $\mct_i$ a tree and $\eta \in \mct_{i,k}$, denote the immediate successors of $\eta$ in $\mct_i$ by
\[ \suc_{\mct_i}(\eta) = \{ \eta^\prime \in \mct_{i,k+1} ~:~ \eta \tlf \eta^\prime \}. \]

\item For $\mct_i$ a tree, denote the leaves of $\mct_i$ by 
\[ \leaves(\mct_i) = \{ \eta \in {^\omega \omega} : \omega \rstr k \in \mct_{i,k} \mbox{ for all } k<\omega \}. \]
\end{enumerate}
\end{ntn}

\begin{defn} \label{d:level}
Call $\xi: \omega \rightarrow \{ 0, 1 \}$ a \emph{level function} if $\{ i<\omega : \xi(i) = 1 \}$ is infinite, and $($for convenience$)$
$\xi(0) = 1$.
\end{defn}

\noindent The idea of a level function, \ref{d:level}, will be that level $i$ of the tree is active if $\xi(i) = 1$,  
and not if $\xi(i) = 0$ (the ``lazy levels''), explained presently. 

The first main ingredient is that of a \emph{parameter} (basic parameter \ref{d:param}, parameter \ref{d:nice}) 
which will give us the blueprint on which a theory can be based. 

\begin{defn} \label{d:param}
A \emph{basic parameter} $\xm$ consists of a pair of trees $\mct_1$, $\mct_2$, a sequence of binary relations $\mcr_k$ for $k<\omega$, and a level function $\xi$, all satisfying the following. 
\begin{enumerate}
\item $\mct_1$ and $\mct_2$ are subtrees of ${^{\omega>}\omega}$ with finite splitting and no maximal node. \item For $k<\omega$, $\mcr_k \subseteq \mct_{1,k} \times \mct_{2,k}$, and for $k=0$ we have equality.
\item If $(\eta_1, \eta_2) \in \mcr_{k+1}$ then $(\eta_1 \rstr k, \eta_2 \rstr k) \in \mcr_k$.
\item If $(\eta_1, \eta_2 ) \in \mcr_k$, $\eta^\prime_2 \in \suc_{\mct_{2}}(\eta_2)$, then for at least two distinct 
$\eta^\prime_{1}, \eta^{\prime\prime}_{1,} \in \suc_{\mct_{1}}(\eta_{1})$ we have $(\eta^\prime_1, \eta^\prime_2) \in \mcr_{k+1}$ 
and $(\eta^{\prime\prime}_1, \eta^\prime_2) \in \mcr_{k+1}$, and the parallel for the trees reversed.  
[Informally, if two elements at one level are connected, then every immediate successor of one of them is connected to at least two 
immediate successors of the other.] 
\item If $\xi(k) = 0$, then $(\mct_{1,k+1}, \mct_{2,k+1}, \mcr_{k+1})$ adds no new constraints meaning:  if $(\eta_1, \eta_2) \in \mct_{1,k+1} 
\times \mct_{2,k+1}$ and $(\eta_1 \rstr k, \eta_2 \rstr k) \in \mcr_k$ then $(\eta_1, \eta_2) \in \mcr_{k+1}$.  

\item Let $\mcr = \bigcup_{k<\omega} \mcr_k$. 
\end{enumerate}
\end{defn}

\begin{cor} \label{d:fullness}
It follows from Definition $\ref{d:param}(4)$ that we may add:
\begin{enumerate}
\item[(7)] Fullness: For every $\eta \in \leaves(\mct_1)$, there are continuum many 
$\rho \in \leaves(\mct_{2})$ such that $(\eta \rstr k, \rho \rstr k) \in \mcr_k$ for all $k<\omega$. 
\\ Likewise, for every $\rho \in \leaves(\mct_2)$, there are continuum many 
$\eta \in \leaves(\mct_{1})$ such that $(\eta \rstr k, \rho \rstr k) \in \mcr_k$ for all $k<\omega$. 
\end{enumerate}
$($Of course, we could have just stated this as a separate axiom.$)$
\end{cor}

\begin{rmk} \label{r:level} On the level functions. 
\emph{Conditions $\ref{d:param}(5)$-$(6)$ tell us essentially that if 
 $\xi(i) = 0$, 
$\mcr_{i+1}$ is set by (\ref{d:param})(5) and contributes no new constraints: if two elements connect in $\mcr_i$, 
then $\mcr_{i+1}$ is a complete bipartite graph on their immediate successors, whereas if two elements don't connect in 
$\mcr_i$, $\mcr_{i+1}$ is an empty graph on their successors.  We call $i$ a ``lazy level'' $($we chose to say this 
about $i$, although we could have said this 
about $i+1$$)$.    
In contrast if $\xi(i) = 1$, we will have a lot of freedom in choosing $\mcr_{i+1}$, subject to \ref{d:param}(6) and \ref{d:nice}.   
The usefulness of this feature, the level function, will be more apparent starting in \S \ref{s:sizes} when we pattern the 
$\mcr_i$'s on tailor-made sequences of bipartite random graphs, and start comparing theories whose level functions 
are in some natural sense independent.}
\end{rmk}

Since we will be interested in varying the edge families $\mcr_k$, the following conditions will ensure there are a minimum 
of edges and edge coherence to define a model completion. In the rest of this paper, we will always assume them to be true. 
We could have included them in \ref{d:param}. 

\begin{defn} \label{d:nice} We say the basic parameter $\xm$ is a \emph{parameter} when, in addition:\footnote{We repeat  the conditions for both sides 
since $\mcr_k$ is not required to be symmetric.}

\begin{enumerate}
\item Left  extension:  if $k<\omega$, $\nu \in \mct_{2,k}$, $u \subseteq \mct_{1, k+1}$, $|u| \leq k$ satisfy 
\[ (\forall \eta \in u) [ ( \eta \rstr k, \nu) \in \mcr_k ]  \]
then there are $\geq k+1$ elements $\rho \in \suc_{\mct_{2}}(\nu)$
such that  
\[ (\forall \eta \in u)[ (\eta, \rho) \in \mcr_{k+1}]. \]

\item Right extension:  if $k<\omega$, $\nu \in \mct_{1,k}$, $u \subseteq \mct_{2, k+1}$, $|u| \leq k$ satisfy 
\[ (\forall \rho \in u) [ ( \nu, \rho \rstr k) \in \mcr_k ] \]
then there are $\geq k+1$ elements $\eta \in \suc_{\mct_{1}}(\nu)$
such that  
\[ (\forall \rho \in u)[ (\eta, \rho) \in \mcr_{k+1}]. \]

\end{enumerate}

\end{defn}

\begin{rmk}
Together, $\ref{d:param}$ 
and the extension axioms of $\ref{d:nice}$ imply that the branching of each $\mct_{\ell}$ at height $k$ is at least $k+1$. 
\end{rmk}

\begin{rmk}
Note that extension does \emph{not} require the elements in the set $u$ to have an immediate common predecessor. 
\end{rmk}

As the results of this paper indicate it may be interesting to further investigate theories in this region, we include two 
comments on alternative definitions. 

\begin{disc}
We could have weakened left and right extension by asking, e.g.: for every $k_1 <\omega$ there is $k_2 > k_1$ such that 
if $k_3 \geq k_2$, $\nu \in \mct_{3-\ell}, k_3$, $u \subseteq \mct_{\ell, k_3+1}$, $|u| \leq k_1$ satisfy, etc. With this we would gain a little 
in some places, e.g. $\ref{k11}$, and lose a little in others, e.g. $\ref{k17}$. 
The clean formulation in $\ref{d:nice}$ is sufficient for our purposes here. 
Informally, rather than working with a fixed branching and letting the number of connections be arbitrarily slow-growing, 
we encode $f(k) \geq k+1$ in our extension axioms and in the construction allow branching to be arbitrarily large. 
\end{disc}

\begin{disc}
Another variation we do not use here would be to say $\xm$ is \emph{very nice} when we may add a non-connection clause to the 
extension axioms, e.g. 
if $k<\omega$, $\nu \in \mct_{2,k}$, $u, v \subseteq \mct_{1, k+1}$, are disjoint, $|u \cup v| \leq k$ satisfy 
\[ (\forall \eta \in u) [ ( \eta \rstr k, \nu) \in \mcr_k ] \]
then there are $\geq k+1$ elements $\rho \in \suc_{\mct_{2}}(\nu)$
such that  
$(\forall \eta \in u)[ (\eta, \rho) \in \mcr_{k+1}]$ and 
$(\forall \eta \in v)[ (\eta, \rho) \notin \mcr_{k+1}]$ -- and similarly for $\mct_1, \mct_2$ reversed.\footnote{One drawback is that this isn't satisfied by the theories of \cite{MiSh:1140}.} 
\end{disc}

Returning to the main line of the construction, an important feature of this setup is its potential for symmetry, which will help in our proofs.  

\begin{defn}
For any \ parameter $\xm_1$, the dual $\xm_2 = \dual(\xm_1)$ is defined by:
\begin{enumerate}
\item $(\mct^{\xm_2}_2, \mct^{\xm_2}_1) = (\mct^{\xm_1}_1, \mct^{\xm_1}_2)$.
\item $\mcr^{\xm_2}_n = \{ (\eta_2, \eta_1) : (\eta_1, \eta_2) \in \mcr^{\xm_1}_n \}$. 
\end{enumerate}
\end{defn}

\begin{obs}
 If $\xm$ is a basic parameter, so is $\dual(\xm)$, and $\dual(\dual(\xm)) = \xm$, and if $\xm$ is a parameter, then so is $\dual(\xm)$. 
\end{obs}

It is worth noting that this definition extends the ``new simple theory'' from \cite{MiSh:1140}, used there to produce 
an example of incomparability in ZFC.  That said, the present version is substantially more general and more flexible, 
both in its set-up and in its  incorporation of symmetry, as the next sections will show.  [The reader unfamiliar with \cite{MiSh:1140} 
can safely skip Observation \ref{k11}.] 

\begin{obs} \label{k11}
For every $f: \omega \rightarrow \omega \setminus \{ 0, 1, 2\}$ which goes to infinity, 
$T_f$ from \cite{MiSh:1140} is equal, up to renaming, to $T_\xm$ for some basic parameter $\xm$.  If in addition 
$f(k) \geq k+1$, then in addition $\xm$ is a parameter.  
\end{obs}

\begin{proof}
Using the notation of \cite{MiSh:1140} Definition 2.4, let's check definitions \ref{d:param} and \ref{d:nice}. 

Let $\mct_{2,n} = \prod_{\ell<n} f(\ell)$ and let $\mct_2 = \bigcup_n \mct_{2,n}$. 

In order to define $\mct_1$, recall that in \cite{MiSh:1140} 2.4, there was a natural tree structure on the left-hand side given as follows. 
We called $s \subseteq \mct_{1,\leq k}$ ``$k$-maximal'' if (a) it is a subtree, thus downward closed (closed under initial segment), and 
(b) it does not contain all immediate successors of 
any given node. The point is that an element of the left-hand side in a model of $T_f$ determined some such $s$ (by its connections on the right) 
and that $\subseteq$ gives a natural partial ordering on the set of all $s$ that are $k$-maximal for some finite $k$, forming an infinite, 
finitely branching tree.  So, we choose $\mct_1 \subseteq {^{\omega>}\omega}$ to be equivalent to this tree (up to renaming) and choose $\mcr$ so that 
$\mcr_k$ holds between $\eta \in \mct_{1,k}$ and $\rho \in \mct_{2,k}$ if and only if $\eta$ was (before renaming) the subtree $s$ and 
$s$ contains $\rho$. 
Let $\xi$ be the sequence constantly equal to $1$.  This completes the specification of $\xm$, so 
let us check \ref{d:param}.  Clearly (1), (2), (3), (4), (6) hold. 
(5) is trivially satisfied as $\xi$ is constantly 1. 
Likewise, it is straightforward to check that as long as $f(k) \geq k+1$, the fullness and extension conditions of \ref{d:nice} follow easily from the use of 
$k$-maximal $s$'s. Thus, $\xm$ is a parameter.  
\end{proof}

\br

Next we use our template $\xm$ to produce a universal theory, and its model completion. 
Note that this theory is in a different signature, and 
a priori has no access to the trees and edges mentioned in $\xm$. 

First we informally describe this universal theory. The signature consists of unary predicates, explained next, plus a binary relation $R$. 
The unary predicates are indexed by the nodes in the trees 
($Q_\eta$ for the left, $P_\nu$ for the right, along with ``$\mcq$'' for $Q_{\langle \rangle}$ and ``$\mcp$'' for $P_{\langle \rangle}$) 
and these predicates in some sense ``hard-code'' the structure of the trees: 
$\mcq$ and $\mcp$ partition the domain, the $Q_\eta$'s are all subsets of $\mcq$, the $P_\nu$'s are all subsets of $\mcp$, 
if $\eta \tlf \eta^\prime$ then $Q_\eta \supseteq Q_{\eta^\prime}$, if $\nu_1, \nu_2$ are immediate successors of 
$\eta$ then $Q_{\nu_1} \cap Q_{\nu_2} = \emptyset$ (indeed, the predicates corresponding to the immediate successors of $\eta$ partition $Q_\eta$), and the parallel for the $P$'s.  This can all be said with universal 
axioms. Finally, we need to address $R$, which is a binary relation which may hold between elements of $\mcq$ and elements 
of $\mcp$. We add universal axioms saying essentially that if $\eta$ and $\nu$ were nodes of the same level in the left and right trees 
respectively which were \emph{not} connected in the template, 
then $R$ cannot hold between any element of $Q_\eta$ and $P_\nu$.  (Now it should be clearer why we remarked that 
the ``lazy level adds no new constraints.'')

It will be convenient to define and work with the finite 
approximations $T^0_{\xm,k}$ where we only have unary predicates for nodes up to level $k$ of the left and right trees. 
So $T^0_\xm$, the theory we've just sketched, will be a universal theory in an infinite language, defined as the union of 
$T^0_{\xm, k}$ for all finite $k$. 

\S \ref{s:woods} contains a more precise discussion of such theories and the model completions $T_\xm$, whose 
existence we will of course have to justify in the rest of this section.
Very informally for now [this remark is made more precise in $\S \ref{s:woods}$], the model completion may be thought of as a ``bipartite random graph filtered through trees,'' in the sense that if we are in a sufficiently 
saturated\footnote{Some saturation is assumed just so that both type-definable sets are infinite.} model of $T_\xm$, whenever we take 
a type-definable set on the left corresponding to a leaf in the left template tree, and likewise a type-definable set on the right 
corresponding to a leaf in the right template tree, then if these two leaves were 
``connected all the way up'' in the template, the restriction of $R$ to these two sets in our model 
will look like an infinite bipartite random graph, in the sense of model theory. 

\begin{rmk} \label{rmk:212} When the context is clear, below,  
we will write  $\xm = (\mct_1, \mct_2, \mcr)$ instead of $(\mct_{\xm, 1}, \mct_{\xm,2}, \mcr_{\xm})$. 
\end{rmk}

In the next definition, we informally think of $\mcq$ as being on the left and $\mcp$ as being on the right. 

\begin{defn} \label{d:universal}
Given a parameter $\xm$ and $k<\omega$, define $T^0_{\xm,k}$, a universal first order theory, as follows. 
Let $\tau_k = \tau_{\xm,k}$ denote\footnote{We could have used predicates $P_1, P_2$, $P_{1,\eta}, P_{2,\rho}$ to emphasize the symmetry and to  continue the notation of $\mct_1, \mct_2$, 
but chose $\cP, \cQ$ for readability.}
\[ \{ \mcq, \mcp, Q_\eta, P_\rho : \eta \in \mct_{1, \leq k}, \rho \in \mct_{2, \leq k} \} \cup \{ R \}. \]
Then $T^0_{\xm,k}$ is the universal theory in $\ml(\tau_{\xm,k})$ such that a $\tau_k$-model $M$ is a model of $T^0_{\xm,k}$ if and only if: 
\begin{enumerate}
\item  $\cQ^M$, $\cP^M$ is a partition of $\dom(M)$. We identify $\mcq$ and $Q_{\langle \rangle}$,  $\mcp$ and $P_{\langle \rangle}$. 

\item  $\langle Q^M_{\eta} : \eta \in \mct_{1, n} \rangle$ is a partition of $\cQ^M$ for each $n\leq k$ and this partition satisfies 
\[ \eta \tlf \nu \in \mct_{1, \leq k} \mbox{ implies } Q^M_{\eta} \supseteq Q^M_\nu. \] 

\item $\langle P^M_{\rho} : \rho \in \mct_{2, n} \rangle$ is a partition of $\cP^M$ for each $n\leq k$ and this partition satisfies 
\[ \rho \tlf \nu \in \mct_{2,\leq k} \mbox{ implies }P^M_{\rho} \supseteq P^M_\nu. \] 

\item $R^M \subseteq \{ (b, a) :  b \in Q^M, a \in P^M $ and for every $n \leq k$, 
there are $\eta_1 \in \mct_{1,n}$, $\eta_2 \in \mct_{2,n}$  such that $(\eta_1, \eta_2) \in \mcr_n$   
and $(b,a) \in Q^M_{\eta_1} \times P^M_{\eta_2} \}$. 

\end{enumerate}
\end{defn}

Informally, condition \ref{d:universal}(4) says there can only be $R$-edges in $M$ between elements which belong to ``leaves'' all of whose 
initial segments of the same height were connected in the template $\mcr$. 

\begin{obs}
$T^0_{\xm,k} \subseteq T^0_{\xm,k+1}$.
\end{obs}

\begin{defn} \label{d:univ} Given a parameter $\xm$ we define $T^0_\xm$, a universal first order theory, as follows. 
The vocabulary  is $\tau = \tau_\xm = \{ \mcq, \mcp, Q_\eta, P_\nu, R : \eta \in \mct_1, \nu \in \mct_2 \}$ where $\mcq, \mcp, Q_\eta, P_\nu$ are unary predicates and 
$R$ is a binary predicate, and 
\[ T^0_\xm = \bigcup \{ T^0_{\xm, k} : k < \omega \}. \]
\end{defn}

\begin{claim}
For each $k<\omega$, the model completion $T_{\xm,k}$ of $T^0_{\xm,k}$ exists. 
\end{claim}

\begin{proof}
$T^0_{\xm,k}$ is a universal theory in a finite relational language, and the class of its models has JEP and AP. Suppose we are given any 
two $M_1, M_2 \models T^0_{\xm,k}$. For JEP, we also assume $M_1 \cap M_2 = \emptyset$; for AP,  we also assume there is 
a model $M_0 \models T^0_{\xm,k}$ such that $M_0 \subseteq M_\ell$ for $\ell = 1,2$, $M_1 \cap M_2 = M_0$. 
Then consider the following model $N$. The domain of $N$ is 
$M_1 \cup M_2$, for each unary predicate $X \in \tau_k$, let 
$X^N = X^{M_1} \cup X^{M_2}$, and let $R^N = R^{M_1} \cup R^{M_2}$.   Thus $T_{\xm,k}$ exists. 
\end{proof}

\begin{claim} \label{k17}
For every $k_*< \omega$ the following holds: if $M \models T_{\xm, k_1}$, and $N \models T_{\xm,k_2}$, where 
$k_1, k_2 \geq k_*$,  
and $\psi$ is a sentence of $\tau_{k_*}$ of length $\leq k_*$ $($or just such that any subformula has $\leq k_*$ free variables$)$, then 
$M \models \psi \iff N \models \psi$. 
\end{claim}

\setcounter{equation}{0}

\begin{proof}
Let $\mcf$ be defined by: $f \in  \mcf = \mcf_{k_*}$ iff for some $k \leq k_*$ and $a_0, \dots, a_{k-1} \in M$, $b_0, \dots, b_{k-1} \in N$, we have that 
$f = \{ (a_\ell, b_\ell) : \ell< k \}$, and for every atomic $\vp(x_0, \dots, x_{k-1}) \in \ml(\tau_{k_*})$,  
we have that 
\[ M \models \vp[a_0, \dots, a_{k-1}] \iff N \models \vp[b_0, \dots, b_{k-1}]. \]
(We could just as well replace ``atomic'' by ``quantifier free.'') 
Thus, $\mcf$ is a set of partial one to one functions $f$ from $M$ into $N$ such that $|\dom(f)| \leq k_*$, and clearly 
if $f \in \mcf$ and $A \subseteq \dom(f)$ then $f \rstr A \in \mcf$. 

We claim that if $f \in \mcf$,
$|\dom(f)| < k_*$ 
 and $a \in M, b\in N$ then there are $a^\prime \in M$, 
$b^\prime \in N$ such that $f \cup \{ (a^\prime, b) \} \in \mcf$ and $f \cup \{ (a, b^\prime) \} \in \mcf$.  
Suppose we are given $f = \{ (a_\ell, b_\ell) : \ell < k < k_* \}$ along with $a, b$.  Since $k_1, k_2 \geq k_*$ in the definition of 
$M$, $N$ are arbitrary, 
it will suffice to find $b^\prime$.  Moreover, since either $a \in \mcp^M$ or $a \in \mcq^M$, by symmetry (i.e. we can use $\dual(\xm)$)
it suffices to consider the case $a \in Q^M$.\footnote{Informally, 
here is the worry: $M$ is a fortiori a model of $T_{\xm,k_*}$, so the best quantifier-free $\tau_{k_*}$-information we have about the $a_\ell$'s in $M$ 
is to know which leaf at level $k_*$ each of them belongs to (i.e. which $Q_\eta$ or $P_\rho$ for $\eta \in \mct_{1,k_*}$ or $\rho \in \mct_{2,k_*}$) and 
whether or not they connect 
via $R$. In the model $N$, the corresponding $b_\ell$'s have the same quantifier-free 
$\tau_{k_*}$-type as their counterparts in $M$, but when looking for $b^\prime$ in $N$ we must consider 
an additional level of resolution, namely the leaf of each $b_\ell$ at   
at level $k_2 \geq k_*$. 
For example, if $(Q^M_{\eta}, P^M_\rho, R^M)$, $\lgn(\eta) = \lgn(\rho) = k_*$ 
form an infinite  bipartite random graph in $M$, then for any finite set $u$ of elements of $Q^M_\eta$ there is $a \in P^M_\rho$ $R$-connecting  
to all of them. But suppose $f$ had mapped the elements of $u$ to elements of $Q^N_\eta$ which happened 
to span $Q^N_{\eta ^\smallfrown \langle i \rangle}$ for $i < |\suc_{\mct_1}(\eta)|$. 
Then we could not find a corresponding $b^\prime$ in $N$. We solve this by limiting the size of sets $u$ in terms of $k_*$ and 
using the extension axioms.}

Consider the sequence $\{ a_\ell : \ell < k \}$ in $M$. Each $a_\ell$ is either in $\cQ^M$ or $\cP^M$. Renumbering, without loss of generality, 
there is $\ell_* \leq k$ such that $a_\ell \in \cP^M$ for $\ell < \ell_*$ and $a_\ell \in \cQ^M$ otherwise (so the corresponding fact is true for the 
$b_\ell$'s in $N$). Also, without loss of generality, the sequence $\langle a_\ell : \ell < \ell_* \rangle$ is without repetition, and 
$a \notin \{ a_\ell : \ell < \ell_* \}$, otherwise it is trivial. 
Since we have assumed that $a$, our new element, is in $\cQ^M$, let $\eta \in \mct_{1,k_*}$ be such that $a \in Q^M_\eta$. 
Looking now at $N$, 
recalling that $k_2 \geq k_*$, 
let $\rho_0, \dots, \rho_{\ell_*-1} \in \mct_{2,k_2}$ be such that $b_\ell \in P^N_{\rho_\ell}$ for $\ell < \ell_*$. 
(It follows by our definition of $f \in \mcf$ that $a_\ell \in P^N_{\rho_\ell \rstr k_*}$ for $\ell < \ell_*$.) It will suffice to find 
$b^\prime \in Q^N_\eta$ such that 
\begin{equation}
\label{e34} (a, a_\ell) \in R^M ~\iff ~ (b^\prime, b_\ell) \in R^N ~~ \mbox{ for $\ell < \ell_*$.} 
\end{equation}
The inequalities are easy so we ignore them. 
Note that the axioms for $T^0_{\xm, k_*} \subseteq T_{\xm, k_*} \subseteq T_{\xm, k_1} \cap T_{\xm, k_2}$ in \ref{d:universal}(4) imply that 
\begin{equation} \label{e34a} (a, a_\ell) \in R^M ~ \implies ~ (\eta, \rho_\ell \rstr k_*) \in \mcr_{k_*}. 
\end{equation}
Thus, for equation (\ref{e34}), it will suffice to show that there is some $\eta^\prime \in \mct_{1,k_2}$ such that 
$\eta \tlf \eta^\prime$ and 
\[ (a, a_\ell) \in R^M ~ \implies ~ (\eta^\prime, \rho_\ell) \in \mcr_{k_2}. \]
(It doesn't matter to us here whether the non-edges come from the randomness between 
leaves or from leaves with no edges between them.)  Let us define by induction on $t \leq (k_2 - k_*)$ a $\tlf$-increasing 
sequence of elements $\eta_t \in \mct_{1,k_* + t}$ such that $\eta_0 = \eta$, $s \leq t \implies \eta_s \tlf \eta_t$, and 
$(a, a_\ell) \in R^M ~ \implies ~ (\eta_t, \rho_\ell \rstr_{k_*+t}) \in \mcr_{k_*+t}$.  For $t=0$, 
this follows from equation (\ref{e34a}). For $t\geq 0$, 
since $\ell_* < k_*$, we may apply the right extension axiom \ref{d:nice} (using 
$k_* + t$, $\eta_t$, $\{ \rho_\ell \rstr_{k_*+t} : \ell < \ell_* \}$ here for $k, \eta, \{ \rho \rstr k : \rho \in u \}$ there) and 
choose any one of the $\eta$'s returned by that axiom to be $\eta_{t+1}$.   
Let $\eta^\prime = \eta_{k_2 - k_*} \in \mct_{1, k_2}$, and 
this completes the proof.  
\end{proof}

\begin{cor} \label{c:mc}  When $\xm$ is a parameter, the sequence $\langle T_{\xm,k} : k <\omega \rangle$ converges. Moreover, for every 
formula $\vp(\bar{x})$ of $\tau_\xm$, for some quantifier free $\psi(\bar{x})$, for every $k<\omega$ large enough, we have 
\[ (\forall \bar{x}) (~ \psi(\bar{x}) \equiv \vp(\bar{x}) ~) \in T_{\xm,k}. \]
\end{cor}

\begin{concl} \label{c:tmc-1}
Let $\xm$ be a parameter and $T^0_\xm$ be the universal theory from $\ref{d:univ}$. 
Then its model completion $T_\xm$ is well defined, eliminates quantifiers, and is equal to the limit of $\langle T_{\xm,k} : k <\omega \rangle$.  
\end{concl}

\begin{lemma}
Continuing in the context of Conclusion $\ref{c:tmc-1}$, the theory $T_\xm$ is simple and the only dividing comes from equality. 
\end{lemma}

\begin{proof} 
Work in the monster model for $T_\xm$.  Observe that this theory has trivial algebraicity (and quantifier elimination).  
Let $\vp(\bar{x}; \bar{a})$ be a formula realized by some $\bar{b}$ with $\bar{b} \cap \bar{a} = \emptyset$. 
Without loss of generality, $\vp(\bar{x}, \bar{y})$ decides  
instances of $R$ and equality between its variables, 
and implies that no two of its variables are equal.  
Suppose for a contradiction that $\vp(\bar{x}; \bar{a})$ divides, witnessed by the set of formulas
$\{ \vp(\bar{x}; \bar{a}^i) : i < (2^{\aleph_0})^+ \}$ being $k$-inconsistent, where $\langle \bar{a}^i : i < (2^{\aleph_0})^+ \rangle$ 
is a (nontrivial) indiscernible sequence in the type of $\bar{a}$. 
Suppose $\lgn(\bar{x}) = m$ and 
$\lgn(\bar{a}) = \lgn(\bar{y}) = n$. 
[In fact, by basic properties of nonforking, it would suffice to consider $m=1$ and $T_{\xm,k}$ for each $k<k_*$, where 
the picture is very much like the random graph;
this simplifies the proof. However, we give the general picture.]

Fix a sequence $\langle \bar{b}^i : i < (2^{\aleph_0})^+ \rangle$ of $m$-tuples such that 
$\models \vp[\bar{b}^i, \bar{a}^i]$ and $\bar{b}^i \cap \bar{a}^i = \emptyset$ for each $i$.  
Write $\bar{b}^i = \langle b^i_0, \dots, b^i_{m-1} \rangle$ and 
$\bar{a}^i = \langle a^i_0, \dots, a^i_{n-1} \rangle$.  
Since $(2^{\aleph_0})^+$ is regular, we may assume the type of $b^i_j$ over the empty set does not depend on $i$, and also 
(by definition of indiscernible) that the type of 
$a^i_t$ over the empty set does not depend on $i$. 
[The ``leaf'' to which the $\ell$-th element of the tuple $\bar{b}^i ~^\smallfrown \bar{a}^i$ belongs 
is constant as we vary $i$.]  
That is, for $j < m$:
\begin{itemize}
\item if $\models \mcq(b^i_j)$, then there is $\eta_* = \eta_*(j) \in \lim(\mct_1)$ such that $\models Q_{\eta_*(j) \rstr \ell}$ for all 
$\ell < \omega$, where $\eta_*(j)$ depends on $j$ but not on $i$. 

\item if $\models \mcp(b^i_j)$, then there is $\nu_*(j) \in \lim(\mct_2)$ such that $\models P_{\nu_*(j) \rstr \ell}$ for all 
$\ell < \omega$,  where $\nu_*(j)$ depends on $j$ but not on $i$. 
\end{itemize}
Likewise, for $t < n$, 
\begin{itemize}
\item if $\models \mcq(a^i_t)$, then there is $\eta_* = \eta_*(t) \in \lim(\mct_1)$ such that $\models Q_{\eta_*(t) \rstr \ell}$ for all 
$\ell < \omega$, where $\eta_*(t)$ depends on $t$ but not on $i$. 

\item if $\models \mcp(a^i_t)$, then there is $\nu_*(t) \in \lim(\mct_2)$ such that $\models P_{\nu_*(t) \rstr \ell}$ for all 
$\ell < \omega$, where $\nu_*(t)$ depends on $t$ but not on $i$. 
\end{itemize}
Recall that by quantifier elimination, the only information $\vp$ can specify about the relation of any given 
$x_j$ to the other $x$ or $y$ variables involves the unary predicates, the relation $R$, and equality. 

Now we shall choose by induction on $j < m$ elements $b^*_j$ such that the sequence $\langle b^*_j : j < m \rangle$ realizes 
$\{ \vp(\bar{x}; \bar{a}^i) : i < (2^{\aleph_0})^+ \}$, and this contradiction will finish the proof.  
At stage $j$, we'll want to keep track of which elements of $\{ a^i_t : i < (2^{\aleph_0})^+, t < n \} \cup 
\{ b^*_s : s < j \}$ are in $\mcq$ and which are in $\mcp$ (keeping in mind that the partition of the first set depends only on $t$). Let 
\[  A^j_\mcq = \{ a^i_t : i < (2^{\aleph_0})^+, t < n, ~ \models \mcq(a^i_t) \} \cup \{ b^*_s : s < j, ~ \models \mcq(b^*_s) \} \] 
and likewise let
\[  A^j_\mcp = \{ a^i_t : i < (2^{\aleph_0})^+, t < n, ~ \models \mcp(a^i_t) \} \cup \{ b^*_s : s < j, ~ \models \mcp(b^*_s) \}. \] 

Thus, for each $j< m$, there are two cases. If $\models \mcq(b^i_j)$ (for some, equivalently every, $i$) 
then it suffices to choose $b^*_j$ such that:

\begin{itemize}
\item $\{ b^*_j \} \cap A^j_\mcq = \emptyset$, i.e. $b^*_j$ is not equal to any other element in $\mcq$ under consideration; and 
\item $b^*_j$ satisfies the appropriate pattern of $R$-edges and non-edges over the elements of $A^j_\mcp$ as specified by $\vp$, i.e. 
\begin{itemize}
\item if $a^i_t \in A^j_\mcp$ and $\vp \vdash R(x_j, y_t)$,  then $R(b^*_j, a^i_t)$
\item if $a^i_r \in A^j_\mcp$ and $\vp \vdash \neg R(x_j, y_r)$,  then $R(b^*_j, a^i_r)$
\item if $s< j$ and $b^*_s \in A^j_\mcp$ and $\vp \vdash R(x_s, x_j)$, then $R(b^*_s, b^*_j)$
\item if $s< j$ and $b^*_s \in A^j_\mcp$ and $\vp \vdash \neg R(x_s, x_j)$, then $\neg R(b^*_s, b^*_j)$.
\end{itemize} 
\end{itemize}
We know from the universal theory $T^0_\xm$ that the consistency of these demands relies on 
the predicates of the elements and the equalities between them, and nothing else.\footnote{To be clear, recall that if 
two  nodes $\eta, \nu$ of finite height are connected in the template trees, then elements of $Q_\eta, P_\nu$ are free to be 
related or unrelated by $R$, whereas if $\eta, \nu$ are not connected, then all corresponding instances of $R$ are forbidden. 
If $\eta_* \in \lim(\mct_1)$ and $\nu_* \in \lim(\mct_2)$ and 
$Q_{\eta_* \rstr \ell}(a)$ for $\ell < \omega$ and $P_{\nu_* \rstr \ell}(b)$ for $\ell < \omega$ \underline{and} 
$(\eta_* \rstr \ell, \nu_* \rstr \ell) \in \mcr_\ell$ for $\ell < \omega$, then there are a priori no constraints on whether or not $R$ holds 
between $a$ and $b$.}
Of course, $\vp$ need not decide all predicates a priori, but on our subsequence, we ensured this information 
is effectively decided, constant across $i$, and consistent with the pattern of edges implied by $\vp$, 
as witnessed by the consistency of each $\vp(\bar{x}, \bar{a}^i)$; meanwhile, 
our choice of indiscernible sequence and inductive hypothesis ensure no trouble is provided by equality. 
So we can carry the inductive step. 

If $\models \mcp(b^i_j)$ (for some, equivalently every, $i$) 
then it suffices to choose $b^*_j$ satisfying the parallel conditions with $\mcq$ and $\mcp$ reversed, and the justification is the same. 

This completes the proof. 
\end{proof}

\begin{concl} \label{c:tmc}
Let $\xm$ be a parameter and $T^0_\xm$ be the universal theory from $\ref{d:univ}$. 
Then its model completion $T_\xm$ exists, eliminates quantifiers, 
is simple unstable, and the only dividing comes from equality. 
\end{concl}

We will continue with a description of the models and types of $T_\xm$ in \S \ref{s:types} after some discussion.

\vspace{2mm}

\section{A dark woods} \label{s:woods}

In this primarily expository section we make some motivating and organizing remarks about the new theories of 
\S \ref{s:nt}.  

Recall that ``$R$ acts as an (infinite, model theoretic) bipartite random graph between the sets $A$, $B$'' is shorthand for: 
for any two disjoint finite subsets $A_0, A_1$ of $A$, there is an element $b \in B$ which has an $R$-edge to 
all elements of $A_0$ and to no elements of $A_1$, and for any two disjoint finite subsets $B_0, B_1$ of $B$, 
there is an element $a \in A$ which has an $R$-edge to all elements of $B_0$ and to no elements of $B_1$. 

Let's begin with some very simple examples, which are too basic to satisfy the definitions of \S \ref{s:nt} outright (the signatures are finite) but illustrative nonetheless. 

\begin{expl}  \label{expl1}
Suppose $\tau$ includes unary predicates $\mcq$ and $\mcp$, a binary relation $R$, and nothing else. Suppose $T_0$ 
asserts that $\mcq$ and $\mcp$ partition the domain, and that $R \subseteq \mcq \times \mcp$. Then the model completion $T$ of 
$T_0$ exists. In models $M$ of $T$, both $\mcp^M$ and $\mcq^M$ are infinite, and partition the domain. 
$R$ acts as a model theoretic bipartite random graph between $\mcq^M$ and $\mcp^M$.  $($There are no instances of $R$ 
within $\mcp^M$ or $\mcq^M$.$)$
\end{expl}

\begin{expl} \label{expl2}
Suppose $\tau$ includes unary predicates $\mcq, Q_0, Q_1, Q_2, \mcp, P_0, P_1, P_2$, a binary relation $R$, and nothing else. 
Suppose $T_0$ asserts that $\mcq$ and $\mcp$ partition the domain; that $Q_0, Q_1, Q_2$ partition $\mcq$; 
that $P_0, P_1, P_2$ partition $\mcp$; and that $R \subseteq \mcq \times \mcp$. Suppose that in addition $T_0$ asserts that:
\begin{itemize}
\item there are no $R$-edges between $Q_0$ and $P_1$, and 
\item there are no $R$-edges between $Q_1$ and $P_2$.
\end{itemize}
Then the model completion $T$ of $T_0$ exists.  In a model $M$ of $T$, we have that $Q^M_0$, $Q^M_1$, $Q^M_2$, $P^M_0$, $P^M_1$, $P^M_2$ 
are all infinite, and partition the domain. There are no instances of $R$ between $Q^M_0$ and $P^M_1$, between 
$Q^M_1$ and $P^M_2$, or within $\mcq^M$ or within $\mcp^M$. Notice that $R$ acts as a model-theoretic bipartite random graph between 
$Q^M_i$ and $P^M_j$ for $(i, j) \in \{ 0, 1, 2\} \times \{ 0, 1, 2 \} \setminus \{ (0, 1), (1, 2) \}$. What's more, $R$ acts as a model-theoretic 
bipartite random graph between $Q^M_0 \cup Q^M_1$ and $P^M_1$, and also between $Q^M_1 \cup Q^M_2$ and $P^M_0 \cup P^M_1$.  Indeed, 
$R$ acts as a model-theoretic bipartite random graph between $\bigcup_{i \in u} Q^M_i$ and 
$\bigcup_{j \in v} P^M_j$ whenever $u \subseteq \{ 0, 1, 2 \}$, $v \subseteq \{ 0, 1, 2 \}$ and none of the pairs $Q_i, P_j$ for 
$(i,j) \in u \times v$ have their $R$-edges forbidden by $T_0$. 
\end{expl}

Observe that a simple way to encode the information at the end of Example \ref{expl2} could be to consider a kind of 
bipartite ``reduced graph'' between the indices for ``leaves,'' here $\{ 0, 1, 2 \}$ and $\{ 0, 1, 2 \}$, where we put a symbolic edge between 
$i$ and $j$ if and only if $T_0$ does not forbid $R$-edges between $Q_i$ and $P_j$.  
Then we can summarize Example \ref{expl2} (slightly abusing notation by not referencing a model) by saying  
$R$ acts as a model-theoretic bipartite random graph between $\bigcup_{i \in u} Q_i$ and 
$\bigcup_{j \in v} P_j$ precisely when the restriction of our bipartite reduced graph to the vertices 
in $u$ and $v$ is complete.\footnote{Of course, the restriction to $\{ i \}$ and $\{j \}$ when there is a symbolic edge 
between $i$ and $j$ is just a special case of a complete bipartite graph.}

We now point out that in the more general context of parameters and the larger signatures of $\S \ref{s:nt}$, 
such ``reduced graphs'' on the ``leaves'' likewise 
give a nice picture of our theories $T_\xm$, and can be phrased naturally in terms of the 
template edges $\mcr$. 
The next two definitions allow us to compare what we are calling the reduced graph to the $R$-graph in a 
model of the theory, starting with the case where we restrict to nodes of a fixed finite height $k$. 

\begin{defn}  \label{d:hgr} Suppose we are given a parameter $\xm$, a finite $k$, and
 nonempty sets $V \subseteq \mct_{1,k}$ and $W \subseteq \mct_{2,k}$. 
Let 
\[ H_k(V,W) = ( V,  W,  \mcr_k \rstr V \times W ). \]
\end{defn}

\begin{defn} \label{d:ggr} Suppose we are given a parameter $\xm$ thus $T_\xm$, a finite $k$, a model $M \models T_\xm$, 
and nonempty sets $V \subseteq \mct_{1,k}$ and $W \subseteq \mct_{2,k}$. 
Let 
\[ G_k(V,W) = G_k(V,W)[M] = ( \bigcup_{\eta \in V} Q^M_\eta,   \bigcup_{\eta \in W} P^M_\rho,  
 R \rstr ~(\bigcup_{\eta \in V} Q^M_\eta \times \bigcup_{\eta \in W} P^M_\rho)~ ). \]
\end{defn}

\begin{rmk} \label{rmk-sz}
\emph{The phrase ``reduced graph'' may bring to mind Szemer\'edi's regularity lemma for graphs, 
where recall that a given finite graph is partitioned into clusters 
in such a way that between most pairs of clusters the edges are distributed $\epsilon$-uniformly.  There one may define a reduced graph
(see \cite{ks} p. 306), for instance by taking one vertex for each cluster, and with an edge between two points whose associated clusters are $\epsilon$-regular with density $\epsilon$-bounded away from $0$ (and if desired, $1$). Such a reduced graph doesn't only record the ``generic interaction'' of a given pair of clusters, but also entails  
that there is a certain further genericity in the interaction of more than two clusters, e.g. if three points in the 
reduced graph form a triangle, we should be able to get many triangles on 
three vertices spanning the associated clusters in the 
original graph. A certain analogue of this in our setting is in the comment after Example \ref{expl2} about complete 
bipartite graphs. }
\end{rmk}

The next definition gives the full analogue for the countable height trees we really use in $\xm$ and $T_\xm$. 
The word \emph{virtual} reflects that the objects are generally not definable, though they may be 
type-definable. 
Though what we call ``$\mcr^\infty$'' is not definable in \ref{d:vrg}, 
the edge relation in \ref{d:vg} is simply $R^M$. 

\begin{defn}[Virtual reduced graph]  \label{d:vrg}
Let $\xm = (\mct_1, \mct_2, \mcr)$ be a parameter. 
\begin{enumerate}
\item Define $\mcr^\infty = \{ (\rho, \eta) : (\rho, \eta) \in \leaves(\mct_1) \times \leaves(\mct_2)$ and  
$(\rho \rstr k, \eta \rstr k) \in \mcr_k$ for all $k<\omega \}$. 

\item Then for any 
nonempty $V \subseteq \leaves(\mct_1)$ and $W \subseteq \leaves(\mct_2)$, define the \emph{virtual reduced graph}
\[ H^\infty(V,W) \]
to be the bipartite graph $(V, W, \mcr^\infty)$. 
\end{enumerate}
\end{defn}

\noindent That is, \ref{d:vrg} defines a bipartite graph whose vertices are the leaves of $\mct_1$ on the left and the leaves of $\mct_2$ on the right and 
where $(\eta, \nu)$ is an edge if and only if $(\eta \rstr n, \nu \rstr n) \in \mcr_n$ for all $n<\omega$.  (2) gives 
various induced subgraphs. 

\begin{defn}[Virtual graph]  \label{d:vg} 
Continuing in the notation of $\ref{d:vrg}$, suppose we are given any model $M \models T_\xm$. 

\begin{enumerate}

\item For any $V \subseteq \leaves(\mct_1)$, let the expression 
$Q^\infty_V = Q^\infty_V[M] $ denote  
\[ \{  a \in \dom(M) : \mbox{ for some $\eta \in V$, } M \models Q_{\eta \rstr k}(a) \mbox{ for all } k<\omega \}. \]
In particular, for any $\eta \in \leaves(\mct_1)$, $Q^\infty_{\{\eta\}} = Q^\infty_{\{\eta\}}[M]$ denotes the subset of $M$ realizing 
the type $\{ Q_{\eta \rstr k}(x) : k < \omega \}$. 

\sbr

\item 
Likewise 
for any $W \subseteq \leaves(\mct_2)$, let the expression 
$Q^\infty_W = Q^\infty_W[M]$ 
denote 
\[ \{  b \in \dom(M) : \mbox{ for some $\rho \in W$, } M \models P_{\eta \rstr k}(b) \mbox{ for all } k<\omega \}. \]
In particular, for any $\rho \in \leaves(\mct_2)$, $P^\infty_{\{\rho\}} = P^\infty_{\{\rho\}}[M]$ denotes the subset of $M$ realizing 
the type $\{ P_{\rho \rstr k}(x) : k < \omega \}$. 

\sbr

\item For any nonempty 
$V \subseteq \leaves(\mct_1)$, $W \subseteq \leaves(\mct_2)$, let the \emph{virtual graph} 
\[ G^\infty(V,W)  = G^\infty(V,W)[M]  \]
be the bipartite graph 
\[ ( Q^\infty_V, ~ P^\infty_W, ~ R^M \rstr Q^\infty_V \times P^\infty_W ). \]
\end{enumerate}\end{defn}

\noindent That is, \ref{d:vg} defines a bipartite graph from $M$ whose vertices are elements of $Q$ belonging to certain 
``leaves'' on the left and the elements of 
$P$ belonging to certain other ``leaves'' on the right, along with the edge relation given by $R$. 

\noindent
\begin{disc} \label{d:analogue}  We defer to \S\S \ref{s:nt}, \ref{s:types} for details. \emph{Given a parameter $\xm$:}

\begin{enumerate}
\item[\emph{a)}]  \emph{The structure of models of $T_\xm$ is in some sense simple: 
in the language of \ref{d:vrg} and \ref{d:vg}, in a sufficiently saturated\footnote{For simplicity, to ensure all countable intersections of predicates are nonempty.} model $M \models T_\xm$, 
given any 
nonempty $V \subseteq \leaves(\mct_1)$ and $W \subseteq \leaves(\mct_2)$, if $H^\infty(V,W)$ is a complete graph, then 
$G^\infty(V,W)[M]$ is an infinite bipartite random graph, and if $H^\infty(\{ \eta \}, \{ \rho \} )$ is empty, then so is $G^\infty(\{ \eta \}, \{ \rho \})[M]$. 
(Letting $V, W$ vary, these two facts together are enough to put together the whole picture.) 
\item[b)] In any $\aleph_1$-saturated model $M$ of $T_\xm$, e.g. in a regular ultrapower, for any leaves 
$\eta \in \leaves(\mct_1)$, $\rho \in \leaves(\mct_2)$, 
the sets $Q^\infty_{\{\eta\}}$, $P^\infty_{\{\rho\}}$ will be infinite, and will have among them the infinite empty or  
random graph structure just mentioned. We will see in detail in \S \ref{s:types} that 
for $\lambda^+$-saturation, we will want each such $Q^\infty_{\{\eta\}}$ and 
each such $P^\infty_{\{\rho\}}$ to have size at least $\lambda^+$, and moreover, for every 
$V \subseteq \leaves(\mct_1)$ and $W \subseteq \leaves(T_2)$ such that $H^\infty(V,W)$ is a complete graph\footnote{What about other $W$s? 
It can't hurt, but won't add anything: see last line of proof of \ref{proof-wx}.},    
$G^\infty(V,W)$ is $\lambda^+$-saturated as a bipartite random graph, i.e. }

\begin{enumerate}
\item[(i)] \emph{ for any two disjoint subsets $A, B$ of $P^\infty_W$ of size $\lambda$, and any
$\eta \in V$, there is $c \in Q^\infty_{\{\eta\}}$ which $R$-connects to all $a \in A$, no $b \in B$. }

\item[(ii)]\emph{  the parallel reversing $V,  W$ and $Q, P$. }

\end{enumerate}
\emph{ In what follows, we will focus on $\xm$ such that $\xm = \dual(\xm)$; so by symmetry, it will be enough to handle one of (i), (ii), and as we will see in \ref{simple-form} and \ref{concl123} below, it will generally be enough to realize partial positive $R$-types. }
\end{enumerate}
\end{disc}

\br

\noindent Where does the potential for widely differing complexity arise?  The following informal discussion may help the reader 
anticipate or follow the proof. 

\br

\emph{Why might these theories interact with ultrapowers in an interesting way?} 
In an ultrapower of a model of $T_\xm$, elements which are ``on average'' part of the same leaf 
may nonetheless appear, when projected to a given index model, 
to be in too many different predicates at a given height $k$, blocking realization of the type in that index model when splitting is 
constrained.  
Both the {size} of allowed splitting at a given height in a  
given tree (and, by extension, the level functions) come into play, which in turn  
 reflect the degrees of the vertices in the reduced graphs 
$H_k$. 

\emph{Why might different parameters $\xm$, $\xn$ produce theories $T_\xm$, $T_\xn$ which look different to ultrapowers?} 
The structure of each theory $T_\xm$ will reflect its sequence of ``reduced graphs,'' based on the 
finite bipartite graphs $\mcr_i = \mcr_i(\xm)$, 
and the related level function $\xi = \xi(\xm)$, which is active at infinitely many $n \in \omega$. 
When $\xi$ is not active, $\mcr_{i+1}$ gives essentially no new information beyond $\mcr_i$.  
A natural way to vary the $\mcr_i$'s will be to consider a single fast-growing sequence of 
sparse graphs $\langle E_i : i < \omega \rangle$, choose many level functions which are independent in a 
natural sense, and build for each such $\xi$ a theory whose $\mcr_i$ essentially copies $E_i$ at active levels 
and copies a complete bipartite graph of the right size at lazy levels.   This allows us to vary the   
sequences of reduced graphs in a very clear way. Remarkably, these differences are detected in a 
very strong sense both by the theories themselves and by ultrafilters. To prove this will also require an advance in 
ultrafilter construction.

\begin{rmk} 
To make these suggestions precise will, of course, require the rest of the paper; but notice that the construction already suggests many 
further modifications and interesting future directions, some discussed at the end of the paper. 
\end{rmk}

\vspace{5mm} 

\section{Models of $T_\xm$} \label{s:types}

In this section we analyze the types over a model $M$ of $T_\xm$, which will help later in dealing with saturation. 

Note that we use almost nothing about the level functions in this section; we just need the extension axioms to ensure a minimum 
increase in the edges. The level functions operate at a different scale in the sense that they control variations in the number of 
edges well beyond the minimum established by the extension axioms, and will mainly play a role in later sections, where we try to 
compare theories. 

\begin{conv}
In this section, $\xm$ is an arbitrary but fixed parameter, and $M$ is a  
model of $T_\xm$. 
\end{conv}

For the purposes of our analysis, because of the symmetry of $\xm$, it will suffice to deal with types $q(x)$ in one free 
variable $x$ which describe an element on the left, i.e. $q(x) \vdash \cQ(x)$.  
Note that any such type, being complete, will specify that $Q_{\rho \rstr n}(x)$ for some $\rho \in \leaves(\mct_1)$ and all $n<\omega$. 

Here is some notation for the connections made along the way by a leaf $\rho$. 

\begin{defn}
For $\rho \in \leaves(\mct_1)$, we define: 
\begin{enumerate} 
\item  $\mcs_\rho = \{ \nu: $ for some finite $n$,~ $\nu \in \mct_{2,n}$ and $(\rho \rstr n, \nu) \in \mcr_n \}$.
\item  $\leaves(\mcs_\rho) = \{ \eta \in \leaves(\mct_2) : \eta \rstr n \in \mcs_\rho $ for $n<\omega \}$ 
\newline $= \{ \eta \in \leaves(\mct_2) : 
(\rho, \eta) \in \mcr^\infty \}$.  
\end{enumerate} 
\end{defn}

\begin{obs} \label{o:max} 
Recalling $\ref{d:fullness}$, 
if $\rho \in \leaves(\mct_1)$, $\mcs_\rho$ is a subtree of $\mct_2$ with no maximal node. 
\end{obs}

Recall some notation from the previous section. For $\eta \in \leaves(\mct_2)$, $P^\infty_{\{\eta\}} = P^\infty_{\{\eta\}}[M]$ denotes the elements of $M$ which are ``in the leaf'' corresponding to $\eta$, and the corresponding notation for $\rho \in \leaves(\mct_1)$ is $Q^\infty_{\{\rho\}} = Q^\infty_{\{\rho\}}[M]$. We had likewise defined $P^\infty_V$, $Q^\infty_W$ in \ref{d:vg}, which also depend on $M$. 
Recall the virtual reduced graph $H^\infty(V, W)$ from \ref{d:vrg}, and the virtual graph $G^\infty(V,W) = G^\infty(V,W)[M]$ from \ref{d:vg}.

\begin{obs}
For any $\rho \in \leaves(\mct_1)$ and $W \subseteq \leaves(\mct_2)$ such that 
$H^\infty( \{ \rho \}, W )$ is complete, we have that $\leaves(\mcs_\rho) \supseteq W$, in other words, 
$\mcs_\rho$ contains all proper initial segments of elements of $W$. 
\end{obs}

\begin{claim} \label{claim5} 
Suppose $\rho \in \leaves(\mct_1)$ and write $W = \leaves(\mcs_\rho)$. 
For any 
\[ A, B \subseteq P^\infty_W[M] \mbox{ with } A \cap B = \emptyset  \]
the following set of formulas is a non-algebraic partial type of $M$:  
\[ \{ Q_{\rho \rstr n}(x) : n <\omega \} \cup \{ R(x,a) : a \in A \} \cup \{ \neg R(x,b) : b \in B \}. \] 
\end{claim}

\begin{proof} 
Consider a finite subset, which without  loss of generality is of the form:
\[ \{ Q_{\rho \rstr n}(x) : n < k \} \cup \{ R(x,a_0), \dots, R(x,a_{\ell-1}) \} \cup \{ \neg R(x,b_0), \dots, \neg R(x,b_{r-1}) \}. \] 
Each of the elements $a_i, b_j$ has a leaf in $M$: let's suppose that for each $i<\ell$, $\eta_i$ is such that $M \models P_{\eta_i \rstr n}(a_i)$ for $n<\omega$ and for each $j<r$, $\nu_j$ is such that $M \models P_{\nu_j \rstr n}(b_j)$ for $n<\omega$, 
though these leaves need not necessarily be distinct. 
By our choice of $A$, $B$ [that is, by the definition of $W$], we have that for any finite level, and in particular for $k$, 
\[ (\rho \rstr k, \eta_i \rstr k) \in \mcr_k~~ \mbox{ and } (\rho \rstr k, \nu_j \rstr k) \in \mcr_k \]
for $i<\ell$, $j<r$. Thus the following sentence belongs to $T_{\xm, k}$:
\[  (\exists x)( \bigwedge_{i<\ell} R(x,a_i) \land \bigwedge_{j<r} R(x,b_j) ).            \]
By \ref{k17}, this remains true all the way to $T_\xm$. 
Since this shows an arbitrary finite subset is consistent, we finish the proof. 
\end{proof}

\begin{cor} \label{proof-wx}
Suppose $\rho \in \leaves(\mct_1)$ and write $W = \leaves(\mcs_\rho)$. 
For any  $A, B \subseteq M$ with $A \cap B = \emptyset $ 
we have that
\[ r(x) = \{ Q_{\rho \rstr n}(x) : n <\omega \} \cup \{ R(x,a) : a \in A \} \cup \{ \neg R(x,b) : b \in B \} \]
is a non-algebraic partial type of $M$ if and only if 
$A \subseteq P^\infty_W[M]$. 
\end{cor}

\begin{proof}
Suppose we denote $A_0= A \cap P^\infty_W$ and $B_0 = B \cap P^\infty_W$. Let
\[ r_0 = \{ Q_{\rho \rstr n}(x) : n <\omega \} \cup \{ R(x,a) : a \in A_0 \ \} \cup \{ \neg R(x,b) : b \in B_0 \}. \]
Claim \ref{claim5} tells us that $r_0$ is a partial type. 

First consider any element $b \in B \setminus B_0$. If $b \in Q^M$, then $\neg R(x,b)$ follows by definition as $R^M \subseteq Q^M \times P^M$.  
If $b \in P^M$, then because $M$ is a model, there is some $\eta$ such that 
$b \in P^M_{\eta \rstr n}$ for all $n<\omega$. If $(\rho, \eta) \notin \mcr^\infty$, then there is some $n<\omega$ 
for which $(\eta \rstr n, \rho \rstr n) \notin \mcr_n$, which translates to 
\[ M \models (\forall x)(\forall y) (Q_{\eta \rstr n}(x) \land P_{\rho \rstr n}(y) \implies \neg R(x,y)) \]
and so $r_0 \vdash \neg R(x,b)$. 

Finally, suppose that $A \setminus A_0$ is nonempty, and let $a$ be any one of its elements. 
Let $\eta$ be such that $a \in P^M_{\eta \rstr n}$ for all $n<\omega$.  By definition of $A_0$, 
$(\rho, \eta) \notin \mcr^\infty$, so it follows from the previous paragraph that 
$r_0 \vdash \neg R(x,a)$. Thus $r$ is consistent if and only if $A \setminus A_0 = \emptyset$. 

Note that this proof shows that if $r$ is consistent, $r_0 \vdash r$. 
\end{proof}

The next claim justifies restricting our saturation arguments to considering types of a very 
simple form.

\begin{defn} \label{d:4.7}
We say a model $M$ of $T_\xm$ is \emph{weakly $\lambda^+$-saturated} when: 
\begin{enumerate}
\item ``the leaves are large'': for any $\eta \in \leaves(\mct_1)$, 
$ | \{ a \in M : Q^M_{\eta \rstr n}(a) $ for all $n<\omega \}| > \lambda$, and likewise for $\nu \in \leaves(\mct_2)$. 
\item if $c \in Q^M$ then $\{ b : (c,b) \in R^M \} \subseteq P^M$ has cardinality $ > \lambda$.
\item the dual to the previous line: 
if $b \in P^M$ then $\{ c : (c,b) \in R^M \} \subseteq Q^M$ has cardinality $ > \lambda$. 
\end{enumerate}
\end{defn}

\begin{claim}[A basic form for non-algebraic types] \label{simple-form}
Suppose $M$ 
 is weakly $\lambda^+$-saturated.  
For any $C \subseteq M$, $|C| \leq \lambda$, and $q \in \ts_1(C)$ such that $q(x) \vdash Q(x)$, 
 there exist $\rho$, $W \subseteq \leaves(\mcs_\rho)$, 
$A \subseteq P^\infty_W[M]$, $B \subseteq P^M$ 
with $A \cap B = \emptyset$ and $|A| + |B| \leq \lambda$, such that writing 
\[ r(x) =  {\{ Q_{\rho \rstr n}(x) : n <\omega \} ~\cup~}
\{ R(x,a) : a \in A \} \cup \{ \neg R(x,b) : b \in B \}. \] 
we have 
$r(x) \vdash q(x)$.  We may also ask that $|A|, |B| = \lambda$. 
\end{claim}

\begin{proof}
By hypothesis, every ``leaf'' $Q^\infty_{\{\eta\}}[M]$ or $P^\infty_{\{\rho \}}[M]$ has size $>\lambda$.

To start, let $\rho \in \leaves(\mct_1)$ be such that $q(x) \vdash Q_{\rho \rstr n}(x)$ for all $n<\omega$, which exists as $q$ is a 
complete type. 
Define $W := \leaves(\mcs_\rho)$, 
$A_0 := \{ c \in C :  q \vdash R(x,c) \} \cap P^\infty_W[M]$,  and $B_0 := \{ c \in C : q \vdash \neg R(x,c) \} \cap P^\infty_W[M]$.  
Let 
\[ r_0(x) =  {\{ Q_{\rho \rstr n}(x) : n <\omega \} ~\cup~}
\{ R(x,a) : a \in A_0 \} \cup \{ \neg R(x,b) : b \in B_0 \} \subseteq q(x).   \] 
Clearly $r_0$ is consistent and implies at least the restriction of $q(x)$ to the 
given unary predicates and to all positive and negative instances of $R(x,y)$ on $P^\infty_W[M] \cap \dom(q)$.  
Since $T_\xm = Th(M)$ eliminates quantifiers, it suffices to consider quantifier-free formulas. 
Let us check that by possibly increasing $A_0$, $B_0$ by 
no more than $\lambda$ elements, we can ensure that any formulas of the following kinds 
which belong to our original $q$ are also implied.  Along the way, we remark on the consequences for 
this choice of larger partial type $r$, formally defined in $(\star)$ below.  

\begin{enumerate}

\item[(a)] $x \neq c ~ \in q$

\br 

\noindent If $c \in \cP^M$, then this follows from $\cQ(x)$. 
Let $\langle c_\alpha : \alpha < \kappa \rangle$ enumerate all elements of $\mcq^M$ such that ``$x \neq c_\alpha$'' 
is implied by $q(x)$, or equivalently belongs to $q(x)$.  For each $\alpha < \kappa$, choose some 
$b_\alpha \notin A_0$ such that $M \models R(c_\alpha, b_\alpha)$.   We can do this because each 
$R(c_\alpha, x)$ defines an infinite subset of $M$ which is by hypothesis and Definition \ref{d:4.7} of size at least $\lambda^+$. 
Adding $\{ b_\alpha : \alpha < \kappa \}$ to $B_0$ to form $B_1$ will not raise its size above $\lambda$ and will  
mean that $r(x) \vdash x \neq c_\alpha$ for each $\alpha < \kappa$. Let $A_1 := A_0$. 

\br

\item[(b)] $\neg Q_\nu(x) ~ \in q$, for $\nu \notin \{ \rho \rstr n : n < \omega \}$, or $\neg P_\eta(x) \in q$. 

\br 
\noindent This follows from our assumption that $\{ Q_{\rho \rstr n}(x) : n <\omega \} \subseteq r_0$. 

\br

\br
\item[(c)] $\neg R(x,b)$, for any $b \in B_1 \setminus P^\infty_W[M]$. 

\br
\noindent For any $b \notin P^\infty_W[M]$, there is a finitary reason for the non-membership, i.e. there is $k<\omega$ 
and $\nu \in \mct_{2,k}$ such that $(\rho \rstr k, \nu) \notin \mcr_k$ and $b \in P^M_\nu$.   Then 
$T_\xm \vdash (\forall x)(\forall y)( Q_{\rho \rstr k}(x) \land P_\nu(y) ~ \implies~ \neg R(x,y))$.     
As $r(x) \vdash Q_{\rho \rstr k}(x)$, necessarily $r(x) \vdash \neg R(x,b)$. 
\end{enumerate}
\br
Define $A = A_1$, $B = B_1$, and define  
\[ (\star) ~~~ r(x) =  {\{ Q_{\rho \rstr n}(x) : n <\omega \} ~\cup~}
\{ R(x,a) : a \in A \} \cup \{ \neg R(x,b) : b \in B \}.   \] 
This is a partial type of size $\leq \lambda$ and implies $q(x)$ as desired.  

At this point, if we would like to ensure $|A| = |B| = \lambda$, there is no harm in choosing disjoint sets $A^\prime, B^\prime$ 
of size $\lambda$ 
from $P^\infty_W[M] \setminus (A_1 \cup B_1)$ and defining $A := A_1 \cup A^\prime$, $B := B_1 \cup B^\prime$.  
In this case the partial type $r(x)$ will still imply $q(x)$ but the reverse need not hold.  
\end{proof}

\begin{rmk}
By symmetry of $\xm$, the analogue of $\ref{simple-form}$ is true for types $p(y)$ such that $p(y) \vdash \cP(y)$, and since 
$\cQ,\cP$ partition $M$, this covers all $1$-types, which are sufficient for saturation. 
\end{rmk}

\vspace{5mm}

\section{Ultrapowers of models of $T_\xm$} \label{s:up}

\begin{conv}
In this section, $\xm$ is an arbitrary but fixed parameter. 
\end{conv}

By \ref{simple-form}, to ensure a model of $T_\xm$ with large leaves is $\lambda^+$-saturated, it suffices 
to realize $R$-types. In our main proofs, we focus on saturating regular ultrapowers.  
This section gives some basic additional features of the ultrapower case.

\begin{fact} \label{fact1}
Suppose $I$ is an infinite set and $\de$ is a regular ultrafilter on $I$, $|I| = \lambda$. Then for any infinite model $M$ in a countable 
language, including but not limited to models of $T_\xm$,  the ultrapower $N = M^I/\de$ is $\aleph_1$-saturated. 
\end{fact}

\begin{proof}
See for example Chang and Keisler \cite{ck} 6.1.1. 
\end{proof}

\begin{fact}[see \cite{mm1}]
Saturation of regular ultrapowers reduces to saturation of $\vp$-types, that is, if $M$ is a model in a countable language and 
$\de$ is a regular ultrafilter on $\lambda$, then $M^\lambda/\de$ is $\lambda^+$-saturated if and only if it is $\lambda^+$-saturated 
for $\vp$-types, for every formula $\vp$. 
\end{fact}

Recall from Convention \ref{c:112} that a regular ultrafilter $\de$ on a set $I$, $|I| = \lambda$ is called 
``good for'' a countable theory $T$ if 
for some, equivalently every, $M \models T$, the ultrapower 
$M^I/\de$ is $\lambda^+$-saturated. (The equivalence is by regularity, see \cite{keisler} 2.1a.)

\begin{fact} \label{fact2}
Suppose $I$ is an infinite set and $\de$ is a regular ultrafilter on $I$, $|I| = \lambda$. Then for any infinite model $M$ in a countable 
language, including but not limited to models of $T_\xm$:

\begin{enumerate} 
\item Suppose in addition that $\de$ is good for every countable stable theory. Then 
any infinite definable subset of $N = M^I/\de$, and indeed any infinite internal predicate in $N$, has size at least $\lambda^+$.

\item Suppose in addition $\de$  is good for the theory of the random graph. Then: 
\begin{enumerate}
\item for any two disjoint $A, B \subseteq M^I/\de$, there is an internal predicate 
separating $A$ and $B$.  
\item $\de$ is good for every countable stable $T$. 
\item given any countable sequence $\langle X_n : n < \omega \rangle$ of definable sets and any $A \subseteq N$ 
of size $\leq \lambda$ such that $X_n \supsetneq X_{n+1} \supseteq A$ for all $n<\omega$, there is an internal 
predicate $X_\infty$ such that $X_n \supsetneq X_\infty \supsetneq A$ for all $n<\omega$. \emph{Thus}, if 
$A$ is infinite, 
the intersection $\bigcap_n X_n$ has size $\geq \lambda^+$ in $N$. 
\end{enumerate}
\end{enumerate}
\end{fact}

\begin{proof}
(1) By \cite{Sh:c} Theorem 5.1(1)-(2) p. 379, the minimum product of an unbounded sequence of finite or infinite cardinals 
modulo $\de$ must be at least $\kappa^+$ in order to have that for any model $M$ of any countable stable theory, $M^I/\de$ is $\kappa^+$-saturated. 
(In fact, this condition is both necessary and sufficient.) 

(2)(a) In fact, this characterizes $\de$ being good for the theory of the random graph, see \cite{mm5} \S 3 p. 1585. 

(2)(b) The stable theories are below the unstable theories in Keisler's order. See \cite{Sh:c} Theorem 4.8. p. 379, 
which says that any ultrafilter which is good for some unstable theory must have so-called 
$\lcf(\omega, \de) \geq \lambda^+$, \cite{Sh:c} Definition 3.5 p. 357. 
Thus, $\mu(\de) \geq \lambda^+$ (this is the quantity mentioned above, informally,  the product of any unbounded 
sequence of finite cardinals modulo $\de$;  it is the $\mu$ defined in \cite{Sh:c} Theorem 3.12 p. 357). 
It follows from the last line of the proof of (1) that the ultrafilter is good for any countable stable theory. 

(2)(c) See \cite{mm1} Lemma 9 p. 223. 
This is a consequence of the a priori weaker fact that $\lcf(\omega, \de) \geq \lambda^+$. The last sentence then 
follows from (1) applied to $X_\infty$. 
\end{proof}

\begin{cor} \label{c53}
If $M_0 \models T_\xm$, $\de$ is a regular ultrafilter on $\lambda$ which is good for the random graph, and $M = (M_0)^\lambda/\de$, 
then $M$ is weakly $\lambda^+$-saturated in the sense of $\ref{d:4.7}$.  In particular,  
the hypotheses of $\ref{simple-form}$ hold for $M$. 
\end{cor}

\begin{proof} 
Following the notation of \ref{d:4.7}, for (1), fix $\eta \in \lim(\mct_1)$.  By Fact \ref{fact1}, $M$ is $\aleph_1$-saturated, so 
we can choose $A \subseteq M$ which is countably infinite and which has the property that 
$Q^M_{\eta \rstr n}(a)$ for all $a \in A$ and $n<\omega$. Apply \ref{fact2}(2)(c) using $Q^M_{\eta \rstr n}$ for $X_n$ and this 
$A$, and let $X_\infty$ be as given there. Then $X_\infty$ is contained in the ``leaf'' we are studying, and it is an infinite internal predicate, so by \ref{fact2}(1) it has size at least $\lambda^+$, so the leaf does as well.    The parallel fact for $P$ is proved symmetrically.

For \ref{d:4.7}(2), let $\rho \in \lim(\mct_1)$ be the leaf of $c$. It follows from the axioms, see e.g. Definition 
\ref{d:fullness} or Observation \ref{o:max}, that we can choose some $\eta \in \lim(\mcs_\rho)$ 
(since this set has size continuum, so in particular is nonempty). 
Then $R$ behaves as a bipartite random graph between $Q^\infty_{\{ \rho \}}[M]$ and 
$P^\infty_{\{\eta\}}[M]$, so as $b$ belongs to the first of these sets, it follows 
by $\aleph_1$-saturation (Fact \ref{fact1}) that $\{ b : (c,b) \in R^M \}$ is an infinite definable set. 
Thus by \ref{fact2}(1) this set has size at least $\lambda^+$. 
\end{proof}

\br

Conclusion \ref{concl123} gives a sufficient collection of types to realize in order to saturate regular ultrapowers for self-dual $\xm$ 
(our main case following \ref{is-self-dual} below). 

\begin{concl} \label{concl123}
Suppose $I$ is an infinite set, $|I| = \lambda$, and $\de$ is a regular ultrafilter on $I$ which is good for the theory of the random graph. 
Suppose that $\xm = \dual(\xm)$. 
Let $M_0 \models T_\xm$.  
Then in order to show that $M = (M_0)^I/\de$ is $\lambda^+$-saturated, it is sufficient to show that:
\begin{quotation}
\noindent  $(\star)_{M_0, I, \de}$ ~~~  every partial type of $M$ of the form 
\[ r(x) = { \{ Q_\nu(x) \} } \cup \{ R(x,a) : a \in A \} \]
is realized, where {$\nu \in \mct_{1,n}$ for some $n<\omega$,} $A \subseteq M$ and $|A| \leq \lambda$.
\end{quotation} 
\end{concl}

\begin{proof} 

\underline{Case 1}.
For types including $\cQ(x)$, 
by \ref{fact2}(2) and our present assumption, the conclusion of \ref{fact2}(2)(b) holds. Hence by \ref{c53}, second sentence, the hypothesis of 
\ref{simple-form} holds.  Hence, by Claim \ref{simple-form}, it suffices to deal with partial types of the form 
\[ r(x) = { \{ Q_{\rho \rstr n}(x) : n <\omega \} \cup} \{ R(x,a) : a \in A \} \cup \{ \neg R(x,b) : b \in B \} \]
where $\rho \in \leaves(\mct_1)$ and for some 
 $W \subseteq \leaves(\mcs_\rho)$, we have $A \subseteq P^\infty_W[M]$ with $|A| = \lambda$, and 
$B \in P^M \setminus A$ with $|B| =\lambda$. 

{As saturation of regular ultrapowers reduces to saturation of $\vp$-types, it suffices to consider $\nu \in \mct_{1,n}$ 
for some $n<\omega$, and so to deal with a partial type of the form 
\[ r(x) = { \{ Q_\nu(x) \} \cup} \{ R(x,a) : a \in A \} \cup \{ \neg R(x,b) : b \in B \}. \] 
Note that the assertion that $r(x)$ is a partial type means that for some $\rho^\prime$ with $\nu \tlf \rho \in \leaves(\mct_1)$ and some 
 $W \subseteq \leaves(\mcs_{\rho^\prime})$, we have $A \subseteq P^\infty_W[M]$ with $|A| = \lambda$, and 
$B \in P^M \setminus A$ with $|B| =\lambda$. }

As we are assuming $\de$ is good for the theory of the random graph, by \ref{fact2}(2) we can assume there is an internal predicate 
$X$ separating $A$ and $B$, so let us justify that it sufficies to realize the positive part of the type.  
Enumerate $r$ as $\langle R(x,a_\alpha) : \beta < \lambda, \alpha = 2\beta \rangle$ and $\langle \neg R(x,b_\alpha) : \beta < \lambda, \alpha = 2\beta +1 \rangle$. 
 Let $\{ X_\alpha : \alpha < \lambda \} \subseteq \de$ be a regularizing family. 
Let $f: [\lambda]^{<\aleph_0} \rightarrow \de$ be the ``distribution'' given by sending $\sigma$ to 
\[ \{ t \in I :  M \models (\exists x)( {Q_\nu(x) ~\land~} \bigwedge_{\alpha \in \sigma \mbox{\footnotesize{ even}}} 
R(x,a_\alpha[t]) \land 
\bigwedge_{\alpha \in \sigma \mbox{\footnotesize{ odd} }} \neg R(x,b_\alpha[t])) \} \cap \bigcap_{\alpha \in \sigma} X_\alpha. \]
Then it is straightforward to see that $r$ is realized if and only if $f$ has a multiplicative refinement. 
Let $g$ be the refinement of $f$ given by
\[  \sigma ~~\mapsto ~~ f(\sigma) \cap \{ t \in I : M \models \bigwedge_{\alpha \in \sigma \mbox{\footnotesize{ even}}} X(a_\alpha[t]) \land \bigwedge_{\alpha \in \sigma \mbox{\footnotesize{ odd}}} \neg (X(b_\alpha[t])) \}. \]
Now let us verify: $g$ is a function with domain $[\lambda]^{<\aleph_0}$ (trivial), $\operatorname{range}(g) \subseteq \mcp(\lambda)$ 
(trivial), $\operatorname{range}(g) \subseteq \de$ 
(by the choice of $X$), $g$ is multiplicative (by its definition), and $g$ refines $f$ 
(by choice of $X$ and the properties of the random graph). 

\underline{Case 2.} For types including $P(x)$, we use the 
 fact that $\xm$ is self-dual, so for any type on the left, 
there is a type on the right with an identical distribution. One will have a multiplicative refinement 
(i.e. be realized) if and only if the other does. 

As $M \models (\forall x)( \mcp(x) \lor \mcq(x))$, this finishes the proof that $M$ is $\lambda^+$-saturated. 
\end{proof}

\begin{rmk} \label{rmk-linf}
Assume $\dual(\xm) = \xm$, we have that $\dual(M_0)$, defined naturally, is a model of $T_\xm$. If $\dual(M_0)$ is   
isomorphic to $M_0$, then the ultrapowers of the two models are isomorphic, and so they have the same saturation. 
But maybe $\dual(M_0) \not\cong M_0$. However we know that 
$Th(M_0) = Th(\dual(M_0))$ hence it is well known that 
$(M_0)^\lambda/\de$, $(\dual(M_0))^\lambda/\de$ are $L_{\infty,\lambda^+}$-equivalent, hence the argument above holds. 
$($Really what we need is just Keisler's lemma that if $M_1 \equiv M_2$ in a language of size no more than $\lambda$ and 
$\de$ is a regular ultrafilter on $\lambda$, then $(M_1)^\lambda/\de$ is $\lambda^+$-saturated if and only if 
$(M_2)^\lambda/\de$ is $\lambda^+$-saturated.$)$
\end{rmk}

\begin{disc}
If we were not assuming $\xm = \dual(\xm)$, then we would just need to add the parallel for types containing $\cP(x)$.  
We may do this in at least two ways: either we update $(\star)_{M_0, I, \de}$ to include the parallel condition for $\cP(x)$ replacing $\cQ(x)$ 
$($with the corresponding minor notational changes$)$, or else, we add the condition 
 $(\star)_{\dual(M_0), I, \de}$, where $\dual(M_0)$ is a model of $T_{\dual(\xm)}$, since $\cQ$ in $\dual(M_0)$ is 
the parallel of $\cP$ in $M_0$. 
\end{disc}

\vspace{5mm}

\setcounter{equation}{0}


\section{Sizes} \label{s:sizes}

In \S \ref{s:nt}, we built theories $T_\xm$ from templates $\xm$ under quite general conditions on the 
template edge relation $\mcr$.  It was mentioned that for our main work below, we will 
want to choose our $\mcr$'s carefully to reflect certain sequences of sparse random graphs. The construction of  
these sparse random graphs is the task of this section, leading to the definition of our continuum many parameters in 
\ref{4.8} (building on \ref{d:32}).  
Claim \ref{c:helpful} verifies any such parameter satisfies the axioms of \S \ref{s:nt}. 

Some remarks on the problems this section must solve and how it solves them: 

First of all, \ref{4.8} must be able to output continuum many parameters which have a reasonable chance of 
being ``very different'' from each other, with no one of any pair obviously more or less complex than the other; more on this soon. 

Second, the data of a parameter is essentially\footnote{Set aside the level function for the moment.} 
summarized by the data of a bipartite graph associated to every 
pair $(\nu, \eta) \in \mcr_k$ for every $k<\omega$ (i.e., the bipartite graph corresponding to the pattern of $\mcr_{k+1}$-edges 
on their immediate successors).  
For the parameters built here, this associated graph will be an invariant of the level $k$ (i.e., branching at each level $k$ is uniform, we specify for each level $k$ a bipartite graph of the right size, and 
then for every $\mcr_k$-edge $(\nu, \eta)$ the pattern of $\mcr_{k+1}$-edges between the immediate successors 
of $\nu$ and $\eta$ copies this fixed bipartite graph).  So we may think of our theories as 
reflecting bipartite graph sequences. 

In considering how to construct and compare bipartite graph sequences, remember that the 
resulting $\xm$ should be self-dual.  Random bipartite graphs may not have the needed symmetry.  
So we shall simply construct sequences of graphs, allowing 
self-loops (edges from a vertex to itself), and only at the last minute convert them to symmetric  
bipartite graphs by doubling the vertices ($\ref{r:6.8}$). 

The key point which will illuminate comparison is that we require our sequences of graphs 
$\langle E_i : i < \omega \rangle$ to have a well-defined notion of 
large and small in the following sense.  The $i$-th graph will have vertex set $m_i = \{ 0, \dots, m_i-1\}$, and 
``small'' and ``large'' will depend on $i$. 
We shall require that for every small subset $u$ of $m_i$, there is some vertex 
connected to all elements of $u$, whereas for every large subset $v$, there is no vertex   
connected to all elements of $v$. Briefly, all small subsets are covered and no large subsets are.  
A main aim of the section is showing this can be arranged for certain explicitly given growth rates 
by using fairly sparse random graphs.  (A clear picture of these rates is essential as they will interface 
directly with the chain condition in \ref{d1:cca}, in a non-obvious interaction of finite and infinite.\footnote{and of 
genericity and randomness in the finite and in the infinite settings.})

Finally, to motivate incomparability, 
one may imagine that the ``small'' and ``large'' conditions could transmute in later sections into conditions on consistency of sets of formulas, 
and so hope to build sequences of graphs 
with orthogonal notions of small and large, meaning perhaps 
there would be many indices $i$ for which ``small'' in the first sequence is larger than ``large'' in the second sequence, and vice versa. 

Our solution is to start with a single sequence of graphs $\langle E_i : i < \omega \rangle$ and a family of level functions which is independent in the sense of 
\cite{ek} (see \ref{fact-p}). The correct variation in large and small is naturally produced by ``turning off and on'' the 
constraints of $E_i$ in accordance with a level function; this is very robust, as summarized in \ref{4.8}.  
(This idea seems to have a certain naturalness; had we not introduced level functions in \S \ref{s:nt}, at this point we might be obliged to define them.)

\br

We begin by laying out requirements on the sequence of integers which will specify the branching (and so the sizes of the 
vertex sets in our sequence of graphs). 

\begin{defn} \label{d:a2} Say that the countable sequence $\bar{m} = \langle m_i : i < \omega \rangle \in {^\omega (\omega \setminus \{ 0 \})}$ is 
a \emph{fast sequence} 
when for each $i$, $m^\circ_i := \prod_{j<i} m_i$ satisfies
\begin{equation} \label{e:growth}
m_i \geq \left( (m^\circ_i)^{i^i}\right)^{4(m^\circ_i)^{i^i}}. 
\end{equation}  
Since the $m^\circ_i$'s are uniquely determined by
the $m_i$'s, we may sometimes display $\bar{m}$ as a sequence of pairs $\langle (m_i, m^\circ_i) : i < \omega \rangle$, 
while referring to it as a sequence of singletons.  To avoid triviality, we assume $m_0 > 1$. 
\end{defn}


In the next definition, note the $E_i$ are graphs, not bipartite graphs.  We have chosen to allow self-loops 
(i.e. $(a,a)$ can be an edge), but this is not crucial. 

\begin{defn} \label{f:a3} \label{d:incr}
Let $\bar{m}$ be a fast sequence.
\begin{enumerate}
\item[(1)] $\bar{E}$ will denote a \emph{sequence of graphs for $\bar{m}$} when 
each $E_i \subseteq [m_i]^2 \cup \{ (a,a) : a \in [m_i] \}$ 
and for $i=0$ we have equality. 
\item[(2)] $\bar{E}$ is a \emph{good sequence of graphs for $\bar{m}$}, or just \emph{good for $\bar{m}$}, when in addition, 
there is some finite $i_*$ such that for all $i \geq i_*$, 
\begin{enumerate}
\item[(i)] if  $u \subseteq m_i$, 
\[ |u| \leq (m^\circ_i)^{i^i} \] 
then $u$ is ``$\bar{E}$-small,'' meaning that 
that there is $s$ such that 
\[  (\forall t \in u) (E_i(s,t)). \] 
\item[(ii)]  if $u \subseteq m_i$ and 
\[ |u| \geq \frac{m_i}{(m^\circ_i)^{i^i}}  \] 
then $u$ is ``$\bar{E}$-large,'' meaning 
that there is no $s$ such that 
\[ (\forall t \in u) (E_i(s,t)). \] 
\end{enumerate} 
We shall omit $\bar{E}$ in ``small'' and ``large'' when it is clear from the context. 
\end{enumerate}
\end{defn}

The definition of ``small'' is used in the proof of existence of a model completion and in the proof of 
\ref{9.28a} below, and the definition of ``large'' in the proof of
non-saturation below.

\begin{conv} \label{fn-g} For the next few lemmas, let $g: \mathbb{N} \rightarrow \mathbb{N}$ be the function given by 
\[ g(i) = 2(m^\circ_i)^{i^i}. \]
$($of course, we could have called this $g(m_i)$.$)$
\end{conv}

\begin{fact}[see e.g. Bollob\'as \cite{bollobas} Corollary 3.14] \label{rg-fact} Let $G$ be a graph. 
Let $\Delta(G)$ denote the maximal degree of $G$ and let $G_p$ denote a 
graph from $\mathcal{G}_{n,p}$, random graphs on $n$ vertices with edge probability $p = p(n)$. 
If $p n/ \log n \rightarrow \infty$, then a.e. $G_p$ satisfies 
\[ \Delta(G_p) = \{ 1 + o(1) \}pn. \]
\end{fact}

\noindent In \ref{rg-fact}, note that $p$ is a function of $n$.

\begin{obs} \label{obs:upper}
Let $g$ be from $\ref{fn-g}$ and suppose 
\[ p = p(m_i) = \frac{1}{(m_i)^{\frac{1}{g(i)}}}. \] 
Then $p\cdot m_i / \log m_i \rightarrow \infty$, so 
$\ref{rg-fact}$ applies. In particular, as $i \rightarrow \infty$, the proportion of $G \in \mathcal{G}_{m_i, p}$ which have  
no vertices of ``large'' degree goes to 1.  
\end{obs}

\begin{proof} Recalling that $g(i)$ from \ref{fn-g} is monotonic and strictly increasing and approaches $\infty$, and recalling that 
$p = \frac{1}{(m_i)^{\frac{1}{g(i)}}}$ we have 
\begin{equation} \lim_{i \rightarrow \infty} \left(    \frac{1}{(m_i)^{\frac{1}{g(i)}}}         \right) \left( \frac{m_i}{\log m_i} \right)  = 
\lim_{i \rightarrow \infty} \left(   \frac{m_i^{1-\frac{1}{g(i)}}  }{ \log m_i }  \right) = \infty. 
\end{equation}
So Fact \ref{rg-fact} applies (actually $g(i) \geq 2$ suffices) 
and for some fixed constant $c$, in almost every $G_{p,m_i}$, every vertex has degree $\leq c p m_i$. 
Let us verify that this degree is eventually not ``large'' in the sense of \ref{f:a3}. For this it would suffice to show that 
\begin{equation} \frac{m_i}{(m^\circ_i)^{i^i}}  \mbox{ is quite a bit  bigger than } \frac{m_i}{{m_i}^{1/g(i)}} 
\end{equation}
[the left-hand side is the threshold for ``large'' from \ref{f:a3} and the right-hand side is $p \cdot m_i$]
and for this it would suffice to show that 
\begin{equation} (m^\circ_i)^{i^i} \mbox{ is quite a bit smaller than }  {m_i}^{1/g(i)}. \end{equation}
which is ensured by (\ref{e:growth}) of Definition \ref{d:a2} and the definition of $g$ in \ref{fn-g}.  
\end{proof}

\begin{obs} \label{four}
For each $i$, $(m^\circ_i)^{i^i} < (m_i)^{1/4}$.
\end{obs}

\begin{proof}
We verify that 
\begin{equation} 
\label{check-4}
(m^\circ_i)^{i^i}  <  \left(\left( (m^\circ_i)^{i^i}\right)^{4(m^\circ_i)^{i^i}}\right)^{1/4}  \leq (m_i)^{1/4} 
\end{equation}
recalling \ref{d:a2}(\ref{e:growth}).
\end{proof}

\begin{lemma} \label{f:a3a}  
Let $E_i \subseteq [m_i]^2 \cup \{ (a,a) : a \in [m_i] \}$ be symmetric and random with probability $p$ from 
$\ref{obs:upper}$ 
for each pair to be an edge. 
In such a graph, with probability close to $1$ $($for us nonzero probability is sufficient$)$:
\begin{enumerate}
\item for every $u \subseteq m_i$, if 
\[ |u| \leq (m^\circ_i)^{i^i} \] 
then there is $s$ so that 
$(\forall t \in u) (E_i(s,t))$.  
\item   for every $u \subseteq m_i$, if 
\[ |u| \geq \frac{m_i}{(m^\circ_i)^{i^i}} \] 
then there does not exist $s$ so that 
$(\forall t \in u) (E_i(s,t))$. 
\end{enumerate} 
\end{lemma}

\begin{rmk} \label{r:6.8}
For now we will define a graph on $m_i$ vertices with edge relation $E_i$, allowing self loops.
In the translation from $E_i$ to $\mcr_i$ in \ref{d:32} below, 
we will use this graph to make a bipartite graph, which will then be symmetric $[$in the sense that the derived $\xm = \dual(\xm)$$]$ by construction. 
\end{rmk}

\begin{proof}  
Recall $[n]$ denotes the $n$-element set $\{ 0, \dots, n-1 \}$. 
We define a probability measure ${{\mu}} ={{\mu}}_i $ on $\{ X : X \subseteq [m_i]^2 \cup \{ (a,a ) : a \in [m_i] \} ~\}$ by flipping a coin $c_{\{a,b\}}$ for each 
potential edge\footnote{note that this edge relation is by definition symmetric.} with probability of heads (=yes) being 
\begin{equation} \frac{1}{   (m_i)^{\frac{1}{g(i)}}   } \end{equation}
where $g(i)$ is again from \ref{fn-g}.  
Condition (2) will be handled by Observation \ref{obs:upper}, so let us address condition (1).   Let us say that a set $v \subseteq [m_i]$ is 
``covered'' if there exists $b \in [m_i]$ such that $(\forall t \in v)(E_i(b,t))$. Clearly it will suffice to show that all sets of size 
$(m^\circ_i)^{i^i}$ are covered. 

Let $\mce^1_i$ be the probability that an arbitrary (given) 
$v \subseteq [m_i]$ of size $x$ is \emph{not} covered ($x$ to be chosen later, in our main case  
$(m^\circ_i)^{i^i}$.)  Before bounding $\mce^1_i$, note that
the probability that {some} $v \subseteq [m_i]$ of size $x$ is not covered is 
\begin{equation} \label{ebin}
\binom{m_i}{x} \left(  \mce^1_i \right)
\end{equation}
i.e. the number of ways to choose a $v$ of size $x$ times the probability that a given $v$ is not covered.  
Now  
\begin{equation} \label{e25}
\mce^1_i = \left( 1 - \frac{1}{(m_i)^{x/g(i)}}\right)^{m_i } 
\end{equation}
[the term in parentheses represents the chance that each particular $b$ fails to cover $v$; there are $m_i $ choices for $b$]. 
Recalling that $(1-1/n)^n$ is well approximated by $\frac{1}{e}$ for large $n$, we may rewrite the right side of  (\ref{e25}) as 
\begin{equation} \label{e266}
\left(        \left(   1 -   \frac{1}{   (m_i)^{x/g(i)}   }  \right)^{(m_i)^{x/g(i)} }     \right)^{ \frac{m_i}{m_i ~^{x/g(i)}}} 
\end{equation}
and then (\ref{e266}) is well approximated by 
\begin{equation}
\frac{1}{e^{[  (m_i)^{1-x/g(i) }      ]} }.
\end{equation}
We need equation (\ref{e266}) to be very small, so we need $e^{[  (m_i)^{1-x/g(i) }      ]}$ to be very large, so 
we need ${[  (m_i)^{1-x/g(i) }      ]}$ to be very large, so we need $1 - x/g(i)$ to be not too small. For our present calculations, 
let us verify that it suffices to have
$x/g(i) = 1/2$, which is satisfied in our main case when $x = (m^\circ_i)^{i^i}$ 
from \ref{f:a3} and $g(i)$ is from \ref{fn-g}.  
In this case, $\mce^1_i$ is well approximated by 
\begin{equation}
 \frac{1}{e^{\sqrt{m_i}}}
\end{equation}
hence an upper bound for equation (\ref{ebin}) is well approximated by 
\begin{equation}
(m_i)^x ( \mce^1_i ) \approx \left(e^{x \ln m_i}\right) \left(\frac{1}{e^{\sqrt{m_i}}}\right)  = \frac{1}{e^{\sqrt{m_i} - x \ln m_i}}. 
\end{equation}
It is sufficient for us that the exponent be nonnegative and growing; for instance, $x < (m_i)^{1/4}$ 
suffices for us, and is indeed satisfied in our main case $x = (m^\circ_i)^{i^i}$ 
recalling \ref{four}.  This completes the proof. 
\end{proof}

\begin{concl} \label{f:a3ab}
If $\bar{m}$ is a fast sequence, there exists $\bar{E}$ which is good for $\bar{m}$. 
\end{concl}

\begin{rmk} \label{remark1}
The bounds in $\ref{f:a3}$ are for definiteness, what we really use is $\ref{f:a3a}$-$\ref{f:a3ab}$:  
$\bar{m}$  grows quickly enough to find a sequence of graphs $\bar{E}$ with a coherent and growing notion of  
``small'' and ``large'' $($every small set has an $x$ connected to all of its elements, and no large set has an $x$ connected to all of its elements$)$. 
Zero-one laws \cite{ShSp:304} suggest much flexibility in choosing such bounds.  
\end{rmk}

Next, given the three key ingredients $\bar{m}$, $\bar{E}$, and 
a level function $\xi$ (recall \ref{d:level}), we construct a parameter $\xm$ whose sequence of reduced graphs 
naturally reflects $\bar{E}$. 

\begin{ntn}
 For $\bar{m}$ a countable sequence of natural numbers and $\eta \in {^{\omega>}\omega}$ or 
$\eta \in {^\omega\omega}$, 
the notation $\eta < \bar{m}$ means 
$\eta(i) < m_i$ for all $i < \lgn(\eta)$ $($i.e. in $\dom(\eta)$$)$.  
\end{ntn}

\begin{defn}[The parameters we use] \label{d:32}
For any fast sequence $\bar{m} = \langle m_i : i < \omega \rangle$, any $\bar{E} = \langle E_i : i <\omega \rangle$ which is good for $\bar{m}$, and any level function 
 $\xi \in {^\omega 2}$, 
we define a parameter $\xm = \prm[\bar{m}, \bar{E}, \xi]$ as follows:
\begin{enumerate}
\item if $\ell = 1, 2$, $\mct_{\xm, \ell} = \{ \eta : \eta \in {^{\omega >}\omega},  \eta < \bar{m} \}$.  
\\ In particular, for each $\nu \in \mct_\ell$, if $\nu \in \mct_{\ell,i}$ then $|\suc_{\mct_\ell}(\nu)| = m_{i}$.

\item for the next two items,  for $i<\omega$,
let $E^1_i = E_i$, and let $E^0_i$ be $m_i \times m_i$, the complete graph with self-loops. 
\item for $i=0$, $\mcr_0 = \{ ( \emptyset, \emptyset ) \}$.  

\item for $i+1$, 
\[ \mcr_{i+1} = \{ (\eta_1, \eta_2) : \eta_\ell \in \mct_{\ell, i+1}, ~ (\eta_1 \rstr i, \eta_2 \rstr i) \in \mcr_i, ~(\eta_1(i), \eta_2(i)) \in E^{\xi(i)}_{i} \}. \]
\end{enumerate}
\end{defn}

\begin{conv}
In what follows, we use the letters $\xm$ and $\xn$, often with subscripts,
for parameters.
We will often write e.g.  
\[ \xm = \xm[\bar{m}, \bar{E}, \xi]  \mbox{ or } \xn = \xn[\bar{m}, \bar{E}, \xi]  \mbox{ or } \xm = \prm[\bar{m}, \bar{E}, \xi] \]  
to indicate the dependence of a given parameter on the three inputs. 
\end{conv}

\begin{obs} \label{6.13a}
Suppose $\bar{m}, \bar{E}$ are as above and $\xi_1, \xi_2 \in {^\omega \{ 0, 1\}}$ and $\xm_\ell = \prm[\bar{m}, \bar{E}, \xi_\ell]$. 
If we have $\xi^{-1}_2 \{ 1 \} \subseteq \xi^{-1}_1 \{ 1\}$, then $\mcr^{\xm_1} \subseteq \mcr^{\xm_2}$. 
\end{obs}

(Observation \ref{6.13a} will be used in \S \ref{s:p-omega} below.  Informally, if there are more active levels, there are more constraints, 
so $\mcr$ has fewer edges.)

\begin{disc}[Indexing]  \label{d:indx}
\emph{So that the key points of the construction in $\ref{d:32}$ are not hidden in the indexing, we review:} 
\begin{itemize}
\item[(a)] \emph{$\mct_{1,0} = \mct_{2,0} = \{ \emptyset \}$, and $\mcr_0 = \{ (\emptyset, \emptyset) \}$. }
\item[(b)] \emph{For $i+1$, recall that the elements of $\mct_{\ell, i+1}$ are sequences of length $i+1$, i.e. 
functions $\eta$ from $\{ 0, \dots, i \}$ to $\omega$
subject to the constraint that $\eta(j) < m_j$ for each $j\leq i$. }
\item[(c)] \emph{$\mct_{\ell, 1}$ has $m_0$ nodes; in general, for $k>0$ $\mct_{\ell, k}$ has $\prod_{j < k} m_j = m^\circ_{k}$ nodes.  } 
\emph{ Also for $k = 0$, $m^\circ_k = 1 = \prod_{j<0} m_j$.}
\item[(d)] \emph{ If $\eta \in \mct_{\ell, k}$ then $\eta$ has $m_k$ immediate successors in $\mct_{\ell, k+1}$, so it follows 
that $\mct_{\ell, k+1}$ has size $(\prod_{j<k} m_j )\cdot m_k = \prod_{j<k+1} m_j = m^\circ_{k+1}$, as desired.  }
\item[(e)] \emph{ If $(\eta, \nu) \in \mct_{1, k} \times \mct_{2,k}$ and $(\eta, \nu) \in \mcr_k$, then 
letting $A = \suc_{\mct_{1,k+1}}(\eta)$, $B= \suc_{\mct_{2,k+1}}(\nu)$, if $k+1$ is an active level, 
then $(A, B,  \mcr_{k+1} \rstr A \times B)$ 
is isomorphic as a bipartite graph to $(m_k, m_k, E_k)$.  If $k+1$ is a lazy level, it is a complete bipartite graph. }
\item[(f)] In a slogan: given two elements of length $k$ connected by $\mcr_k$, 
$E_k$ gives us the pattern of edges between their sets of successors, provided those successors are at an active level. 
\textbf{Note to the reader:} so there is no confusion, we repeat that $\xi(k)$ affects edges at level $k+1$, recalling $\ref{r:level}$. 
\end{itemize}
\end{disc}

\br
\noindent Note that \ref{d:32} establishes self-duality essentially for free.
(The key point is that each of the $\mcr_{i}$'s is 
symmetric as a bipartite graph; from each $\mcr_{i+1}$ one can naturally recover $E^{\xi(i+1)}_i$ by identifying each 
vertex on the left with its parallel on the right.)

\begin{claim} \label{is-self-dual} 
If $\xm$ is constructed from any $\bar{m}$, $\bar{E}$ which is good for $\bar{m}$ as in $\ref{d:32}$, and $\xi$ which 
satisfies $\xi(i) = 0$ for $i < i_* = i_*(\bar{E})$ $($recalling $\ref{f:a3}(2)$$)$, then 
$\xm = \dual(\xm)$. 
\end{claim}

\begin{proof}
Immediate: this is the symmetry of the construction in \ref{d:32}. 
\end{proof}

\begin{claim} \label{c:helpful}
If $\xm$ is constructed from any $\bar{m}$, $\bar{E}$, $\xi$ as in $\ref{is-self-dual}$,
then $\xm$ is a parameter.   Moreover, for every $k < \omega$, 

\begin{enumerate}
\item if $\{ \rho_i : i < s \} \subseteq \mct_{2,k+1}$  and $s \leq (m^\circ_k)^{k^k}$, and $\nu \in \mct_{1,k}$ with 
$(\nu, \rho_i \rstr k) \in \mcr_k$ for each $i < s$, then there exists an immediate successor $\nu^\smallfrown \langle t \rangle$ of $\nu$ such that 
$(\nu^\smallfrown \langle t \rangle, \rho_i) \in \mcr_{k+1}$ for each $i<s$.
\item the parallel condition to (1) holds switching $\mct_{2,k+1}$ and $\mct_{1,k+1}$. 
\end{enumerate}
\end{claim}

\begin{proof}
We check \ref{d:param} and \ref{d:nice}.   \ref{d:param}(1), (2), (3) are obvious from our construction. (4) holds since the degree of 
a vertex in $E^a_i$ is at least two. (5) is immediate from the construction and (6) does not require verification. 

For \ref{d:nice}, it will be  helpful to first prove the ``moreover'' clauses of our Claim, which greatly strengthen Extension 
in one aspect.  Note that as $\xm = \dual(\xm)$ it suffices to 
prove (1). 
Consider $s \leq (m^\circ_k)^{k^k}$ and $\nu \in \mct_{1,k}$ and $\rho_0, \dots \rho_{s-1} \in  \mct_{2,k+1}$ 
such that $(\nu, \rho_i \rstr k) \in \mcr_k$ for $i < s$. Write $\rho_i = {( \rho_i \rstr k) }^\smallfrown \langle \ell_i \rangle$ for each $i < s$.  In the graph $E_k$, $\{ \ell_i : i < s \}$ is a small 
subset of $[m_k]$, so there is some $t \in [m_i]$ such that $(t, \ell_i)$ is an edge in $E_k$ for every $i < s$. Recall that if $\xi(k) = 1$, for each $i < s$, to form $\mcr_k$ in \ref{d:32} we simply put a bipartite copy of $E_k$ [always the same $E_k$] between 
the immediate successors of $\nu$ and the immediate successors of $\rho_i \rstr k$. As a result, for each $i<s$ (we just look one by one), $(\nu^\smallfrown \langle t \rangle, { \rho_i \rstr k }^\smallfrown \langle \ell_i \rangle) \in \mcr_{k+1}$. 
In other words, $\nu^\smallfrown \langle t \rangle$ connects to each $\rho_i$ (again note: even though they may not have an immediate common predecessor).  
If $\xi(k)=0$, we use a complete graph instead of $E_i$, so this is all true a fortiori and there are even many such elements. 
This proves (1), and also (2). 

Finally, we check the extension conditions from \ref{d:nice}, and as $\xm$ is self-dual, it will suffice to prove one way.   
Applying the preceding paragraph in the case $s = k$ tells us there is at least one successor of $\nu$ connecting to 
$\rho_0, \dots \rho_{k-1}$.  The additional point is to note that there is not just one $t$ but at least $k +1$. When $\xi(k) = 0$, as noted, there will be many such successors, actually all. 
To see that there are many when $\xi(k)=1$,  by induction on $\ell \leq k$ choose $t_\ell \in [m_k] \setminus \{ t_0, \dots, t_{\ell-1} \}$ 
such that $\bigwedge_{i<k} (\nu^\smallfrown \langle t_\ell \rangle, \rho_i) \in \mcr_{k+1}$, as follows:  
note that we could choose $\rho$ such that $j < \ell \implies (\nu^\smallfrown \langle t_j \rangle, \rho) \notin \mcr_{k+1}$, possible as no vertex 
has large degree, and apply the previous paragraph to the still small set $\{ \rho_i : i < k\} \cup \{ \rho \}$ to  
find $\nu^\smallfrown \langle s \rangle$ which connects to $\rho_0, \dots, \rho_{k-1}, \rho$, necessarily for $s \neq t_j$ for $j < \ell$, so $s$ can serve as $t_\ell$. Continuing in this way, we find the $k+1$ elements for the extension axiom. 
\end{proof}

\begin{rmk} \label{r:same-graph}
In the proof of $\ref{c:helpful}$, it is worth observing the following helpful feature of using the same $E_k$ across the entire level.  
Suppose we have $\nu \in \mct_{1,k}$, $\eta, \rho \in \mct_{2,k}$, and $(\nu, \eta) \in \mcr_k$, 
$(\nu, \rho) \in \mcr_k$. Suppose $(i,j)$ and $(i,\ell)$ are edges in $E_k$, and to avoid triviality, suppose that 
$\xi(k) = 1$.  Then of course
$(\nu^\smallfrown \langle i \rangle, {\eta}^\smallfrown \langle j \rangle)$ and 
$(\nu^\smallfrown \langle i \rangle, {\eta}^\smallfrown \langle \ell  \rangle)$ are both in $\mcr_{k+1}$, but also notice 
$(\nu^\smallfrown \langle i \rangle, {\eta}^\smallfrown \langle j \rangle)$ and 
$(\nu^\smallfrown \langle i \rangle, {\rho}^\smallfrown \langle \ell  \rangle)$ are both in $\mcr_{k+1}$, simply because 
we consulted the same $E_k$ in both cases. 
\end{rmk}

We may now invoke the level functions to build a family of parameters whose active levels are independent in the following precise sense.  

\begin{ntn}
Recall that $A \subseteq^* B$ means that $\{ a \in A : a \notin B \}$ is finite.
\end{ntn}

\begin{fact} \label{fact-p} For any $i_* < \omega$, ~
there is $\Xi = \{ \xi_\alpha : \alpha < 2^{\aleph_0} \} \subseteq {^\omega \{0,1\}}$ of size continuum such that: 

\begin{quotation}
\noindent if $u \subseteq 2^{\aleph_0}$ is finite and $\beta < 2^{\aleph_0}$ is not from $u$, then 
\[   \xi^{-1}_\beta \{1\} \not \subseteq^* \bigcup \{ \xi^{-1}_\alpha \{1\} : \alpha \in u \} \]
\end{quotation}
and moreover $\xi_\alpha(i) = 0$ for all $i<i_*$ and all $\alpha < 2^{\aleph_0}$.
\end{fact}

\begin{proof}
We can use e.g. the existence of a family $\mathcal{G} = \{ g_\alpha : \alpha < 2^{\aleph_0} \} \subseteq {^\omega \omega}$ 
of continuum many independent functions, see \cite{ek} or \cite{cn} or \cite{Sh:c} Appendix, Theorem 1.5(1) p. 656. 
Recall that this means that each $g \in \mathcal{G}$ 
is a function from $\omega$ to $\omega$ and for any finite $n$, any distinct $\alpha_0, \dots, \alpha_{n-1} < 2^{\aleph_0}$, and any 
values 
$t_0, \dots, t_{n-1} < \omega$ (not necessarily distinct), 
the set $\{ i < \omega : g_{\alpha_0}(i) = t_0 ~\land \cdots \land g_{\alpha_{n-1}}(i) = t_{n-1} \} \neq \emptyset$.  
In particular (as we can play with setting the values of any finitely many other functions) it follows from the definition of independent that for 
any distinct $\alpha_0, \dots, \alpha_{n-1}, \beta < 2^{\aleph_0}$ and any $s_0, \dots, s_{n-1}, t < \omega$, the set 
\[ \{ i < \omega : g_{\alpha_0}(i) = s_0 ~\land \cdots \land g_{\alpha_{n-1}}(i) = s_{n-1} ~\land ~ g_\beta(i) = t \} \] is infinite. 
For each $g_\alpha$ in our family $\mathcal{G}$, let $\xi_\alpha$ be the function $i \mapsto ( g_\alpha(i) \mod 2 )$. It follows that 
$\{ \xi_\alpha : \alpha < 2^{\aleph_0} \}$ is as desired.   

Since changing the first $i$ values for every $g_\alpha$ in the family to be zero    
does not alter the independence, we can ensure the last line for any finite $i_*$ given in advance. 
\end{proof}

\begin{rmk} 
We could have also demanded that the sets $\xi^{-1}_\alpha \{1\}$ are infinite and pairwise almost disjoint.
\end{rmk}

\begin{cor} \label{4.8}
Thus, for any fast sequence $\bar{m}$, any $\bar{E}$ which is good for $\bar{m}$, and $\Xi = \langle \xi_\alpha : \alpha < 2^{\aleph_0} \rangle$ from 
$\ref{fact-p}$ with $i_* = i_*(\bar{E})$, 
we can define a set 
\[ \sM_* = \{ \xm_\alpha = \prm[\bar{m}, \bar{E}, \xi_\alpha] : \alpha < 2^{\aleph_0} \} \]
with no repetition. For any $\mathcal{M} \subseteq \sM_*$, define 
\[ \mcim = \{ v \subseteq \omega : \mbox{ for some } \xm_{\alpha_0}, \dots, \xm_{\alpha_{n-1}} \in \mcm, ~ v \subseteq^* 
\bigcup \{ \xi^{-1}_{\alpha_\ell}\{1\} : \ell < n \} \}. \]
Then if $\mcn \subseteq \sM_* \setminus \mcm$, we will have that:
\begin{enumerate}
\item[(a)] if $\xm_\alpha \in \mcm$, then $\xi^{-1}_\alpha \{1\} \in \mcim$.
\item[(b)] if $\xm_\beta \in \mcn$, then $\xi^{-1}_\beta \{1\} \neq \emptyset \mod \mcim$. 
\end{enumerate}
\end{cor}

\begin{rmk}
Corollary $\ref{4.8}$ summarizes a natural sense in which any two elements, or indeed any two disjoint subsets, of $\sM_*$ are orthogonal. 
\end{rmk}

\vspace{5mm}

\section{Possibility patterns and ultrapowers} \label{d:separation}

We will be interested in analyzing types in regular ultrapowers, and the following set-up is especially useful for this. 
To readers familiar with ``separation of variables'' in the sense of \cite{MiSh:999}, there is nothing new here;  
we simply explain that framework and fix some notation. 

The first idea is that a regular ultrafilter on $\lambda$ can be ``projected'' onto any reasonable Boolean algebra (complete, of size $\leq 2^\lambda$, 
with the $\lambda^+$-c.c.) and studied there.  Let us give the definition, then discuss how it can be used. 

\begin{defn}[Regular ultrafilters built from tuples, from \cite{MiSh:999} Theorem 6.13] \label{d:built}
Suppose $\de$ is a regular ultrafilter on $I$, $|I| = \lambda$. We say that $\de$ is built from 
$(\de_0, \ba, \de_*, \jj)$ when: 

\begin{enumerate}
\item {$\de_0$ is a regular, $|I|^+$-good filter on $I$} 
\item {$\ba$ is a Boolean algebra}
\item {$\de_*$ is an ultrafilter on $\ba$}
\item {$\jj : \mcp(I) \rightarrow \ba$ is a surjective homomorphism such that:}
\begin{enumerate}
\item $\de_0 = \jj^{-1}(\{ 1_\ba \})$ 
\item $\de = \{ A \subseteq I : \jj(A) \in \de_* \}$.
\end{enumerate}
\end{enumerate}
\end{defn}

Since \ref{d:built} is defined with an eye towards Keisler's order, an important feature of this definition is that the problem 
of realizing types is naturally projected to the Boolean algebra, too, as \ref{d:built-b} explains.\footnote{The next 
definition seems to suggest that in $\vp(x,y)$, $\ell(y)=1$. This will be our main case in this paper, but it's not a constraint.} 

\begin{defn} \label{d:built-b}
Continuing in the notation of $\ref{d:built}$, 
suppose that $\de$ is built from $(\de_0, \ba, \de_*, \jj)$.  Consider a complete countable $T$ and $M \models T$. 
Suppose $N = M^\lambda/\de$ and $p$ is a partial type over $\kappa \leq \lambda$ parameters in $N$, 
so $p = \langle \vp_\alpha(x,{a}_\alpha) : \alpha < \kappa \rangle$. 
We say that $\bar{\mb}$ represents the type $p$ when:  
for each finite $u \subseteq \kappa$, the \Los map \L ~sends $u \mapsto B_u$ where 
\[ B_u := \{ t \in I : M \models (\exists x) \bigwedge_{\alpha \in u} \{ \vp_\alpha(x, a_\alpha[t]) \}. \]
and $\mb_u = \jj(B_u)$.   We write $\bar{B} = \langle B_u : u \in [\kappa]^{<\aleph_0} \rangle$, and 
$\bar{\mb} = \langle \mb_u : u \in [\kappa]^{<\aleph_0} \rangle$. 
\end{defn}

Continuing in the notation of \ref{d:built-b}, recall that $\langle B_u : u \in [\lambda]^{<\aleph_0} \rangle$ of 
elements of $I$ is called \emph{multiplicative} if $B_u \cap B_v = B_{u \cup v}$ for every $u,v \in [\lambda]^{<\aleph_0}$. 
To say that $\bar{B}$ has a multiplicative refinement in $\de$ means that there is $\bar{B}^\prime = \langle B^\prime_u : u \in [\lambda]^{<\aleph_0} \rangle$
such that for each $u$, $B^\prime_u \in \de$ and $B^\prime_u \subseteq B_u$, and moreover, $\bar{B}^\prime$ is 
multiplicative. Likewise (just a little more briefly) $\bar{\mb}$ has a multiplicative refinement in $\de_*$ if there exists a sequence
$\bar{\mb}^\prime = \langle \mb^\prime_u : u \in [\lambda]^{<\aleph_0} \rangle$ of elements of $\de_*$ such that
$\mb^\prime_u \subseteq \mb_u$ for each $u \in [\lambda]^{<\aleph_0}$, and $\mb^\prime_u \cap \mb^\prime_v = 
\mb^\prime_{u \cup v}$ for all $u,v, \in [\lambda]^{<\aleph_0}$. 

\begin{fact}[Separation of variables theorem, special case, \cite{MiSh:999} Theorem 6.13]  \label{t:separation}
In the notation of $\ref{d:built-b}$, 
$\langle B_u : u \in [\kappa]^{<\aleph_0} \rangle$ has a multiplicative refinement in $\de$ [so $p$ is realized in 
$N$] if and only if $\langle \mb_u : u \in [\kappa]^{<\aleph_0} \rangle$ has a multiplicative refinement in $\de_*$. 
\end{fact}

Given $\ba$ and $\de$, it can be useful to remember an ultrapower it came from. 

\begin{ntn}  \label{n:env}
Given some $\de_*$ and $\ba$, and given a corresponding choice of 
$\de$, $\de_0$, and $M$ as in \ref{d:built-b}, we may call $M^\lambda/\de$ 
an ``enveloping ultrapower'' for $\de_*$ and $\ba$. 
\end{ntn}

In light of \ref{d:built-b}, \ref{d:built} allows us to build regular ultrafilters in interesting ways 
by focusing on the construction of ultrafilters on quite general Boolean algebras: 
such a $\de_0$ can always 
be built when $\ba$ is complete, $|\ba| \leq 2^\lambda$ and has the $\lambda^+$-c.c.  
In earlier papers, $\ba$ was often just the completion of some free Boolean algebra. 
In this paper we take this quite a bit further, 
building the Boolean algebras and the ultrafilters on them together, by induction.

\emph{Possibility patterns.} 
Definition \ref{d:built-b} highlights sequences $\langle \mb_u : u \in [\kappa]^{<\aleph_0} \rangle$ of 
elements of $\ba^+$ which come directly from patterns of types in some enveloping ultrapower.   
In \cite{MiSh:999} and subsequent papers, we use 
\emph{possibility pattern} to mean any sequence $\bar{\mb}$ that arises from some $\bar{B}$ in this way. 
There is a combinatorial definition 
which doesn't rely on specifying an enveloping ultrapower, see \cite{MiSh:999} Definition 6.1. In the present paper, 
we will only need to handle sequences $\bar{\mb}$ which obviously come directly from types, so 
we won't need the full generality of that definition. So while we will often call such sequences ``possibility patterns,'' the reader 
may substitute a phrase like ``sequences $\bar{\mb}$ which represent some $\vp$-type $p$ in some enveloping ultrapower of the 
theory in question, in the sense of $\ref{d:built-b}$.'' 

\begin{conv}
When building Boolean algebras by induction, possibility patterns will arise as part of the data of \emph{problems}, and 
we will often refer to multiplicative refinements we add for them as \emph{solutions}.  
\end{conv}

In the proofs below, we use \ref{d:built} to repeatedly  
calibrate the building of a Boolean algebra $\ba$ along with an ultrafilter $\de_*$ on $\ba$, as follows. 
Suppose we are building $\de_*$ and $\ba$ together 
by induction, and at each stage in the construction we have some Boolean algebra $\ba_\alpha$ and some ultrafilter $\de_\alpha$ on $\ba_\alpha$, 
and we want to ensure by the time we get to $\de_*$, $\ba$ that all relevant possibility patterns $\bar{\mb}$, for a given countable theory $T$ or set  
$\mct$ of countable theories, have multiplicative refinements. At each stage in the construction 
we'll be handling some such $\bar{\mb}$, and at that point we may choose some enveloping ultrapower (for $\ba_\alpha, \de_\alpha$) and work there with a choice of corresponding $\bar{B}$, where the picture may be clearer. 
In this way we eventually construct $\ba$ and $\de_*$ to handle all relevant possibility patterns. At the end of the construction, one final use of \ref{d:built} 
connects it to Keisler's order:  
any regular $\de$ built from this $\de_*$ will be 
a regular ultrafilter on $\lambda$ with the property that for any $T \in \mct$ and any model $M \models T$, $M^\lambda/\de$ is $\lambda^+$-saturated.    

\begin{conv} \label{c:complete-subalgebra}
For Boolean algebras, write that $\ba_1 \lessdot \ba_2$ to mean that $\ba_1$ is a complete subalgebra of $\ba_2$. 
\end{conv}

\begin{fact} \label{d:extend}
If $\bar{\mb}$ is a possibility pattern for $T$ in a complete Boolean algebra $\ba$, and $\ba \lessdot \ba^\prime$, then 
$\bar{\mb}$ remains a possibility pattern for $T$ in $\ba^\prime$. 
\end{fact}

\begin{ntn} \label{n:sv} \emph{In the context of separation of variables: }
\begin{itemize}
\item[(a)] Let $\langle B_u : u \in [\lambda]^{<\aleph_0} \rangle$ be from the \los map, as usual, recalling $\ref{d:built-b}$. 
\item[(b)] For $\psi[\bar{a}]$ any formula, possibly with parameters, of $M^I/\de$, let 
$A[~\psi[\bar{a}]~]$ be defined as 
\[ \{ t \in I : M \models \psi[\bar{a}[t]] \}. \]
\item[(c)] Given $B_u$ or $A[ ~\psi[\bar{a}]~]$ from $(a)$ or $(b)$, let $\mb_u$ or $\ma [ ~\psi[\bar{a}]~]$ be their images under $\jj$, respectively. 
\end{itemize}
\end{ntn}
Together, we might say that the $\mb_u$'s and the $\ma[~\psi~]$'s transfer first-order information from the ultrapower to the 
index models (via \lost theorem) and thence to the Boolean algebra (via $\jj$). 
We conclude the section by restating \ref{t:separation} in this language:

\begin{fact} \label{fact:nonsat}
Suppose $N = M^\lambda/\de$ is 
an ``enveloping ultrapower'' for $\de_*$ and $\ba$, 
$N \models T$, and $\bar{\mb} = \langle \mb_u : u \in [\kappa]^{<\aleph_0} \rangle$ is a possibility pattern for $T$ in $\ba$ 
which has no solution, i.e. no multiplicative refinement in $\de_*$.  Then $N$ is not $\kappa^+$-saturated $($specifically, it 
omits a type over a set of size $\leq \kappa$ corresponding naturally to $\bar{\mb}$$)$. 
\end{fact}

\vspace{5mm}
\section{The chain condition} \label{s:cc}

This section begins work on the ultrafilters. To motivate the construction, 
consider the limitations of the example of \cite{MiSh:1140}, the one previous example of incomparability in ZFC.  
In that paper, as usual, we had considered completions of free Boolean algebras of the form 
\[ \ba = \ba^1_{2^\lambda, {{\mu}}, \theta} \]
where $\theta$ was not necessarily countable 
[recall the notation means: $\ba$ is  
generated by $2^\lambda$ independent partitions each of size ${{\mu}}$, where the intersection of fewer than $\theta$ elements 
from different partitions is nonzero].
It was shown there that theories called $T_f$, 
distant precursors of our theories here, could be handled by \emph{some} ultrafilter on $\ba$ when $\theta > \aleph_0$, and 
by \emph{no} ultrafilter when $\theta = \aleph_0$.  This suggests that if we want to handle theories of this general form 
while keeping $\theta = \aleph_0$, \emph{we should use Boolean algebras which are in a strong sense not free}.   
In our present case, if we want to handle some $T_\xm$ while not handling another $T_\xn$, for $\xm, \xn$ orthogonal, 
perhaps we can start the induction with the completion of a free Boolean algebra and keep enough of the freeness 
for some failure of saturation for $T_\xn$ to persist even as we build the Boolean algebra to be gradually less free in a sense 
relevant to $T_\xm$ (adding formal solutions below). 
To successfully carry this out, the arbiter of the freeness we need will come in the form of a key chain condition \ref{d1:cca}.\footnote{Although this is not needed for reading the proof, 
readers who are set theorists may recognize in the chain condition some methods intimately connected to finite support iteration (and may also 
conjecture that there may be 
potential for very interesting further interaction here).} 

Specifically, in this section, we first define and explain the new chain condition in \ref{d1:cca}. Then, since in later sections we will 
build Boolean algebras by induction (as in \ref{lessdot-seq}), we will need some tools for ensuring the chain condition is preserved under 
such constructions. We define a so-called pattern transfer property in \ref{paircc}; 
using this property and \ref{9.14a}, we prove Lemma \ref{lem15}, which verifies we may preserve our  
chain condition along certain increasing continuous sequences of Boolean algebras. 

We continue notation from Section \S \ref{s:sizes}.  Recall: 

\begin{defn}  Let $\kappa$ be an uncountable regular cardinal. 
The Boolean algebra $\ba$ has the $\kappa$-c.c. when: given $\langle \ma_\alpha : \alpha < \kappa \rangle$ a sequence of elements of 
$\ba^+$, we can find $\alpha \neq \beta < \kappa$ such that $\ma_\alpha \cap \ma_\beta > 0$. 
\end{defn}

The following definition will be central for the next few sections. 

\begin{defn}[The $(\kappa, \mci, \bar{m})$-c.c.] \label{d1:cca}
Let $\kappa$ be an uncountable regular cardinal.  
Let $\mci$ be an ideal on $\omega$ extending $[\omega]^{<\aleph_0}$ and $\bar{m}$ a fast sequence. 
We say that the Boolean algebra $\ba$ has the $(\kappa, \mci, \bar{m})$-c.c.  when: 
given $\langle \ma_\alpha : \alpha \in \uu_2 \rangle$ with $\uu_2 \in [\kappa]^\kappa$ 
a sequence\footnote{or renaming, without loss of generality, $\uu_2 = \kappa$.}  
 of elements of $\ba^+$, we can find 
$j < \omega$, $\uu_1 \in [\kappa]^\kappa$ and $A \in \mci$ such that: 

\begin{quotation}
\noindent $\oplus$  for every $n \in \omega \setminus A$ and every finite $u \subseteq \uu_1$ and every $i < n - j$, if 
\[ \frac{m_n}{(m^\circ_n)^{n^i}} < |u| \leq m_n \]
then there is some $v \subseteq u$ such that 
\[ |v| \geq \frac{|u|}{(m^\circ_n)^{n^{i+j}}} ~~\mbox{ ~~and~~ }~~ \bigcap \{ \ma_\alpha : \alpha \in v \}  > 0_\ba. \]

\end{quotation}
\end{defn}

\begin{obs} \label{has-cc}
If $\ba$ has the $(\kappa, \mci, \bar{m})$-c.c. for some $\mci$ and $\bar{m}$, 
then $\ba$ has the $\kappa$-c.c. In particular, if $\kappa = \aleph_1$, then 
$\ba$ has the c.c.c. 
\end{obs}

\begin{proof}
Clearly, the condition in \ref{d1:cca} implies that given any $\kappa$ nonzero elements of the Boolean algebra, 
at least two of them must have nonzero intersection. 
\end{proof} 

\begin{disc} \emph{ Informal discussion of the $\kim$-c.c. For any $n$, remember from \ref{d:incr} that any $u$ such that }
\[ |u| \geq \frac{m_n}{(m^\circ_n)^{n^n}} \]
\emph{is \emph{large}, so if  $|u| \geq {m_n}/{(m^\circ_n)^{n^{i+j}}}$ for some $i+j<n$, then $u$ is, by monotonicity, still large. 
In the statement of \ref{d1:cca}, we use $\uu_2$ and $\uu_1$ for easy quotation later on, but of course 
without loss of generality $($after renumbering$)$ $\uu_2 = \kappa$. In this notation, given any sequence of 
$\kappa$ nonzero elements, after possibly moving to a subsequence $\uu$ of the same size, we can find a 
``shrinking factor'' $j$ so that $($after possibly excluding a small set of $n$'s$)$ whenever we have a \emph{really} 
large $u$, i.e. $|u| \geq m_n / (m^\circ_n)^{n^i}$ for some $i$ such that $i+j$ is still $<n$, then we can shrink it a little bit $($in the denominator the exponent  
$n^i$ becomes $n^{i+j}$$)$ to find a subset $v$ which is still large, and all consistent.}
\end{disc}

The next claim shows that our upgraded c.c. is sometimes easy to satisfy.
Recall from \ref{fact-iffa} that $\ba^0_{\alpha, {{\mu}}, \aleph_0}$ is the free Boolean algebra generated by  
$\alpha$ independent partitions each of size ${{\mu}}$, and $\ba^1_{\alpha, {{\mu}}, \aleph_0}$ is its completion. 

\begin{claim}  \label{931fa} 
Suppose ${{\mu}}$ is any infinite cardinal, $\kappa$ is regular and uncountable, $\alpha$ is an ordinal, and ${{\mu}} < \kappa \leq \alpha$. Then for any ideal $\mci$ on $\omega$ extending $[\omega]^{<\aleph_0}$, and  
any fast sequence $\bar{m}$, 
$\ba = \ba^1_{\alpha, {{\mu}}, \aleph_0}$ has the $(\kappa, \mci, \bar{\xm})$-c.c.
\end{claim}

\begin{proof}
Recall from \ref{d:1a} above that the elements of the form $\mx_f$  for  $f \in \fin_{\aleph_0}({\alpha})$  are dense in $\ba$.  
Suppose we are given $\langle \ma_\alpha : \alpha < \kappa \rangle$ a sequence of positive elements of $\ba$. 
First, for each $\alpha < \kappa$, we may choose $f_\alpha \in \fin_{\aleph_0}(\alpha)$ such that 
$\mx_{f_\alpha} \leq \ma_{\alpha}$.
For each $\alpha < \kappa$, let $u_\alpha = \dom(f_\alpha)$, so $u_\alpha$ is finite. 
By the $\Delta$-system lemma \ref{delta-system}, there is some $u_*$ and $\vv \in [\kappa]^{\kappa}$ such that $u_\alpha \cap u_\beta = u_*$ 
for $\alpha, \beta \in \vv$. Since the range of each $f_\alpha$ is a subset of ${{\mu}} < \kappa$, there is $\uu \in [\vv]^\kappa$ such that 
$f_\alpha \rstr u_* = f_\beta \rstr u_*$ for $\alpha, \beta \in \uu$. Notice this tells us for any finitely many 
$\alpha_0, \dots, \alpha_{n-1}$ from $\uu$, $f = \cup_{i<n} f_i$ is a function, thus $\mx_f > 0$.  
In other words, for any finite $v \subseteq \uu$, 
\[ \bigcap \{ \ma_\alpha : \alpha \in v \} > 0_\ba \] 
so condition is trivially satisfied [indeed, taking $A = \emptyset$ and $j=0$] for any $n$ and $u$, simply by using $v = u$. 
\end{proof}

\begin{rmk}
We have stated \ref{931fa} 
to be compliant with Definition \ref{d1:cca}, but notice the proof would go through  
for any ideal $\mci$ including the trivial ideal $\emptyset$. 
\end{rmk}

Starting with the next Discussion, we add asterisks to a few statements and definitions to indicate that they are useful variations. 

\begin{disc-star}  In the next Definition $\ref{newba}$ 
note we replace $\mu$ in the subscript by ``$<\kappa$'' $($it will be explained what this means$)$.  
This includes the natural case where $\kappa$ is a regular limit cardinal, 
thus weakly inaccessible.   This gives a natural alternate candidate for $\ba_*$ in $\ref{cc-conv}(2)$,  so we include it, but 
again, it is not our main case. 
For our main theorems, the case $\kappa = \mu^+$, and thus using as a base for our induction Boolean algebras of the form $\ba^1_{\alpha, \mu, \aleph_0}$, or even $\ba^1_{\kappa, \aleph_0, \aleph_0}$, will suffice. 
\end{disc-star}

\begin{defn-star} \label{newba}
Let $\theta = \cf(\theta) \leq \kappa = \cf(\kappa)$ and let $\alpha$ be an ordinal.

\begin{enumerate}
\item $\mcf_{\alpha, <\kappa, \aleph_0} = \{ f : \dom(f) \in [\alpha]^{<\theta}$, and if $ \beta \in \dom(f)$ then 
$f(\beta) < \rem_\kappa(\beta) \}$ where $\rem_\kappa(\beta) = \min \{ i : \bigvee_\gamma (\beta = \kappa \cdot \gamma + i ) \}$ 
is the remainder of $\beta$ mod $\kappa$.  

\item $\ba^0_{\alpha, <\kappa, \theta}$ is the Boolean algebra generated by $\{ \mx_f : f \in \mcf_{\alpha, <\kappa, \theta} \}$ 
freely except: 
\begin{enumerate} 
\item $\mx_f \leq \mx_g$ when $g \subseteq f \in \mcf_{\alpha, <\kappa, \theta}$
\item $\mx_f \cap \mx_g = 0$ when $(\exists \beta \in \dom(f) \cap \dom(g)) [ f(\beta) \neq g(\beta) ]$. 
\end{enumerate}

\item Let $\ba_{\alpha, <\kappa, \theta} = \ba^1_{\alpha, <\kappa, \theta}$ be the completion of 
$\ba^0_{\alpha, <\kappa, \theta}$.
\end{enumerate}
\end{defn-star}

We verify that Claim \ref{931fa} likewise holds for $\ba_{\alpha, <\kappa, \theta}$.

\begin{claim-star} \label{fact-ba-k}
 Suppose $\kappa$ is regular and uncountable and 
$\alpha \in \operatorname{Ord}$.  Suppose that in addition $(\forall \alpha < \kappa)(\alpha^{<\theta} < \kappa)$. 
 Then $\ba^0_{\alpha, <\kappa, \aleph_0}$ and $\ba = \ba^1_{\alpha, <\kappa, \theta}$ satisfy the $\kappa$-c.c., and indeed the 
$\kim$-c.c. for any $\bar{m}$ and $\mci$. 
\end{claim-star}

\begin{proof}
Similarly to \ref{931fa},  
suppose we are given $\langle \ma_\alpha : \alpha < \kappa \rangle$ a sequence of positive elements of $\ba$.  
For each  $\alpha < \kappa$, we may choose $f_\alpha \in \mcf_{\alpha, <\kappa, \aleph_0}$,  
so $u_\alpha = \dom(f_\alpha) \in [\alpha]^{<\theta}$, such that 
$ \mx_{f_\alpha} \leq \ma_{\alpha}$.  Assuming that 
$(\forall \alpha < \kappa)(\alpha^{<\nu} < \kappa)$, the hypotheses of the $\Delta$-system lemma \ref{fact-d} apply, and 
there is some $u_*$ and $\vv \in [\kappa]^{\kappa}$ such that $u_\alpha \cap u_\beta = u_*$ 
for $\alpha, \beta \in \vv$.  Let $\gamma = \sup \{ \rem_\kappa(\beta) : \beta \in u_* \}$, so $\gamma < \kappa$ as $|u_*| < \aleph_0$. 
By Definition \ref{newba}(1), each $f_\alpha \rstr u_*$ has range $\subseteq \gamma+1  < \kappa$, so there is $\uu \in [\vv]^\kappa$ such that 
$f_\alpha \rstr u_* = f_\beta \rstr u_*$ for $\alpha, \beta \in \uu$. So just as before, for 
any finite $v \subseteq \uu$, 
\[ \bigcap \{ \ma_\alpha : \alpha \in v \} > 0_\ba \] 
so condition for the $\kim$-c.c. is trivially satisfied for any $n$ and $u$, simply by using $v = u$. 
\end{proof}

Recall from \ref{c:complete-subalgebra} that $\ba_1 \lessdot \ba_2$ means $\ba_1$ is a complete subalgebra of $\ba_2$. 

In our inductive construction, the so-called ``pattern transfer property'' (Definition \ref{paircc} below) will play a key role in ensuring that the 
$(\kappa, \mci, \bar{m})$-c.c. is preserved at limit stages.  In some sense, it shows a close connection between a smaller and larger 
Boolean algebra which is enough to allow the $\kim$-c.c. to transfer from the smaller to the larger one. 
We first need a definition. 

\begin{defn} \label{d:bp}  Let $\ba_1 \lessdot \ba_2$ and let $\mb \in \ba^+_2$. We say that $\ma \in \ba^+_1$ is \emph{below the 
projection of $\mb$}, 
in symbols 
\[ \ma \leq_{\proj(\ba_2, \ba_1)} \mb \]
when for every $\mc \in \ba^+_1$ such that 
$\ba_1 \models \mc \leq \ma$, we have that $\ba_2 \models \mc \cap \mb > 0$.   When $\ba_2, \ba_1$ are clear from context, 
we may omit them from the subscript.
\end{defn}

\begin{obs} \label{o:verify}
If $\ba_1 \lessdot \ba_2 \lessdot \ba_3$, $\ma \in \ba^+_1$, $\mb \in \ba^+_2$, $\mc \in \ba^+_3$, 
\[ \left( \ma ~~ \leq_{\proj(\ba_2, \ba_1)} ~~ \mb  ~~ \leq_{\proj(\ba_3, \ba_2)} ~~ \mc \right) 
~~ \implies ~~\left( \ma ~~ \leq_{\proj(\ba_3, \ba_1)} ~~ \mc \right).  \]
\end{obs}

\begin{proof}
Given $\mc_1 \in \ba^+_1$ such that $\ba_1 \models \mc_1 \leq \ma$, the first inequality ensures that 
$\ba_2 \models \mc_1 \cap \mb > 0$. Let $\mc_2 := \mc_1 \cap \mb \in \ba^+_2$, so $\ba_2 \models 
\mc_2 \leq \mb$, and the second inequality ensures that $\ba_3 \models \mc_2 \cap \mc > 0$, thus 
(since $0 < \mc_2 \leq \mc_1$) $\ba_3 \models \mc_1 \cap \mc > 0$. 
\end{proof}

\begin{defn}[Pattern transfer property] \label{paircc}  Let $\kappa$ be an uncountable cardinal, $\mci$ an ideal on $\omega$ extending 
$[\omega]^{<\aleph_0}$, and $\bar{m}$ a fast sequence. 
 The pair $(\ba_1, \ba_2)$ has the $(\kappa, \mci, \bar{m})$-pattern transfer property when: 
 $\mathbf{(1)}$ $\ba_1$ and $\ba_2$ are both complete Boolean algebras, 
$\mathbf{(2)}$ $\ba_1$ satisfies the $\kappa$-c.c.\footnote{In this definition, we don't ask that $\ba_1$ have the $\kim$-c.c., only the 
$\kappa$-c.c., though in every application in the paper, $\ba_1$ will have the $\kim$-c.c.}, 
$\mathbf{(3)}$ $\ba_1 \lessdot \ba_2$, 
and 
$\mathbf{(4)}$ whenever $\uu_2 \in [\kappa]^\kappa$ and $\bar{\mb} = \langle \mb_\alpha : \alpha \in \uu_2 \rangle$ 
is a sequence of 
elements of $\ba^+_2$, we can find a quadruple 
$(j, \uu_1, A, \bar{\ma} )$ such that: 

\begin{enumerate}
\item[(a)] $j<\omega$
\item[(b)] $\uu_1 \in [\uu_2]^\kappa$
\item[(c)] $A \in \mci$
\item[(d)] $\bar{\ma} = \langle \ma_\alpha : \alpha \in \uu_1 \rangle$ is a sequence of distinct elements of $\ba^+_1$ 
\item[(e)] $\alpha \in \uu_1$ implies $\ma_\alpha \bp \mb_\alpha$
\item[(f)] (i) implies (ii) where:
\begin{enumerate}
\item[(i)] we are given $n \in \omega \setminus A$, $i + j < n$, $u \subseteq \uu_1$, and $\ma_* \in \ba^+_1$ such that 
$m_n / (m^\circ_n)^{n^i} < |u| < m_n$ and 
\[ \ba_1 \models \ma_* \leq \bigcap_{\alpha \in u} \ma_\alpha \]
\item[(ii)] there is $v$ such that $v\subseteq u$ and $|v| \geq |u| / (m^\circ_n)^{n^{i+j}}$ {and} 
\[ \ba_2 \models \bigcap_{\alpha \in v} \mb_\alpha ~ \cap ~ \ma_* > 0.  \]
\end{enumerate}
\end{enumerate}
\end{defn}

\begin{rmk}
One reason $\ref{paircc}(4)(f)$ is not trivial is that $\ma^1_\alpha \bp \ma^2_\alpha$ does not imply $\ma^1_\alpha \leq \ma^2_\alpha$. 
\end{rmk}

\begin{obs}[The pattern transfer property is transitive]  \label{pt-trans} 
Suppose \[ \ba_1 \lessdot \ba_2 \lessdot \ba_3 \] 
are complete Boolean algebras and the pairs $(\ba_1, \ba_2)$ and $(\ba_2, \ba_3)$ have the 
$\kim$-pattern transfer property. Then so does $(\ba_1, \ba_3)$. 
\end{obs}

\begin{proof} 
We start with $\uu_3 \in [\kappa]^\kappa$ and $\bar{\mc} = \langle \mc_\alpha : \alpha \in \uu_3 \rangle$ in $\ba^+_3$.
Applying the pattern transfer property for $(\ba_2, \ba_3)$ to $\bar{\mc}$, we find 
$j_2$, $\uu_2$, $A_2$, and a sequence $\bar{\mb} = \langle \mb_\alpha : \alpha \in  \uu_2 \rangle$ of elements of $\ba^+_2$. 
Next, applying the pattern transfer property for $(\ba_1, \ba_2)$ to  $\bar{\mb}$ 
[which, note, is indexed by $\uu_2$, so we will get $\uu_1 \subseteq \uu_2$] 
we find $j_1$, $\uu_1$, $A_1$, and a sequence $\langle \ma_\alpha : \alpha \in \uu_1 \rangle$ of elements of $\ba^+_1$. 
In short, 
\begin{align*} 
\langle \mc_\alpha ~: ~&  \alpha \in \uu_3 \rangle  \mbox{ from $\ba^+_3$}  \\
& \downarrow \mbox{ \small{via $(j_2, \uu_2, A_2)$}} \\
\langle \mb_\alpha ~:~ &  \alpha \in \uu_2 \rangle \mbox{ from $\ba^+_2$}  \\
& \downarrow \mbox{ \small{via $(j_1, \uu_1, A_1)$}} \\
\langle \ma_\alpha ~:~ &  \alpha \in \uu_1 \rangle \mbox{ from $\ba^+_1$}. 
\end{align*}
Let $j = j_1 + j_2 + 1$ and $A = A_1 \cup A_2 \cup \{ 0, 1 \}$, and we shall verify that 
\begin{align*} 
\langle \mc_\alpha ~: ~&  \alpha \in \uu_3 \rangle  \mbox{ from $\ba^+_3$}  \\
& \downarrow \mbox{ \small{via $(j, \uu_1, A)$} } \\
\langle \ma_\alpha ~:~ &  \alpha \in \uu_1 \rangle \mbox{ from $\ba^+_1$}  
\end{align*}
The point to check is \ref{paircc}(4). Conditions (a)-(d) are clear, condition (e) is by \ref{o:verify}, so it remains to verify 
(f)(i)$\rightarrow$(ii).
Suppose we are given $n$, $i$, $u$, $\ma_*$ such that 
 $n \in \omega \setminus A$, $i + j <n$, $u \subseteq \uu_1$, $\ma_* \in \ba^+_1$,
 $m_n / (m^\circ_n)^{n^i} < |u| < m_n$, and $\ba_1 \models \ma_* \leq \bigcap_{\alpha \in u} \ma_\alpha$. 
Since $i+j_1 < i+j <n$, we may apply the pattern transfer from $\ba_1$ to $\ba_2$ to this data to find 
$v^\prime \subseteq u$ such that 
$|v^\prime| \geq |u|/(m^\circ_n)^{n^{i+j_1}}$ and $\ba_2 \models \bigcap_{\alpha \in v^\prime} \mb_\alpha \cap \ma_* > 0$. 
Choose $a_{**}$ to be any element of $\ba^+_2$ below $\bigcap_{\alpha \in v^\prime} \mb_\alpha \cap \ma_*$. 
Note that $|v^\prime| \geq  |u| / (m^\circ_n)^{n^{i+j_2}}$ and $|u| > m_n / (m^\circ_n)^{n^i}$ together imply 
\[ |v^\prime| > \frac{m_n}{ (m^\circ_n)^{n^i} (m^\circ_n)^{n^{i+j_2}}} \geq \frac{m_n}{(m^\circ_n)^{n^{i+j_1+1}}} ~~.  \]
So we may apply the pattern transfer from $\ba_2$ to $\ba_3$ using the same $n$ and using $i+j_1+1$, $v^\prime$, and 
$\ma_{**}$ 
in place of $i$, $u$, $\ma_*$ respectively in (f)(i).    [Note that we have $(i+j_1+1) + j_2 \leq i+j <n$ by our choice of $j$.] 
We obtain from (f)(ii) a subset $v \subseteq v^\prime$ [thus, $\subseteq u$] with 
$|v| \geq |v^\prime| / (m^\circ_n)^{n^{(i+j_1+1)+j_2}}$  such that 
$\ba_3 \models \bigcap_{\alpha \in v} \mc_\alpha ~ \cap ~ \ma_{**}> 0$, and since $\ma_{**} \leq \ma_*$, 
$\ba_3 \models \bigcap_{\alpha \in v} \mc_\alpha ~ \cap ~ \ma_* > 0$. This shows that for $\bar{\mc}$ the quadruple 
$(j = j_1+j_2+1, \uu, A, \bar{\ma})$ works, which completes the verification. 
\end{proof}

\br 
\begin{claim}[The pattern transfer property indeed transfers our chain condition] \label{o:upgrade}
Suppose $\ba_1 \lessdot \ba_2$, $\ba_1$ has the $\kim$-c.c., and $(\ba_1, \ba_2)$ has the $\kim$-pattern transfer property. Then 
$\ba_2$ has the $\kim$-c.c. 
\end{claim}

\begin{proof}
Suppose we are given $\langle \ma^2_\alpha : \alpha \in \uu_2 \rangle$, $\uu_2 \in [\kappa]^\kappa$ in $\ba_2$.  
Applying the pattern transfer property, we find 
$j_1, \uu_1, A_1$, and $\langle \ma^1_\alpha : \alpha \in \uu_1 \rangle$ in $\ba^+_1$ satisfying \ref{paircc}. Next, working in 
$\ba_1$, we apply the $\kim$-c.c. to the sequence $\langle \ma^1_\alpha : \alpha < \kappa \rangle$ and find 
$j_0, \uu_0, A_0$ satisfying \ref{d1:cca}.    Now let us check that $\ba_2$ satisfies the $\kim$-c.c. using 
$\uu = \uu_0 \subseteq \uu_1 \subseteq \uu_2$, $A = A_1 \cup A_0$, and $j = j_1 + j_0 + 2$. 
Suppose we are given $n \in \omega \setminus A$ and a finite $u \subseteq \uu$.   We know that if $i+j<n$ and 
\[ \frac{m_n}{(m^\circ_n)^{n^i}} < |u| \leq m_n \]
there is some $v^\prime \subseteq u$ such that in $\ba_1$, 
\[ |v^\prime| \geq \frac{|u|}{(m^\circ_n)^{n^{i+j_1}}}
\geq \frac{m_n}{(m^\circ_n)^{ n^{ i+j_1+1 } } } ~~\mbox{ ~~and~~ }~~ \ba_1 \models \bigcap \{ \ma^1_\alpha : \alpha \in v^\prime \}  > 0. \] 
Now \ref{paircc}(f)(i) holds, using $v^\prime$, $i+j_1 +1$ here for $u$, $i$ there. 
(Note: the identity of $\ma_*$ in that equation is not important here; we'll use it later when we deal with omitting types.) 
So \ref{paircc}(f)(ii) tells us there is $v \subseteq v^\prime$ such that 
$|v| \geq |v^\prime| / (m^\circ_n)^{n^{i+j_1+j_0+1}}$ {and} 
$\ba_2 \models  \bigcap_{\alpha \in v} \ma^2_\alpha > 0$ . 
Checking the size, 
\[ |v| \geq  \frac{|v^\prime|}{(m^\circ_n)^{n^{i+j_1+j_0+1}}} \geq \frac{|u|}{(m^\circ_n)^{n^{i+j_1}} (m^\circ_n)^{n^{i+j_1+j_0+1}} } 
\geq \frac{|u|}{(m^\circ_n)^{n^{i+j_1 + j_0 + 2}}}. \]
This completes the verification. 
\end{proof}

We recall the following fact which we will upgrade to the case of our chain condition in \ref{9.14a}. 

\begin{fact}[Jech, \cite{Jech} Corollary 16.10]
Let $\ba_0 \subseteq \ba_1 \subseteq \cdots \subseteq \ba_\beta \subseteq \cdots$ $(\beta < \alpha)$ be a sequence 
of complete Boolean algebras such that for all $\beta < \gamma$, $\ba_\beta \lessdot \ba_\gamma$, and for each 
limit ordinal $\gamma$, $\bigcup_{\beta < \gamma} \ba_\beta$ is dense in $\ba_\gamma$. If every $\ba_\beta$ 
satisfies the $\kappa$-c.c. then $\bigcup_{\beta < \alpha} \ba_\beta$ satisfies the $\kappa$-c.c.  
\end{fact}

\begin{defn} \label{lessdot-seq}
We say that $\bar{\ba} = \langle \ba_\gamma : \gamma \leq \delta \rangle$ is a $\lessdot$-increasing continuous sequence of complete Boolean algebras 
when:
\begin{enumerate}
\item for $\gamma < \delta$, $\ba_\gamma \lessdot \ba_{\gamma +1}$; 
\item at limit stages $\bar{\ba}$ is continuous, meaning that if $\gamma < \delta$ is a limit ordinal then $\bigcup \{ \ba_\gamma : \gamma < \delta \}$ is a 
dense subset of $\ba_\delta$. $($Requiring that the union is $\lessdot~ \ba_\delta$ is ok also.$)$ 
\end{enumerate}
\end{defn}


Our next claim \ref{9.14a} will carry the induction for Lemma \ref{lem15}. 

\setcounter{equation}{0}
\begin{claim} \label{9.14a} 
 Suppose that $\bar{\ba} = \langle \ba_\gamma : \gamma \leq \delta \rangle$ is a $\lessdot$-increasing continuous sequence 
of complete Boolean algebras, where: 
\begin{enumerate}
\item $\ba_\gamma$ satisfies the $(\kappa, \mci, \bar{m})$-c.c. for every $\gamma < \delta$.
\item if $\gamma < \delta$ then the pair $(\ba_\gamma, \ba_{\gamma+1})$  satisfies the 
$(\kappa, \mci, \bar{m})$-pattern transfer property. 
\item if $\gamma < \beta < \delta$ then the pair $(\ba_\gamma, \ba_\beta)$  satisfies the 
$(\kappa, \mci, \bar{m})$-pattern transfer property. 
\end{enumerate}
Then $\ba_\delta$ satisfies the $(\kappa, \mci, \bar{m})$-c.c., and moreover, 
if $\gamma < \delta$ then the pair $(\ba_\gamma, \ba_\delta)$-satisfies the $(\kappa, \mci, \bar{m})$-pattern transfer property.
\end{claim}

\begin{proof}  
The case to consider is $\delta$ a nonzero limit.  
Fix $\gamma_* < \delta$, and we will show that $(\ba_{\gamma_*}, \ba_\delta)$ has the $\kim$-transfer property.
We may then conclude by \ref{o:upgrade} that $\ba_\delta$ has the $\kim$-c.c.  

Suppose we are given $\langle \ma^\delta_\alpha : \alpha < \kappa \rangle$ a sequence of elements of $\ba^+_\delta$. 
Without loss of generality (by the definition of continuous) each $\ma^\delta_\alpha$ is an element of $\bigcup_{\gamma < \delta} \ba_\gamma$
(i.e., each element is already a member of some $\ba_\gamma$).

Recall $\kappa$ is regular. 
If $\cf(\delta) \neq \kappa$, there is $\gamma_{**} < \delta$ such that 
$\uu = \{ \alpha < \kappa : \ma^\delta_\alpha \in \ba_{\gamma_{**}} \}$ has size $\kappa$.  
Since $\ba_{\gamma_{**}}$ has the $\kim$-c.c. by hypothesis,  $\ba_\delta$ will inherit it from the restriction to $\uu$. 
Since without loss of generality $\gamma_{**} > \gamma_*$, and 
$(\ba_{\gamma_*}, \ba_{\gamma_{**}})$ has the $\kim$-pattern transfer property by hypothesis, we can similarly see that  
$(\ba_{\gamma_*}, \ba_\delta)$ has the $\kim$-pattern transfer property too (by restricting to $\uu$). 

Suppose then that $\cf(\delta) = \kappa$.  In this case fix a strictly  
increasing continuous sequence 
$\langle i_\alpha : \alpha < \kappa \rangle$ of ordinals $<\delta$ but above $\gamma_*$, 
with limit $\delta$.  
First we choose a sequence of elements $\langle \mb_\alpha : \alpha < \kappa \rangle$, as follows. 
There is $\zeta(\alpha) \in (\alpha, \kappa)$ such that 
\begin{equation} \label{e:zeta} 
\ma^\delta_{\alpha} \in \ba_{i_{  \zeta(\alpha)   }} ~.   
\end{equation}
As $\alpha < \zeta(\alpha)$, 
by our assumption 
$ \ba_{i_\alpha} \lessdot \ba_{i_{\zeta(\alpha)}}$. 
As $\ma^\delta_\alpha \in \ba_{i_{\zeta(\alpha)}}$, by definition of $\lessdot$, 
we may find a 
$\mb_\alpha \in \ba^+_{i_\alpha}$ such that 
\begin{equation}
\label{b:lproj}
\mb_\alpha ~ 
{\leq_{\operatorname{proj}(\ba_{i_{\zeta(\alpha)}}, \ba_{i_\alpha})}}
~ \ma^\delta_\alpha.
\end{equation}
Let 
$\langle \mb_\alpha : \alpha < \kappa \rangle$ 
be the sequence of nonzero elements defined in this way. 

Next we choose a sequence of elements $\langle \mc_\epsilon : \epsilon < \kappa $ is a limit$\rangle$, as follows. 
For every limit ordinal $\epsilon < \kappa$,  recall that $\bigcup_{\alpha < \epsilon} \ba_{i_\alpha}$ is dense in $\ba_{i_\epsilon}$ by 
definition of continuous.  Recall that $\mb_\epsilon \in \ba_{i_\epsilon}$ by construction (\ref{b:lproj}). 
So for every limit ordinal $\epsilon < \kappa$, there is $\mc_\epsilon$ such that 
$\mc_\epsilon \in \bigcup_{\alpha < \epsilon} \ba_{i_\alpha}$, 
$\ba_{i_\epsilon} \models 0 < \mc_\epsilon < \mb_\epsilon$, 
and (hence) $\mc_\epsilon \in \ba^+_{i_{\rho(\epsilon)}}$  for some $\rho(\epsilon) < \epsilon$.  
Now the function $\epsilon \mapsto \rho(\epsilon)$ is defined and regressive on the limit ordinals $<\kappa$, 
so by Fodor's lemma there is 
$\rho_* < \kappa$ such that $ \uu_0 = \{ \epsilon < \kappa : \epsilon \mbox{ is a limit $> \rho_*$, and }\rho(\epsilon) < \rho_* \}$ 
is a stationary subset of $\kappa$, 
recalling that $\kappa = \cf(\kappa) > \aleph_0$ by \ref{d1:cca}.  
Recalling $\zeta$ from (\ref{e:zeta}), 
let $E$ be a club of $\kappa$ such that $\epsilon < \beta \in E$ implies $\zeta(\epsilon) < \beta$. Then 
\begin{equation}
\uu = E \cap \uu_0
\end{equation} 
is a stationary subset of $\kappa$, hence of size $\kappa$. 
[For orientation: for $\epsilon \in E$, the $\mc_\epsilon$'s all belong to $\ba_{i_{\rho_*}}$, while  
$\mb_\epsilon \in \ba_{i_\epsilon}$, $\ma^\delta_\epsilon \in \ba_{i_{\zeta(\epsilon)}}$, 
and recall $\rho_* < \epsilon < {\zeta(\epsilon)} < \kappa$, while $i_{\rho_*} < i_\epsilon < i_{\zeta(\epsilon)} < i_\kappa = \delta$.  
Moreover, if 
$\gamma < \epsilon$ are from $\uu$, then $\ba_{i_{\zeta(\epsilon)}}$ 
already contains $\mb_\gamma$ and $\ma^\delta_{\gamma}$.]

We will now show the following relationship between $\langle \mc_\epsilon :  \epsilon \in \uu \rangle$ and 
$\langle \ma^\delta_\epsilon : \epsilon \in \uu \rangle$:   
for any finite $u \subseteq \uu$, 
\begin{equation} \label{eq:int}
\ba_{i_{\rho_*}} \models \bigcap_{\alpha \in u} \mc_\alpha > 0  ~ \implies ~ \ba_\delta \models \bigcap_{\alpha \in u} \ma^\delta_\alpha > 0. 
\end{equation}
Let $n = |u|$, and let $\epsilon(0) < \cdots < \epsilon(n-1)$ list $u$. By induction on $m \leq n$, let us show that 
\[ (\star)_m  \mbox{  if $\ba_{i_{\rho_*}} \models \bigcap_{\ell < m} \mc_{\epsilon(\ell)} \geq \md > 0$  [so $\md \in \ba_{i_{\rho_*}}$] then 
$\ba_\delta \models \bigcap_{\ell < m} \ma^\delta_{\epsilon(\ell)} \cap \md > 0$. } \]
Clearly for $m=n$ this will imply equation (\ref{eq:int}).  Note that for $m=0$, the intersections $\bigcap_{\ell < 0} \mc_{\epsilon(\ell)}$ 
and $\bigcap_{\ell < 0} \ma^\delta_{\epsilon(\ell)}$ are $1_{\ba_{\rho_*}}$ and $1_{\ba_{\delta}}$ respectively.   For $m = 1$, we 
are assuming  
$\ba_{i_{\rho_*}} \models 0 < \md \leq \mc_{\epsilon(0)}$ thus 
$\ba_{i_{\epsilon(0)}} \models 0 < \md \leq \mc_{\epsilon(0)} \leq \mb_{\epsilon(0)}$, so by equation (\ref{b:lproj}),  
$\ba_{i_{\zeta(\epsilon(0))}} \models 0 < \md \cap \ma^\delta_{\epsilon(0)}$, and thus 
$\ba_\delta \models  0 < \md \cap \ma^\delta_{\epsilon(0)}$, proving $(\star)_1$. 
For the inductive step, assume $m < n$ and that $(\star)_m$ holds and we shall prove $(\star)_{m+1}$. 

For $(\star)_{m+1}$, 
we are assuming $\ba_{i_{\rho_*}} \models \bigcap_{\ell \leq m} \mc_{\epsilon(\ell)} \geq \md > 0$.  By inductive hypothesis, 
$\ba_\delta \models \md^\prime = \bigcap_{\ell < m} \ma^\delta_{\epsilon(\ell)} \cap \md > 0$. 
Now $\{ \ma^\delta_{\epsilon(\ell)} : \ell < m \} \cup \{ \md \} \subseteq \bigcup \{ \ba_{i_{\zeta(\epsilon(\ell))}} : \ell < m \} 
\cup \ba_{i_{\rho_*}} \subseteq \ba_{i_{\epsilon(m)}}$, where the last inclusion is by the definition of $E$, recalling that  
$\bar{\ba}$ is increasing.  Thus, 
$\ba_{i_{\epsilon(m)}} \models \md^\prime = \bigcap_{\ell < m} \ma^\delta_{\epsilon(\ell)} \cap \md > 0$, and at the same time 
(recall the phrase after ``For $(\star)_{m+1}$'')  we have 
$\ba_{i_{\epsilon(m)}} \models 0 < \md^\prime \leq \mc_{\epsilon(m)}$.  Moreover, since $\mb_{\epsilon(m)} \in \ba_{i_{\epsilon(m)}}$, 
we have $\ba_{i_{\epsilon(m)}} \models 0 < \md^\prime \leq \mc_{\epsilon(m)} \leq \mb_{\epsilon(m)}$.   Then by equation 
(\ref{b:lproj}), $\ba_{i_{\zeta(\epsilon(m))}} \models 0 < \md^\prime \cap \ma^\delta_{\epsilon(m)}$, and thus 
$\ba_\delta \models  0 < \md^\prime \cap \ma^\delta_{\epsilon(m)}$. But then the choice of $\md^\prime$ implies that  
$\ba_\delta \models \bigcap_{\ell\leq m} \ma^\delta_{\epsilon(m)} \cap \md > 0$, 
proving $(\star)_{m+1}$. 
This proves (\ref{eq:int}). 

Clearly, this allows us to transfer the $\kim$-c.c. from $\ba_{i_{\rho_*}}$ to $\ba_\delta$, and 
it also tells us that $(\ba_{i_{\rho_*}}, \ba_\delta)$ has the $\kim$-pattern transfer property.  By construction 
$\gamma_* < i_{\rho_*}$, 
and so by hypothesis $(\ba_{\gamma_*}, \ba_{i_{\rho_*}})$ has the $\kim$-pattern transfer property, so recalling \ref{pt-trans}, 
$(\ba_{\gamma_*}, \ba_\delta)$ has the $\kim$-pattern transfer property too, which completes the proof. 
\end{proof}

\begin{lemma} \label{lem15}
If $\langle \ba_\gamma : \gamma \leq \delta \rangle$ is $\lessdot$-increasing continuous and $\ba_0$ satisfies the $\kim$-c.c. and 
$(\ba_{\gamma}, \ba_{\gamma+1})$ satisfies the $\kim$-pattern transfer property for all $\gamma < \delta$, then 
$\ba_\delta$ satisfies the $\kim$-c.c. and for all $\gamma <\delta$, the pair $(\ba_\gamma, \ba_\delta)$ satisfies the 
$\kim$-pattern transfer property. 
\end{lemma}

\begin{proof}
By induction on $\delta$; immediate from the previous claim and \ref{pt-trans}, \ref{o:upgrade}.
\end{proof}

We shall use the property from the proof of \ref{9.14a} later on so we phrase it below. 
It is stronger than \ref{paircc}, so does not 
supercede that definition, but as we have seen it will imply it, and 
occasionally it will be simpler to show. 

\begin{cor} \label{9.14f}  $\ba_2$ satisfies the $(\kappa, \mci, \bar{m})$-c.c., and even the pair 
$(\ba_1, \ba_2)$ satisfies the $(\kappa, \mci, \bar{m})$-pattern transfer property, when: 
\begin{enumerate}
\item[(a)] $\ba_1 \lessdot \ba_2$ are complete Boolean algebras 

\item[(b)] $\ba_1$ satisfies the $(\kappa, \mci, \bar{m})$-c.c.

\item[(c)] given $\ma_\alpha \in \ba^+_2$ for $\alpha < \kappa$, 
 we can find $\uu \in [\kappa]^\kappa$ and $\mx_\alpha \in \ba^+_1$ 
for $\alpha \in \uu$ such that: if $u \in [\uu]^{<\aleph_0}$ and $\ba_1 \models$ ``~$\bigcap_{\alpha \in u} \mx_\alpha > 0$'' \underline{then} 
$\ba_2 \models $ ``~$\bigcap_{\alpha \in u} \ma_\alpha > 0$.''
\end{enumerate}
\end{cor}

\begin{rmk}  Note that 
$\ref{9.14f}$ is superficial: 
we require no a priori connection between the $\ma_\alpha$'s and the $\mx_\alpha$'s 
other than a common enumeration. Occasionally, as for the random graph, this is enough. 
Still, it would be natural to add $\mx_\alpha \bp \ma_\alpha$.  
\end{rmk}

\vspace{5mm}

\section{The c.c. and omitting types}

The main work of this section (Lemma \ref{9.14aa}) is to directly connect the chain condition from \ref{d1:cca} to omitting types. 

\begin{conv} \label{cc-conv}
In this section we fix:
\begin{enumerate}
\item ${{\mu}}, \kappa, \lambda$ infinite cardinals, with $\kappa$ regular, and 
  ${{\mu}} < \kappa \leq \lambda$.  
\item $\ba_* = \ba^1_{\kappa, \mu, \aleph_0}$.  
\item $\bar{m}$ a fast sequence. 
\item $\bar{E}$ a sequence of graphs which is good for $\bar{m}$.
\item $\mci$ an ideal on $\omega$ extending $[\omega]^{<\aleph_0}$. 
\end{enumerate}
\end{conv}

\begin{disc} \label{disc-card}
In terms of the three cardinals in $\ref{cc-conv}(1)$: $\mu$ is the size of maximal antichains $($partitions$)$ in the completion of the free 
Boolean algebra $\ba_*$ we shall study in this section and use to begin our inductive construction in the next section; 
$\kappa$ is the cardinal in our chain condition 
$\ref{d1:cca}$. The cardinal $\lambda$ will ultimately be the level of saturation we are aiming for:
the ultrafilter we build in later sections will realize all types over sets of size $\leq \lambda$ for theories it can handle.
Lemma $\ref{9.14aa}$ will say something a priori stronger regarding non-saturation, however, 
since it will tell us that in certain theories, we already omit types over sets of size $\kappa$. 
\end{disc}

\begin{disc}
Although we allow the generality of $\mu < \kappa$ with no constraints on their distance, for our main results on incomparability 
it would suffice to use $\mu = \aleph_0$ and $\kappa = \lambda$ an uncountable successor, e.g. $\aleph_1$.  
In other words, 
to see the 
differences between these theories we do not need to go out very far; 
however, 
it is not a phenomenon limited to small sizes, 
due to the freedom in the construction.  
\end{disc}

Historical note: the prototype for this lemma is \cite{MiSh:1140} Lemma 3.2, which 
amounts to showing non-saturation in our base case, for $\ba = \ba_*$ and any nontrivial ultrafilter $\de_*$ on $\ba$.  
We will use notation like $\ma[~\psi[\bar{a}]~]$ from \ref{n:sv} above. 

\begin{lemma}  \label{9.14aa} 
Suppose $\ba$ is a complete Boolean algebra, $\ba_* \lessdot \ba$, and $(\ba_*, \ba)$ has the $\kim$-pattern transfer property.  
Let $\xi$ be any level function such that $\xi^{-1}\{1\} \neq \emptyset \mod \mci$.
Let $T = T_\xn$ where $\xn = \prm[\bar{m}, \bar{E}, \xi]$. 
Let $\de$ be any ultrafilter on $\ba$. 
Then there is a possibility pattern $\langle \mc_u : u \in [\kappa]^{<\aleph_0} \rangle$ for $T_\xn$ which has no multiplicative refinement. 
\end{lemma}

\begin{rmk}
So, recalling \ref{fact:nonsat}, given any $M \models T_\xn$ and any enveloping ultrapower $N = M^\lambda/\de$ for $\de$ and $\ba$,  $N$ omits a type of $T_\xn$ over a set of size $\kappa$. 
\end{rmk}

\begin{rmk}
The proof of \ref{9.14aa} below will show that if we fix any nonprincipal ultrafilter $\de_*$ on $\ba_*$ in advance, then for any such $T_\xn$, 
 there is a specific possibility pattern in $(\ba_*, \de_*)$ [that is, in the initial Boolean algebra!] 
 which fails to have a multiplicative refinement in 
$(\ba, \de)$ for any later $\ba$ satisfying the hypotheses of the claim, and any ultrafilter 
$\de$ on $\ba$ which extends $\de_*$. 
\end{rmk}

\setcounter{equation}{0}

\begin{proof}[Proof of Lemma \ref{9.14aa}]   
Using the framework of separation of variables, we work in $T_{\xn}$, recalling that 
\begin{equation}
\label{eq00}
 \xn = \prm[\bar{m}, \bar{E}, \xi_{\xn}].    
\end{equation}
and recalling that 
\begin{equation}
\bar{m} = \langle m_k : k <\omega \rangle. 
\end{equation}
For each $\rho \in \mct_{2,k}$  and $\alpha < \kappa$, we define $\ma[~P_\rho(x_\alpha)~] \in \ba^+_*$ by induction on $k<\omega$:

\begin{itemize}
\item if $k = 0$, let $\ma[~P_{\langle~\rangle}(x_\alpha)~] = 1_{\ba_*}$. 

\item for $k > 0$, recalling $\alpha$ is also given, 
we first specify a finite partition of $1_{\ba_*}$ into $m_k$ pieces 
$\langle \mx_{\alpha, k, i} : i < m_k \rangle$ such that every element but one of this sequence is 
a member of the maximal antichain $\{ \mx_{(\omega \alpha + k, j)} : j < \mu \}$ (so necessarily the 
remaining element is the union of the remaining members of the antichain).  
So without loss of generality, we may assume $\mx_{{\alpha, k, 0}} \in \de_*$.  
Then define: for $i < m_k$, 
\[ \ma[~P_{\rho ^\smallfrown \langle i \rangle}(x_\alpha)~] = \ma[~P_\rho(x_\alpha)~] \cap \mx_{{\alpha, k, i}}. \]
\item for $\alpha < \beta < \kappa$, 
\[    \ma[~x_\alpha \neq x_\beta~] = 1_{\ba_*}. \] 
\item Note that $\langle \ma[~P_\rho(x_\alpha)~] : \rho \in \mct_{2,k} \rangle$ is a partition of $1_\ba$. 
\end{itemize}
As $\mx_{{\alpha, k, 0}} \in \de_*$, letting $\langle 0_k \rangle$ denote 
the constant $0$ sequence of length $k$, we have that for each $\alpha < \kappa$ and each $k<\omega$, 
\[ \ma[~P_{\langle 0_k \rangle} (x_\alpha)~] \in \de_*.\]   For each $u \in [\kappa]^{<\aleph_0}$, define 
\[ \mc_u = \ma[~\exists x \bigwedge_{\alpha \in u} R(x,x_\alpha)~]. \]
It follows from the construction that 
\begin{equation} 
\label{eq7} 
\bar{\mc} = \bar{\mc}[\xn] = \langle \mc_u : u \in [\kappa]^{<\aleph_0} \rangle = \langle \mc_{\xn, u} : u \in [\kappa]^{<\aleph_0} \rangle 
\end{equation} is a possibility pattern in $(\ba_*, \de_*)$, hence in $(\ba_{}, \de_{})$ 
[thus, we could choose appropriate parameters $c_\alpha$ 
to fill in for $x_\alpha$ in any 
enveloping ultrapower]. 

\emph{Note to the reader}: if we had run this construction for any other $\xn^\prime \in \sM_*$,  the elements 
$\mc_{\{\alpha\}}$ would be exactly the same (and equal to $1_\ba$); but the sets $\mc_u$ for $|u| >1$ could 
differ depending on $\xn^\prime$.  So although this construction will work for any parameter in $\sM_*$, it does 
not necessarily give the same $\omc$.  

Assume for a contradiction that $\bar{\ma}^2 = \langle \ma^2_\alpha : \alpha < \kappa \rangle$
is a multiplicative refinement of $\bar{\mc}$ in $\ba_{}$, where $\bar{\ma}^2$  
is a sequence of members of $\ba^+_{}$. 
We apply the definition of ``$(\ba_*, \ba_{})$ satisfies the $(\kappa, \mci, \bar{m})$-pattern transfer property,''  hence there is a 
quadruple $(j, \uu_0, A, \bar{\ma}^1)$ as there, noting $\bar{\ma}^1$ is a sequence of elements of 
$\ba^+_*$. 

Now by the choice of $\ba_*$, for each $\alpha < \kappa$ there is $f_\alpha \in \fin_{\mu,\aleph_0}(\kappa)$, i.e. $f_\alpha$ is a finite function from $\kappa$ to ${{\mu}}$, such that 
\[ \ba_* \models  \mx_{f_\alpha} \leq \ma^1_\alpha \mbox{ ~~ [hence this holds also in $\ba_{}$ ]. } \]
Since each $f_\alpha$ has finite domain, there is a 
smallest positive integer $k_\alpha$ such that for every $\beta \in \dom(f_\alpha)$, the remainder
of $\beta \mod \omega$ is $< k_\alpha$.  
So there is $\uu_1 \in [\uu_0]^\kappa$ and $n \in \omega \setminus A$ so that $k_\alpha = k_\beta < n$ for $\alpha, \beta \in \uu_1$.
Without loss of generality $j<n$ 
and $\xi_\xn(n) = 1$, possible as $\xi^{-1}_\xn \{ 1 \} \neq \emptyset \mod \mci$, while $A \in \mci$. 
For each $\alpha \in \uu_1$, the elements $\{ \ma[P_\nu(x_\alpha)] : \nu \in \mct_{2,n} \}$ form a finite, maximal antichain 
of $\ba_*$.  Let us justify that 
for some $\uu_2 \in [\uu_1]^\kappa$ and some $\nu_* \in \mct_{2,n}$, for every $\alpha \in \uu_2$, we may extend 
$f_\alpha$ to a possibly larger finite function $f^*_\alpha$ such that 
\begin{equation}
\label{eq182}
 0 < \mx_{f^*_\alpha} \leq \ma[P_{\nu_*}(x_\alpha)]~~ \mbox{ and moreover } ~~ \mx_{f^*_\alpha} \leq \mx_{f_\alpha}. 
\end{equation}
[Why? First, for each given $\alpha$, we define $f^*_\alpha \supseteq f_\alpha$ so that for some $\nu_\alpha \in \mct_{2,n}$, 
$0 < \mx_{f^*_\alpha} \leq \ma[P_{\nu_\alpha}(x_\alpha)]$.  We do this by defining 
$f^k_\alpha$ by induction on $k<n$. Let $f^{-1}_\alpha = f_\alpha$. 
For $k\geq 0$, remember our finite partition $\langle \mx_{\alpha, k, i} : i < m_i \rangle$ from the beginning of the proof. 
We want to choose $f^k_\alpha \supseteq f^{k-1}_\alpha$ to remain a function and so that 
$\mx_{f^k_\alpha}$ is below one of the elements of this finite partition. 
Remember that all but one element of this partition was of the form $\mx_{(\omega \alpha + k, j)}$, and the 
remaining element, call it it the ``overflow element,'' was the union of the remaining elements of the 
antichain $\{ \mx_{(\omega \alpha + k, j)} : j < \mu \}$.  There are two cases. If $\mx_{f^{k-1}_\alpha}$ is consistent 
with at least one of the elements of the partition of the form $\mx_{(\omega \alpha + k, j)}$, then choose one and define 
$f^k_\alpha = f^{k-1}_\alpha \cup \{ (\omega \alpha + k, j) \}$.
If not, it must already be the case that $\mx_{f^{k-1}_\alpha}$ is below the overflow element, so define 
$f^k_\alpha = f^{k-1}_\alpha$. This completes the induction. 
Let $f^*_\alpha := f^{n-1}_\alpha$. 
The choice of $f^*_\alpha$ determines a unique $\nu_\alpha \in \mct_{2,n}$ so that
$ 0 < \mx_{f^*_\alpha} \leq \ma[P_{\nu_\alpha}(x_\alpha)]$.  Since, again, $\mct_{2,n}$ is finite, 
we may choose $\uu_2$ to be a set of size $\kappa$ 
on which $\nu_\alpha$ is constant, call it $\nu_*$.  This completes the justification of (\ref{eq182}).]

Note that for every $\beta \in \dom(f^*_\alpha)$ the remainder of $\beta \mod \omega$ is still $< n$, 
and it is still the case that $\dom(f^*_\alpha)$ is finite. 

For each $\alpha < \kappa$, let $u_\alpha = \dom(f^*_\alpha)$. 
By the $\Delta$-system lemma \ref{delta-system}, there is some $u_*$ and $\uu_3 \in [\uu_2]^{\kappa}$ such that $u_\alpha \cap u_\beta = u_*$ 
for $\alpha, \beta \in \uu_3$. Since the range of each $f^*_\alpha$ is a finite subset of ${{\mu}}$ and ${{\mu}}^+ \leq \kappa$, there is 
$\uu_4 \in [\uu_3]^\kappa$ such that 
$f^*_\alpha \rstr u_* = f^*_\beta \rstr u_*$ for $\alpha, \beta \in \uu_4$. Notice this tells us for any finitely many 
$\alpha_0, \dots, \alpha_{n-1}$ from $\uu_4$, $f = \cup_{i<n} f^*_i$ is a function thus $\mx_f > 0$.  

To summarize, for any finite $u \subseteq \uu_4$, we have that in $\ba_*$, 
\begin{equation}
\label{eq11}
0 < \left( \bigcap_{\alpha \in u}  \mx_{f^*_{\alpha}} \right) \leq \bigcap_{\alpha \in u} \ma[P_{\nu_*}(x_{\alpha})]. 
\end{equation}
Next,  note that for every $\alpha \in \uu_4$, $\dom(f^*_\alpha) \cap \{ \omega \alpha + n \} = \emptyset$, 
by the remark after equation (\ref{eq182}). 
It follows that by our definition of $\omc$, 
 for any $\ell < m_{n }$ [recalling $\bar{m}$ from 
(\ref{eq00})]  
and any $\alpha \in \uu_4$, we have that in $\ba_*$, 
\begin{equation}
\label{eq12} \mx_{f^*_\alpha} \cap 
\ma[P_{{\nu_*}^\smallfrown \langle \ell \rangle}(x_{\alpha})] > 0.
\end{equation}

\noindent Recall that we chose $n$ so that $\xi_{\xn}(n) = 1$.
Let $w \subseteq \uu_4$ be such that $|w| = m_{n} - 1$. 
By equation (\ref{eq11}), 
\[ \my_0 := \bigcap_{\alpha \in w} \mx_{f^*_{\alpha}}  > 0 \]
and also recall that for each $\alpha \in w$, 
\[ \my_0 \leq \ma[P_{\nu_*}(x_\alpha)]. \]
Thus, if we enumerate $w$ as $\alpha_0, \dots, \alpha_{m_n-1}$,  then in $\ba_*$ (hence also in $\ba_{}$) 
\begin{equation}
\label{eq13} 
\bigcap_{\ell < m_n}  \mx_{f^*_{\alpha_\ell}} ~~\cap ~ 
\bigcap_{\ell < m_n } \ma[P_{\nu_*}(x_{\alpha_\ell})] ~~\cap ~
\bigcap_{\ell < m_n}  \ma[P_{{\nu_*}^\smallfrown \langle \ell \rangle}(x_{\alpha_\ell})]  ~~ > 0. 
\end{equation} 
Call the quantity on the left side of this inequation ``$\my_1$.''
\noindent Now we use the choice of the quadruple $(j, \uu_0, A, \bar{\ma}^1)$ for our given $w$. 

In particular, we apply clause (4)(f) of \ref{paircc} using $n = n$, $i = 1$, $u = w$, and $\ma = \my_1 \in \ba_*$. 
Then indeed 
\[ m_{n} / (m^\circ_{n}) < |w| < m_{n} \]
and (\ref{eq13}) translates to tell us that $0 < \my_1 \leq \bigcap_{\alpha \in w} \ma^1_\alpha$ in $\ba_*$. 
So (i) of \ref{paircc}(f) holds, and by (ii) of that clause there is $v$ such that 
$v\subseteq w$ and $|v| \geq |w| / (m^\circ_{n})^{{n}^{1+j}}$ \underline{and} 
\[ \ba_{} \models  \bigcap_{\alpha \in v} \ma^2_\alpha ~ \cap ~ \my_1 > 0.  \]
Let us name this intersection: 
\[   \my_2 = \bigcap_{\alpha \in v}  \ma^2_\alpha ~ \cap \my_1 . \]
Recall that $\bar{\ma}^2$  is a solution to $\bar{\mc}$, so by our definition of multiplicative refinement
\begin{equation}
\label{eq14} 
\ba_{} \models ~~ \bigcap_{\alpha \in v} \ma^2_{\alpha} \leq \mc_v
~~~\mbox{ which tells us that }~~~
 \my_2 \leq \mc_v. 
\end{equation}
However, since $\xi_\xn(n) = 1$, the definition of a type in $T_{\xn}$ doesn't allow ``large'' splitting at $n$, so necessarily  
in $\ba_{}$
\begin{equation} \label{eq15}
\mc_{v} ~~ \cap ~ \left( 
\bigcap_{\ell < m_n} \ma[P_{\nu_*}(x_{\alpha_\ell})] ~~\cap ~ 
\bigcap_{\ell < m_n} \ma[P_{{\nu_*}^\smallfrown \langle \ell \rangle}(x_{\alpha_\ell})]  ~\right) = 0. 
\end{equation}
Together (\ref{eq13}), (\ref{eq14}) and (\ref{eq15}) are a contradiction.  This shows that $\bar{\mc}$ has no multiplicative 
refinement. 
\end{proof}

\begin{cor}
Recall from \ref{931fa} that  
$\ba_*$ has the $\kim$-c.c., so by Observation $\ref{o:upgrade}$, 
it will follow that $\ba$ has the $\kim$-c.c. 
\end{cor}

\begin{rmk} Why in the present proof do we not use the weaker choice of the pattern transfer condition \ref{9.14f} above? 
We will see in the next section, and 
in the preservation in the inductive construction. 
\end{rmk}

\begin{disc}  We could apply the first part of the proof to any $\xm \in \sM_*$ and get a corresponding $\bar{\ma}^1$, but when we fix 
$\sM_*$, $\mcm, \mcn$ the cases in $\sM_* \setminus \mcn$ are not useful; we needed the active level for $\xn$ to get 
a contradiction.  In particular for them we will not necessarily 
have the failure of saturation, even though we can define the possibility pattern. 
\end{disc}

\vspace{5mm}

\section{The inductive construction}

Our construction problem naturally has two sides: we want to realize some types while omitting others. 

Working in the framework of separation of variables, 
our strategy will be to build the Boolean algebra and the ultrafilter on it together, 
by induction.  This is a significant advance over earlier approaches, so let us explain what makes it 
possible to carry it out.  
The rough idea will be to start with $\ba_*$ as in \ref{cc-conv}(2),  
and at stage $\alpha$ to extend $\ba_{\alpha}$ and the ultrafilter on it to solve a single problem,  
say, a problem from a theory $T_\xm$, $\xm \in \mcm$, see \ref{conv10}(3).   
The work of the previous section tells us that if we can do this 
while maintaining ``$(\ba_*, \ba_{\alpha+1})$ has the $\kim$-pattern transfer property,'' we will be able to 
preserve non-saturation for theories $T_\xn$, $\xn \in \mcn$. 

How, then, might we add solutions for problems (i.e., add multiplicative refinements for possibility patterns) 
while leaving our chain condition undisturbed? The idea, which seems to have much in it beyond the present proof, 
is that we will add the solutions in as minimal a way as possible.    
A Boolean algebra has generators and equations\footnote{By ``equations'' we mean ``conditions,'' which can include inequalities. Also,  
$\mb^1_\alpha \equiv \mb^1_{\{\alpha\}}$.} they must satisfy. 
Roughly, given at some inductive step 
a possibility pattern $\langle \mb_u : u \in [\lambda]^{<\aleph_0} \rangle$, we shall extend the given Boolean algebra by 
adding new elements $\langle \mb^1_{\alpha} : \alpha \in [\lambda]^{<\aleph_0} \rangle$ to its generating set, to form the 
Boolean algebra generated freely by these elements, subject to any already given equations, along with the conditions 
saying $\mb^1_u := \bigcap_{\alpha \in u} \mb^1_{\alpha} \leq \mb_u$ which ensure $\bar{\mb}^1$ is indeed a multiplicative refinement of $\mb$. 
We then take any ultrafilter on this Boolean algebra which extends our previous ultrafilter and contains the elements of $\bar{\mb}^1$.  
(See \ref{922}.) 

The proof will appeal directly to the theories involved to show that this process of adding ``formal solutions'' 
does indeed preserve our chain conditions, and the desired ultrafilters do exist. Perhaps they can be thought of as 
ultrafilters (and Boolean algebras) \emph{tailor-made} for the theories at hand. There will be three main cases: 
problems coming from types over sets of size $<\kappa$ in any theory (\ref{10.12A}), problems coming from types over 
sets of size $\lambda$ in models of $T_\xm$ for $\xm \in \mcm$ (\ref{9.28a}) and 
problems coming from types over sets of size $\lambda$ in the random graph (\ref{claimtrg}).
Around this we set up the frame for the construction: the notion of general construction sequence, $(\theta, T)$-extension, and 
the normal form of (\ref{smooth}). 

Notice that this section establishes something quite a bit stronger than pairwise incomparable theories: 
we are simultaneously separating all theories $\{ T_\xm : \xm \in \mcm \}$ from all theories $\{ T_\xn : \xn \in \mcn \}$ in the sense of 
\ref{conv10}(3). 

\begin{conv} \label{conv10} For this section we fix:
\begin{enumerate}
\item $\bar{m}$, $\bar{E}$, $\Xi = \{ \xi_\alpha : \alpha < 2^{\aleph_0} \}$ satisfying the hypotheses of 
$\ref{4.8}$, 
\item a set $\sM_* = \{ \xm_\alpha = \prm[\bar{m}, \bar{E}, \xi_\alpha] : \alpha < 2^{\aleph_0} \}$ of parameters as in $\ref{4.8}$, 
\item $\mcm, \mcn$ two nonempty disjoint subsets of $\sM_*$. 
\item $\aleph_0 \leq {{\mu}} < \kappa \leq \lambda$, and $\kappa$ is regular $($and uncountable$)$.  
The intent of these cardinals was discussed in $\ref{disc-card}$ above. 
\item $\mci_\mcm$ the ideal corresponding to $\mcm$.
\item $\ba_* = \ba^1_{\kappa, {{\mu}}, \aleph_0}$.
\item $\de_*$ an arbitrary but fixed nonprincipal ultrafilter on $\ba_*$.
\end{enumerate}
\end{conv}

\begin{defn} \label{d:ap0}
Let $\AP_0$ be the class of objects $\mfa = (\ba_\mfa, \de_\mfa)$ where:  
\begin{enumerate}
\item $\ba_\mfa$ is a complete Boolean algebra and $\ba_{*} \lessdot \ba_\mfa$ $($note that this is satisfied by 
$\ba_*$ itself$)$. 
\item $\de_\mfa$ is an ultrafilter on $\ba_\mfa$ extending $\de_*$. 
\end{enumerate}
\end{defn}

\begin{defn} \label{d:ap}
Let $\AP$ be the class of objects $\mfa = (\ba_\mfa, \de_\mfa) \in \AP_0$ such that in addition 
$\ba_\mfa$ satisfies the $(\kappa, \mci, \bar{m})$-c.c. 
\end{defn}

\begin{conv} \label{mfa0}
For this section, let $\mfa_*$ denote $(\ba_*, \de_*)$ from $\ref{conv10}$.
\end{conv}

\begin{rmk} \label{base-case}
By definition $\mfa_* \in \AP_0$, and by $\ref{931fa}$, $\mfa_* \in \AP$. 
\end{rmk}

\begin{defn} \label{d:partialorder} \emph{ }
\begin{enumerate}
\item We define a partial order on the elements of $\AP_0$: 
\[ \mfa \leq_{\AP_0} \mfb \]
when $\ba_\mfa \lessdot \ba_\mfb$ and $\de_\mfa \subseteq \de_\mfb$. 

\item $\leq_{\AP}$ is the following partial order on $\AP$:
\[ \mfa \leq_{\AP} \mfb \]
if and only if:
\begin{enumerate}
\item $\mfa \leq_{\AP_0} \mfb$
\item $\mfa, \mfb \in \AP$
\item the pair $(\ba_\mfa, \ba_\mfb)$ satisfies the $\kim$-pattern transfer property. 
\end{enumerate}
\end{enumerate}
\end{defn}

\begin{disc} \emph{This is a partial order by \ref{pt-trans} (transitivity for pattern transfer).  
Recall that to show $\mfa \leq_{\AP} \mfb$, by \ref{o:upgrade} it suffices to verify that   
$\mfa$ has the $\kim$-c.c. (i.e., $\mfa \in \AP$) and $(\ba_\mfa, \ba_\mfb)$ has the $\kim$-pattern transfer property.}
\end{disc}

\begin{defn}  \label{d:constr}
Call $\bar{\mfb} = \langle \mfb_\gamma : \gamma < \gamma_* \rangle$ a \emph{general construction sequence} when:
\begin{enumerate}
\item[(A)] $\mfb_0 = \mfa_*$, so $\mfb_0 \in \AP$. 
\item[(B)] for $\gamma < \gamma_*$, $\mfb_\gamma \in \AP_0$. 
\item[(C)] for $\gamma < \gamma_*$, $\mfb_\gamma \leq_{\AP_0} \mfb_{\gamma+1}$.
\item[(D)] for $\gamma$ a nonzero limit ordinal, $\bigcup_{\beta < \gamma} \ba_{\mfb_\beta}$ is a dense subset of $\ba_{\mfb_\gamma}$ and 
$\de_{\mfb_\gamma}$  is an ultrafilter on $\ba_{\mfb_\gamma}$ which includes $\bigcup_{\beta < \gamma} \de_{\mfb_\beta}$.  
\end{enumerate}
We say the length of $\bar{\mfb}$ is $\gamma_*$. 
\end{defn}

This definition is justified by: 

\begin{claim}  \label{concl1a}
Suppose $\bar{\mfb} = \langle \mfb_\gamma : \gamma \leq \gamma_* \rangle$ satisfies $\ref{d:constr}(A)+(C)+(D)$. 
Then for each $\gamma \leq \gamma_*$, the ultrafilter $\de_{\mfb_\gamma}$ exists, and for each $\beta < \gamma \leq \gamma_*$, 
$\ba_{\mfb_\beta} \lessdot \ba_{\mfb_\gamma}$.   
In particular, each $\mfb_\gamma \in \AP_0$, and for each $\beta < \gamma \leq \gamma_*$, 
\[ \mfb_{\beta} \leq_{\AP_0} \mfb_{\gamma}. \]
\end{claim}

\begin{proof}
Let us prove, by induction on $\gamma \leq \gamma_*$, that each 
$\mfb_\gamma \in \AP_0$ and that $\delta < \gamma$ implies $\mfb_\delta \leq_{\AP_0} \mfb_\gamma$. 
For the base case, we know $\mfb_0 \in \AP_0$, indeed $\mfb_0 \in AP$.  For the successor case, apply \ref{d:constr}(C), 
which implies membership in $\AP_0$.  

Suppose we are at a limit ordinal. 

For the ultrafilter:  it suffices to check that for limit $\gamma$, 
$\bigcup_{\beta < \gamma}  \de_{\mfb_\beta}$ has the finite intersection 
property, which follows from the fact that each $\de_{\mfb_\beta}$ is itself a filter. 

For the Boolean algebras: suppose $\beta < \gamma \leq \gamma_*$ and 
$\gamma = \beta + \alpha$  and argue by induction on $\alpha$.  If $\alpha = 0$ this is immediate, if $\alpha$ is a successor also clear. 
Suppose $\alpha$ is a limit and let $X \subseteq \ba_\beta$. Let $\ma_\beta$ be the supremum of $X$ in $\ba_{\mfb_\beta}$ 
and let $\ma_\gamma$ be the supremum of $X$ in $\ba_{\mfb_\gamma}$. 
Suppose for a contradiction that in $\ba_{\mfb_\gamma}$, $\ma_\beta \setminus \ma_\gamma = \mc> 0$.  By definition of general construction sequence, 
$\bigcup_{\beta < \gamma} \ba_{\mfb_\beta}$ is dense in $\ba_{\mfb_\gamma}$, so there is $\delta < \gamma$ and $\mc_\delta \in \ba^+_{\mfb_\delta}$ 
such that $\mc_\delta < \mc$.  Then in $\ba_{\mfb_\delta}$, $(\ma_\beta \setminus \mc_\delta) \geq \mx$ for all $\mx \in X$,  contradicting the inductive hypothesis. 
\end{proof}

\begin{cor} \label{constr}
Suppose $\bar{\mfb} = \langle \mfb_\gamma : \gamma \leq \gamma_* \rangle$ is a general construction sequence. 
Suppose that for every $\beta < \gamma_*$,  the pair $(\ba_{\mfa_\beta}, \ba_{\mfa_{\beta+1}})$ satisfies the 
$\kim$-pattern transfer property. Then each $\mfb_\gamma \in \AP$, and indeed for every $\beta < \delta \leq \delta_*$, 
\[ \mfb_{\beta} \leq_{\AP} \mfb_{\delta}. \]
\end{cor}

\begin{proof} 
By Claim \ref{931fa}, Lemma \ref{lem15}, and Claim \ref{concl1a}. 
\end{proof}

\br

\begin{rmk}
Note to the reader: sometimes $\theta$ has a special meaning in Boolean algebras, such as an upper bound on the intersections allowed in 
a free Boolean algebra under consideration. In the present paper, that role is always played by $\aleph_0$, so $\theta$ is free to 
be used for other things, as in the next definition. 
\end{rmk}

Our next definition expresses that we extend $(\ba_\mfa, \de_\mfa)$ in a certain minimal way: 
by simply adding a formal solution to some possibility pattern 
$\bar{\mb} = \langle \mb_u : u \in [\theta]^{<\aleph_0} \rangle$ 
for some theory.  For now, the definition is general, allowing the size $\theta$ and the theory to vary.  We could think 
about such extensions as simply ensuring an instance of goodness, adding some multiplicative refinement to some 
monotonic function.  The crucial point is that we do this as freely as possible, essentially only requiring that the equations in $\ba_\mfa$ are still respected, and 
the new addition $\langle \mb^1_{\{\alpha\}} : \alpha < \theta \rangle$ is a formal solution to $\bar{\mb}$, i.e., for each $u \in [\theta]^{<\aleph_0}$, 
$\bigcap_{\alpha \in u} \mb^1_\alpha \leq \mb_u$. 
Since we are adding multiplicative refinements, it suffices to specify $\mb^1_u$ for 
$|u| = 1$. As noted there, in the rest of the paper,  we will often drop 
parentheses when $|u| = 1$, writing $\mb^1_\alpha$ instead of $\mb^1_{\{\alpha\}}$. 
Notice that by \ref{922}(3), we will need to check existence. 

\begin{defn} \label{922} 
Suppose $\mfa \in \AP_0$, $T$ is a complete first-order theory, and $\theta \leq \lambda$ is an infinite cardinal. 
Say that $\mfb = (\ba_\mfb, \de_\mfb)$ is a  \emph{$(\theta, T)$-extension of $\mfa$} when 
there exists a possibility pattern  $\bar{\mb} = \{ \mb_u : u \in [\theta]^{<\aleph_0} \}$ 
in $(\ba_\mfa, \de_\mfa)$ for the theory $T$  
and:
\begin{enumerate}
\item $\ba_\mfb$ is the completion of the Boolean algebra $\ba$ generated by the set $\mcy_{\mfa, \mfb}$ which is 
$\ba_\mfa$ along with the set of new elements 
$\{ \mb^1_{\{\alpha\}} : \alpha < \theta \}$, freely except for the set of equations $\Gamma_{\mfa, \mfb}$ which are:\footnote{i.e. freely 
except for the rules already in $\ba_\mfa$ and the new rules stating that $\bar{\mb}^1$ is 
a formal solution to $\bar{\mb}$.}  
\begin{enumerate}
\item the equations already in $\ba_\mfa$.
\item for every nonempty finite $u \subseteq \theta$, 
\[  \bigcap_{\alpha \in u} \mb^1_{ \{\alpha\} } \leq \mb_u. \]
\end{enumerate}
\item Notation: for $|u| > 1$, let $\mb^1_u := \bigcap_{\alpha \in u} \mb^1_{\{\alpha\}}$.  Convention: $\mb^1_\emptyset = 1_\ba$. 
\item Convention: for readability, when $u = \{ \alpha \}$ and it is unlikely to cause confusion, we may drop parentheses and write $\mb^1_\alpha$ for $\mb^1_{\{ \alpha \}}$, 
so the new elements are $\{ \mb^1_\alpha : \alpha < \theta \}$. 
\item $\de_\mfb$ is an ultrafilter on $\ba_\mfb$ which agrees with $\de_\mfa$ on $\ba_\mfa$, and such that $\mb^1_\alpha \in \de_\mfb$ 
for all $\alpha < \theta$, \emph{if} such an ultrafilter exists; otherwise not defined. 
\end{enumerate}
We may say $\mfb$ is an $(\theta, T, \bar{\mb})$-extension of $\mfa$ to emphasize that $\bar{\mb}$ is the possibility pattern acquiring a solution. 
\end{defn}

\begin{rmk}  
Recalling $\ref{n:sv}$, a possibility pattern for $T$ comes naturally equipped with 
the data of $\ma[\psi[\bar{x}_u]] \in \ba_\mfa$ for $\psi \in \ml(\tau_T)$; we will use these 
in some proofs. 
\end{rmk}

We record the following here though it refers to upcoming proofs: 

\begin{obs} \label{o:ba-sizes}
If $\mfb \in \AP$ is a $(\theta, T, \bar{\mb})$-extension of $\mfa \in \AP$, and:
\begin{enumerate}
\item $\theta < \kappa$, or 
\item $T = T_\xm$ for $\xm \in \mcm$ and $\bar{\mb}$ is a possibility pattern coming from a positive $P_\nu(x) \land R(x,y)$-type 
for some $\nu \in \mct_{1}$, or 
\item $T= \trg$ and $\bar{\mb}$ is a possibility pattern coming from a type in positive and negative instances of the graph edge relation, 
\end{enumerate}
then 
\[ |\ba_\mfb| \leq (| \ba_\mfa | + \lambda)^{\kappa}. \] 
\end{obs}

\begin{proof}
By the $\kappa$-c.c. which follows from ``$\mfa, \mfb \in \AP$'', and which is proved in Claim \ref{10.12A} for (1), 
Theorem \ref{9.28a} for (2)  and Lemma \ref{claimtrg}  for (3).  Alternately, we could use $\lambda \geq \theta, \kappa$ and conclude 
$ |\ba_\mfb| \leq (| \ba_\mfa | + \lambda)^{\lambda}$, which suffices for our purposes. 
\end{proof}

\begin{cor} \label{c:size} 
Suppose $\langle \mfb_\gamma : \gamma \leq \gamma_* \rangle$ is a general construction sequence with 
$\gamma_* \leq 2^\lambda$, and 
for each $\gamma < \gamma_*$, $\mfb_{\gamma+1}$ is a $(\theta, T, \bar{\mb})$-extension of $\mfb_\gamma$ in the sense of 
$\ref{o:ba-sizes}(1)$, $(2)$ or $(3)$. Then  $|\ba_{\mfb_{\gamma}}| \leq 2^\lambda$ for all $\gamma \leq \gamma_*$. 
\end{cor}

\begin{proof}
By induction on $\gamma$. For $\gamma = 0$, $|\ba_*| \leq 2^\lambda$; for $\gamma = \beta+1$, apply \ref{o:ba-sizes}, 
and at nonzero limit stages, use the $\kappa$-c.c. and the fact that $|\gamma| \leq 2^\lambda$.
\end{proof}

\begin{rmk} In the definition we put no restrictions on the theory; only in the actual construction do we use 
$T_\xn$, $T_\xm$ and $\trg$.  
\end{rmk} 

\begin{claim} \label{ext1} 
Suppose $\mfb$ is a $(\theta, T, \bar{\mb})$-extension of $\mfa \in \AP$, for some theory $T$ and some $\theta \leq \lambda$. Then:
\begin{enumerate}
\item $\ba_{\mfa} \subseteq \ba_{\mfb}$,  
\item indeed, $\ba_{\mfa} \lessdot \ba_{\mfb}$.
\item there exists an ultrafilter $\de$ on $\ba_\mfb$ which agrees with $\de_\mfa$ on $\ba_\mfa$ and contains 
$\mb^1_\alpha$ for all $\alpha < \theta$, hence $\de_\mfb$ is such an ultrafilter.  
\end{enumerate}
\end{claim}

\setcounter{equation}{0}

\begin{proof} 
Recall from \ref{922} that $\ba_\mfb$ is the completion of the Boolean algebra generated by $\mcy = \ba_\mfa \cup \{ \mb^1_\alpha : \alpha < \theta \}$ 
freely except for the set of equations $\Gamma$, which include all equations already in $\ba_\mfa$ along with 
equations saying that for each finite $u \in [\theta]^{<\aleph_0}$, $\bigcap_{\alpha \in u} \mb^1_\alpha \leq \mb_u$.

For each $u \in [\theta]^{<\aleph_0}$, define $h_u : \mcy \rightarrow \ba_\mfa$ as follows: 
$h_u \rstr \ba_\mfa$ is the identity, $h_u(\mb^1_\alpha) = \mb_u$ if $\alpha \in u$, and $h_u(\mb^1_\alpha) = 0_{\ba_\mfa}$ if 
$\alpha \in \theta \setminus u$.  Note that $h_u$ respects the equations in $\Gamma$. 

To see that $\ba_\mfa \subseteq \ba_\mfb$, note that in the case $u = \emptyset$ (as the generators are dense 
in the completion) the map $h_\emptyset$ induces an endomorphism $\hat{h}_\emptyset$ 
from $\ba_\mfb$ onto $\ba_\mfa$ which extends the 
identity map on $\ba_\mfa$.  This proves (a).

Next we work towards showing that $\ba_\mfa$ is a complete subalgebra of $\ba_\mfb$. Note that $\ba_\mfb$ is by definition a 
complete Boolean algebra.  Fix for awhile $\mc \in \ba^+_\mfb$. As the generators are dense in the completion,  
we may find 
\begin{equation}
\label{eq-iseq} \mx \in \ba^+_\mfa, ~u, u_0, \dots, u_{n-1} \in [\theta]^{<\aleph_0}
\end{equation} 
such that 
$u_\ell \not\subseteq u$ for $\ell < n$ and 
\[ \ba_\mfb \models  0 < \mx \cap \mb^1_u \cap \bigcap_{\ell < n} (- \mb^1_{u_\ell})  \leq \mc. \]
[Note that $\{ \mb^1_\alpha : \alpha < \theta \}$ generates a multiplicative sequence: 
$\mb^1_u = \bigcap_{\alpha \in u} \mb^1_\alpha$ for any finite $u \subseteq \theta$. So the positive intersection in the inset equation 
may be given by a single $u$.  We could have taken the $u_\ell$'s to be singletons, without loss of generality.
Note that we could also have assumed, 
without loss of generality, that $\mx \leq \mb_u$.]

Thus, for $\mc, \mx, u, u_0, \dots, u_{n-1}$ as in the previous paragraph, the map $\hat{h}_u$ is constant on $\ba_\mfa$, 
takes $\mb^1_\alpha$ to $\mb_u$ for $\alpha \in u$, and $\mb^1_\alpha$ to $0_{\ba_\mfa}$ for $\alpha \in \theta \setminus u$, 
hence takes $\mb^1_u$ to $\mb_u$, and each $\mb^1_{u_\ell}$ to $0_{\ba_\mfa}$ (for $\ell < n$). 

It follows that for any $\md \in \ba^+_{\mfb}$, if $\ba_{\mfa} \models 0 < \md \leq \mx \cap \mb_u$ then 
$\ba_\mfb \models 0 < \md \cap \mc$ [i.e., $\mx \cap \mb_u$ is below the projection of $\mc$, or if we 
chose $\mx \leq \mb_u$, that already $\mx$ is below the projection of $\mc$].  

Since $\mc$ was arbitrary, we have shown that any maximal antichain of $\ba_{\mfa}$ will remain a maximal antichain of $\ba_{\mfb}$ 
(if not, there is some nonzero $\mc \in \ba^+_{\mfb}$ which does not have nonempty intersection with any element of the antichain; 
but its corresponding $\mx \cap \mb_u$ must, contradiction).   
This completes the proof of (b). 

Finally, to verify (c), it suffices to show that $\de_\mfa \cup \{ \mb^1_\alpha : \alpha < \theta \}$ has the finite intersection property in 
$\ba_{\mfb}$, as then it can be extended to an ultrafilter.  This follows from the existence of the $\hat{h}_u$'s. (Suppose that 
for some finite $u$ and some set $\ma \in \de_\mfa$, $\ma \cap \bigcap \{ \mb^1_\alpha : \alpha \in u \} = 0_{\ba_\mfb}$. Then 
$\hat{h}_u(\mb^1_\alpha) = \mb_u$ for each $\alpha \in u$.  Recall that $\langle \mb_u : u \in [\theta]^{<\aleph_0} \rangle$ was 
a possibility pattern for $(\ba_\mfa, \de_\mfa)$, thus a sequence of elements of $\de_\mfa$; in particular, $\mb_u \in \de_\mfa$, 
so $\ba_\mfa \models  ~ \mb_u \cap \ma > 0$, contradiction.)  \end{proof}

We record a simple variant for future proofs:  

\begin{obs-star} \label{ext2} 
Suppose that $\ma \in \AP$ but instead of taking a $(\theta, T)$-extension of $\ba_\mfa$, 
we consider $\ba$ which is generated from $\ba_\mfa$ along with up to $\theta$ new antichains each of cardinality $<\kappa$, as freely as possible: 
that is, for some $h \in {^\theta \kappa}$, $\ba_\mfa \cup \{ \mc_{\alpha, \epsilon} : \epsilon < h(\alpha), \alpha < \theta \}$ freely 
except for the equations already in $\ba_\mfa$ and the equations saying that for each $\alpha$, 
$\{  \mc_{\alpha, \epsilon} : \epsilon < h(\alpha) \}$ is an antichain. Let $\ba_\mfb$ be the completion of $\ba$. Then 
the proof that 
\begin{enumerate}
\item $\ba_\mfa \subseteq \ba_{\mfb}$
\item indeed $\ba_\mfa \lessdot \ba_{\mfb}$ 
\item $\ba_\mfb$ has the $\kappa$-c.c.
\end{enumerate}
is easier than in \ref{ext1}, and just as in \ref{c:size}, we may conclude $|\ba_\mfb| \leq ( |\ba_\mfa| + {{\mu}})^{<\kappa}$.  
\end{obs-star}

Keeping in mind \ref{constr}, our main task now will be to show that we can preserve the pattern transfer property at 
successor stages realizing certain specific types for certain specific theories. 
We will make repeated use of the move in the proof of \ref{ext1}, equation (\ref{eq-iseq}), giving a useful ``normal form'' for 
elements, so we start by summarizing it here.  
Note in the next observation that we bound $\theta$ by $\lambda$, and in particular, 
it is possible for $\theta$ to be larger than $\kappa$ (recalling the remark before 
\ref{delta-system}). 

\setcounter{equation}{0}
\begin{obs} \label{smooth}
Suppose $\mfa \in \AP$ and $\mfb$ is a $(\theta, T, \bar{\mb})$-extension of $\mfa$ for some $\theta \leq \lambda$. Suppose we 
are given a sequence $\langle \ma^2_\alpha : \alpha < \kappa \rangle$ of elements of $\ba^+_\mfb$. Then: 

\begin{enumerate}
\item[$(1)$]  for each $\alpha < \kappa$, there is $\ii_\alpha = (\mx_\alpha, u_\alpha, \langle u_{\alpha, \ell} : \ell < n_\alpha \rangle)$  such that  $\mx_\alpha \in \ba^+_\mfa$; $n_\alpha \in \mathbb{N}$; $u_\alpha$, $u_{\alpha,0}, \dots, u_{\alpha,n_\alpha-1} 
\in [\theta]^{<\aleph_0}$; 
$u_\alpha \not\subseteq u_{\alpha,\ell}$ for $\ell < n_\alpha$; $\mx_\alpha \leq \mb_{u_\alpha}$; 
and 
\[ \ba_\mfb \models  0 < \mx_\alpha \cap \mb^1_{u_\alpha} \cap \bigcap_{\ell < n_\alpha} (- \mb^1_{u_{\alpha,\ell}})  \leq \ma^2_\alpha. \] 
\emph{\small{[i.e., since the generators are dense in the completion, we can find a positive element below $\ma^2_\alpha$ which is the intersection of 
an element from $\ba_\mfa$, some number of new elements, and some number of negations of (intersections of) new elements.]}}
\br
\item[$(2)$] 
Given $\ii_\alpha$ for $\alpha < \kappa$ from $(1)$, define $w_\alpha = u_\alpha \cup \bigcup\{ u_{\alpha, \ell} : \ell < n_\alpha \}$.  
Then there are $\uu \in [\kappa]^\kappa$, 
$w_*$, $u_*$, $n_*$, ${\langle u^*_\ell : \ell < n_* \rangle}$ such that for every $\alpha \in \uu$, 
$w_\alpha = w_*$, 
$n_\alpha = n_*$, $u_\alpha \cap w_* = u_*$, $u_{\alpha, \ell} \cap w_* = u^*_\ell$.

\br
\noindent \emph{\small{[i.e., by applying the $\Delta$-system lemma we can smooth this out on a large set.]}}
\br

\item[$(3)$] For every $\alpha \in \uu$ and $\mx_\alpha$ from $\ii_\alpha$, we have that $\mx_\alpha \bp \ma^2_\alpha$.

\br
\item[$(4)$] Suppose $\uu$ is from $(2)$ and $X \subseteq \uu$ is finite and $\ma_* \in \ba^+_\mfa$. Suppose
\[ \ba_\mfb \models \ma_* \cap  \bigcap_{\alpha \in X} \left(\mx_\alpha \cap \mb^1_{u_\alpha}\right)  > 0. \]
Then also 
\[ \ba_\mfb \models  \ma_* \cap \bigcap_{\alpha \in X}  \left(\mx_\alpha \cap \mb^1_{u_\alpha} \cap \bigcap_{\ell < n_\alpha} (- \mb^1_{u_{\alpha,\ell}}) \right)  > 0 \]
\emph{\small{[i.e., when checking for positive intersections we may safely ignore complements.]}}
\end{enumerate}
\end{obs}

\begin{proof}
For part (1), the generators are dense in the completion, and as mentioned in the proof of \ref{ext1}, we can gather the 
intersection of elements of the form $\mb^1_\beta$ into a single $\mb^1_{u_\alpha}$.  
Since $\mx_\alpha \cap \mb_u > 0$ in $\ba_\mfa$, there is no harm in assuming $\ba_\mfa \models \mx_\alpha \leq \mb_u$. 

For part (2), 
by the $\Delta$-system lemma (recall that $\kappa$ is regular) 
there is $\uu \in [\kappa]^\kappa$ such that $\langle w_\alpha : \alpha \in \uu \rangle$ is a $\Delta$-system 
with heart $w_*$. So we can assume for some $u_*$, $n_*$, $\langle u^*_\ell : \ell < n \rangle$, for every $\alpha \in \uu$, we have that 
$n_\alpha = n$, $u_\alpha \cap w_* = u_*$, $u_{\alpha, \ell} \cap w_* = u^*_\ell$.  Note we may ask for 
additional uniformity, e.g.  that the $u_\alpha$'s have constant size. 

To verify  $\mx_\alpha \bp \ma^2_\alpha$ for (c), it suffices to show that there is an endomorphism $f$ from 
$\ba_\mfb$ onto $\ba_\mfa$ fixing $\ba_\mfa$ pointwise 
such that $f(\ma^2_\alpha) \geq \mx_\alpha$.  
Consider the map $\hat{h}_{u_\alpha}: \ba_\mfb \rightarrow \ba_\mfa$ defined in the proof of 
\ref{ext1}. Then $\hat{h}_{u_\alpha}(\mx_\alpha) = \mx_\alpha$, 
$\hat{h}_{u_\alpha}(\mb^1_{u_\alpha}) = \mb_{u_\alpha}$, and by 
the disjointness conditions of the $\Delta$-system, 
$\hat{h}_{u_\alpha}(\mb^1_{u_{\alpha,\ell}}) = 0_{\ba_\mfa}$ for all $\ell < n$.
So $\hat{h}_{u_\alpha}(\ma^2_\alpha) \geq \mx_\alpha \cap \mb_{u_\alpha} = \mx_\alpha$, recalling that we 
assumed $\mx_\alpha \leq \mb_{u_\alpha}$ in part (1).  

Part (4) is similar. 
First consider the simple case where $\ma_* = 1_{\ba_{\mfa}}$, so we may ignore it. 
Let 
\[ u = \bigcup_{\alpha \in X} u_\alpha. \] 
Recall that $\ba^1_\mfb \models \bigcap_{\alpha \in X} \mb^1_{u_\alpha} = \mb^1_u$, since $\bar{\mb}^1$ is multiplicative, and 
$\mb^1_u \leq \mb_u$, since the sequence $\bar{\mb}^1$ refines the sequence $\bar{\mb}$. 
Thus in $\ba^1_\mfb$, starting with our assumption, 
\begin{equation}
\label{eq:gz} 0 <  \bigcap_{\alpha \in X} ( \mx_\alpha \cap \mb^1_{u_\alpha})  = \bigcap_{\alpha \in X} \mx_\alpha 
\cap \bigcap_{\alpha \in X} \mb^1_{u_\alpha} = 
\bigcap_{\alpha \in X} \mx_\alpha ~ \cap ~ \mb^1_{u} \leq \bigcap_{\alpha \in X} \mx_\alpha ~ \cap ~ \mb_{u} . 
\end{equation}
Let $\my = \bigcap_{\alpha \in X} ( \mx_\alpha \cap \mb^1_{u_\alpha})$.
It suffices to show there is an endomorphism $f$ from 
$\ba_\mfb$ onto $\ba_\mfa$ such that $f(\my) > 0$ but $f(\mb^1_{u_{\alpha,\ell}}) = 0_{\ba_\mfa}$ for all $\alpha \in X, \ell < n$. 
Recalling again the map from the proof of $\ref{ext1}$,  $\hat{h}_u : \ba_{\mfb} \rightarrow \ba_{\mfa}$.  Then $\hat{h}_u(\mx_\alpha) = \mx_\alpha$ 
for $\alpha \in X$, and $\hat{h}_u(\mb^1_\beta) = \mb_u$ for all $\alpha \in X$ and $\beta \in u_\alpha$.
Thus, $\hat{h}_u(\mb^1_{u_\alpha}) = \mx_\alpha \cap \mb_u$ for $\alpha \in X$, so 
remembering equation (\ref{eq:gz}), 
$\hat{h}_u(\my) = \bigcap_{\alpha \in X} \mx_\alpha \cap \mb_u > 0$.  
The effect of the $\Delta$-system ensures that $u_{\alpha,\ell} \cap u = \emptyset$ 
for $\alpha \in X, \ell < n$, so $\hat{h}_u(\mb^1_{\beta}) = 0_{\ba_\mfa}$ for $\beta \in u_{\alpha,\ell}$, $\alpha \in X, \ell <n$, as desired. 

Now assume that $\ma_* \in \ba^+_\mfa$ is arbitrary but given. 
Equation (\ref{eq:gz}), condensed for space reasons, becomes:
\begin{equation}
\label{eq:gza} 0 <  \ma_* \cap \bigcap_{\alpha \in X} ( \mx_\alpha \cap \mb^1_{u_\alpha})  = \cdots  \cdots 
\leq \ma_* \cap \bigcap_{\alpha \in X} \mx_\alpha ~ \cap ~ \mb_{u} . 
\end{equation}
Let $\mz = \ma_* \cap \my$. Under the same map $\hat{h}_u$, note $\hat{h}_u(\ma_*) = \ma_*$ since it is an element of 
$\ba_\mfa$. So by equation (\ref{eq:gza}), $\hat{h}_u(\mz) = \ma_* \cap 
 \bigcap_{\alpha \in X} \mx_\alpha \cap \mb_u > 0$ as desired. 
\end{proof}

Our next claim says essentially that if $\theta < \kappa$ then there is no problem realizing any 
$(T,\theta)$-type for any $T$.   
Thus, we can arrange for our final ultrafilters to be $\kappa$-good, even though they will be far from $\kappa^+$-good. 

\begin{claim}[Realizing small types] \label{10.12A}
Assume $\mfa \in \AP$, $\mfb$ is a $(\theta, T, \bar{\mb})$-extension of $\mfa$, where 
$\theta < \kappa$.    Then  $(\ba_{\mfa}, \ba_{\mfb})$ satisfies the $\kim$-pattern transfer property. 
\end{claim}

\setcounter{equation}{0}

\begin{proof}  
This proof and the proof of \ref{9.28a} share a picture, so in this simpler case, we go slowly to motivate the 
second proof.  

Let $\langle \ma^2_\alpha : \alpha < \kappa \rangle$ be given, with each $\ma^2_\alpha \in \ba^+_\mfb$.  First we appeal to 
the normal form of \ref{smooth}(1):  
for all $\alpha < \kappa$ we can find $\ii_\alpha = (\mx_\alpha, u_\alpha, \langle u_{\alpha, \ell} : \ell < n_\alpha \rangle)$ as there,
so $\mx_\alpha \leq \mb_{u_\alpha}$, and 
\[ \ba_\mfb \models  0 < \mx_\alpha \cap \mb^1_{u_\alpha} \cap \bigcap_{\ell < n_\alpha} (- \mb^1_{u_{\alpha,\ell}})  \leq \ma^2_\alpha. \] 
In our present case, we don't need to appeal to the $\Delta$-system reduction of \ref{smooth}(2), since something stronger is true: 
we have assumed $\theta < \kappa$, and note each $u_\alpha, u_{\alpha, \ell}$ is a finite subset of $\theta$. 
So we may find $\uu \in [\kappa]^\kappa$ on which the sequence  
\begin{equation} 
\label{eqn26}
\langle (u_\alpha, \langle u_{\alpha,\ell} : \ell < n_\alpha \rangle) : \alpha \in \uu \rangle 
\end{equation} 
is constant, and say equal to $(u_\oplus, \langle u_{\oplus,\ell} : \ell < n_\oplus \rangle)$.
Set \begin{equation} 
\label{eq325a}
\ma^1_\alpha = \mx_\alpha \cap \mb_{u_\oplus} = \mx_\alpha \mbox{ for each $\alpha \in \uu$}.  
\end{equation}
Then each $\ma^1_\alpha \bp \ma^2_\alpha$ (by \ref{smooth}(3)). 
Let us verify \ref{paircc}(4)(f) holds for $A = \emptyset$, $j= 0$ to transfer from $\bar{\ma}^1$ to $\bar{\ma}^2$. 
Suppose we are given $n \in \omega \setminus A$, a finite $u \subseteq \uu$, and a nonzero $\ma_* \in \ba^+_\mfa$ such that 
$m_n / (m^\circ_n)^{n^i} < |u| < m_n$ and \begin{equation} \label{e155}
 \ba_\mfa \models 0 < \ma_* \leq \bigcap_{\alpha \in u} \ma^1_\alpha.  
\end{equation}
As $j=0$, to fulfill the $\kim$-pattern transfer we would like to find $v$ 
such that $v \subseteq u$ and $|v| \geq |u| / (m^\circ_n)^{n^{i+0}}$ {and} 
\begin{equation} \label{e156a} \ba_2 \models \bigcap_{\alpha \in v} \ma^2_\alpha ~ \cap ~ \ma_* > 0.  
\end{equation}
Remembering \ref{smooth}(4), to prove (\ref{e156a})  it would suffice to show that for this $v$, 
\begin{equation} \label{e156}
 \ba_\mfb \models \bigcap_{\alpha \in v} (\mx_\alpha \cap \mb^1_{u_\alpha}) \cap \ma_* > 0. 
\end{equation} 
Let us verify (\ref{e156}) holds already for $v = u$. [Clearly this choice of $v$ has an acceptable size.]   
Let $w = \bigcup_{\alpha \in v} u_\alpha$. Observe that to prove (\ref{e156}), it would suffice to show that
\footnote{Presently, 
it would be both sufficient and possible to show just that $\bigcap_{\alpha \in v} \mx_\alpha \leq \mb_w$, 
but this a priori more general criterion will be useful later.}
\begin{equation} \label{e157}
\mbox{for some nonzero $\ma_{**} \in \ba_\mfa$ with $0 < \ma_{**} \leq \ma_*$, we have $\ma_{**} \leq \mb_{w(v)}$. }
\end{equation} 
[Why would this suffice? By the choice of $\ma_{**}$ and equation (\ref{e155}),  we know that 
$0 < \ma_{**} \leq \ma_* \leq \bigcap_{\alpha \in v} \mx_\alpha$, 
since by (\ref{eq325a}) $\ma^1_\alpha$ is just another name for $\mx_\alpha$. 
Now similarly to earlier proofs \ref{ext1} and \ref{smooth}, $\hat{h}_w $ is an endomorphism from 
$\ba_\mfb$ onto $\ba_\mfa$ which is the identity on $\ba_\mfa$ and for each $\alpha \in v$ takes  
$\mx_\alpha \cap \mb^1_{u_\alpha} = \mx_\alpha \cap \bigcap_{\gamma \in u_\alpha} \mb^1_\gamma$ 
to $\mx_\alpha \cap \mb_w$.  Equation (\ref{e157}) 
would imply 
\[ \hat{h}_w\left(\ma_{**} \cap \bigcap_{\alpha \in v} (\mx_\alpha \cap \mb^1_{u_\alpha})\right) = 
\ma_{**} \cap \bigcap_{\alpha \in v} \mx_\alpha \cap \mb_w = \ma_{**} > 0. \]
Equation (\ref{e156}) follows.]  

It remains to prove (\ref{e157}). In fact $\ma_{**} = \ma_*$ already works. 
Remember that we are in a very special case: $v = u \subseteq \uu$, so $\langle u_\alpha : \alpha \in v \rangle$ is constantly equal to $u_\oplus$, so $w = \bigcup_{\alpha \in v} u_\alpha = u_\oplus$ (!). So $\mb_w = \mb_{u_\oplus}$. 
Meanwhile, from \ref{smooth}(1) we have $\mx_\alpha \leq \mb_{u_\alpha}$ for any $\alpha < \kappa$, 
so $\ma_* \leq \mx_\alpha \leq \mb_{u_\oplus}$ for $\alpha \in u$. 

This completes the proof of the $\kim$-pattern transfer, and so the proof of the Claim. 
\end{proof}

\begin{rmk} \label{rmk25}
Although \ref{10.12A} assumes $\theta < \kappa$, the same proof will show another hypothesis 
also works:  $\kappa \leq \theta \leq \lambda$ but  for some $\mu < \kappa$ we have 
$(\forall u \in [\theta]^{<\aleph_0})[~ \ba_\mfa \models 
\mb_u = \bigcap \{ \mb_{\{ \alpha \}} : \alpha \in (u \setminus {{\mu}}) \} \cap \mb_{u \cap {{\mu}}} ~]$.
\end{rmk}

\begin{proof}
Suppose we are given such a $\mu$. In equation 
(\ref{eqn26}) replace $\theta$ by ${{\mu}}$ and then let 
$\ma^1_\alpha = \mx_\alpha \cap \mb_{u_\oplus} \cap \bigcap_{\alpha \in u_\alpha \setminus {{\mu}}} \mb_\alpha$ for 
$\alpha \in \uu_1$, and verify that $\langle \ma^1_\alpha : \alpha \in \uu \rangle$ is as required, for $j = 1$ (or 0). 
\end{proof}

\br


\begin{theorem}[Realizing $T_\xm$-types] \label{9.28a}
Assume $\mfa \in \AP$ and $\mfb$ is a $(\theta, T, \bar{\mb})$-extension of $\mfa$
where $\theta \leq \lambda$, $T = T_\xm$ for some $\xm \in \mcm$, and 
$\bar{\mb}$ is a possibility pattern arising from 
a type of the form 
\[ { \{ Q_{\rho_*}(x) \}~ \cup ~}\{ R(x,a_\beta) : \beta < \theta \}\]  
{ for some $\rho_* \in \mct_1$}. Then $(\ba_\mfa, \ba_\mfb)$ satisfies the $(\kappa, \mci, \bar{\xm})$-pattern transfer property.
\end{theorem}

\setcounter{equation}{0}

\begin{proof}  
By hypothesis our $T_\xm$ is given by some 
\begin{equation}
\label{e:xm} \xm = \xm(\bar{m}, \bar{E}, \xi) \in \mcm. 
\end{equation}
This proof is similar to \ref{10.12A}, so we will be lighter on motivation already given there.  The main difference is that 
now we will have to handle the larger 
$\theta$ by leveraging an understanding of the type in our theory $T_\xm$.

Recall notation from \ref{922}: $\bar{\mb} = \langle \mb_u : u \in [\theta]^{<\aleph_0} \rangle$, a sequence of 
elements of $\ba_\mfa$, is the problem which was solved in $\mfb$, and was given along with the data of the form  
$\ma[\psi] \in \ba_\mfa$; and
$\mb^1_\alpha, \mb^1_u$  
in $\ba_\mfb$ are from the formal solution. 

We aim to prove the pattern transfer property, \ref{paircc}. 
Suppose we are given $\langle \ma^2_\alpha : \alpha < \kappa \rangle$ with each $\ma^2_\alpha \in \ba^+_\mfb$.  
First, following \ref{smooth}(1), for each $\alpha < \kappa$ we choose $\ii_\alpha = (\mx_\alpha, u_\alpha, n_\alpha, 
\langle u_{\alpha, \ell} : \ell < n_\alpha \rangle)$ as there, so $\mx_\alpha \leq \mb_{u_\alpha}$ for each $\alpha < \kappa$, and 
\begin{equation} 
\label{1-dot} \ba_\mfb \models  0 < \mx_\alpha \cap \mb^1_{u_\alpha} \cap \bigcap_{\ell < m_\alpha} (- \mb^1_{u_{\alpha,\ell}})  \leq \ma^2_\alpha. 
\end{equation}
By way of orientation: recall that each $u_\alpha$ is a 
finite subset of $\theta$, and $\mb_{u_\alpha}$ can be thought of as the $\jj$-image of 
$B_{u_\alpha} = \{ t \in I : M[t] \models \exists x \bigwedge_{\beta \in u_\alpha}  Q_{\rho_*}(x) \land R(x,a_\beta[t]) \}$ 
in some enveloping ultrapower, as in \S \ref{d:separation}.
In this language, we can say  
\begin{equation}
\label{eq-explain} \mb_{u_\alpha} \leq \ma[(\exists x) \bigwedge \{ Q_{\rho_*}(x) \land R(x,a_\gamma) : \gamma \in u_\alpha \}] 
\end{equation}
noting that on different positive $\ma \cap \mb_{u_\alpha}$, different additional constraints may be in force on the 
the $a_\beta$'s, for instance as described by the $\ma[\psi]$'s. 

Specifically, in our present context, we make the following upgrade to \ref{smooth}(1). For $\alpha < \kappa$, for 
each of the finitely many 
atomic formulas 
$\psi \in \{  a_\gamma = a_\beta$, $\mcq(x)$, $\mcp(a_\beta)$: 
$\gamma, \beta \in u_\alpha \}$, we may without loss of generality assume $\mx_\alpha$ from (\ref{1-dot}) decides it, that is, 
without loss of generality  
\begin{equation}
\label{4-dot} \mx_\alpha \leq \ma[\psi]  \mbox{  or  }  \mx_\alpha \cap \ma[\psi] = \emptyset ~~\mbox{ for each such $\psi$.}
\end{equation} 
We can now be sure by (\ref{eq-explain}) that 
\begin{equation}
\label{5-dot} 
\mx_\alpha  \leq \ma[Q_{\rho_*}(x)]   
\mbox{, and also } \mx_\alpha \leq \ma[\mcp(a_\gamma)] \mbox{ for each $\gamma \in u_\alpha$. } 
\end{equation}
By \ref{smooth}(2), 
there are $\uu \in [\kappa]^\kappa$, $w_*$, $u_*$, $n_*$, $\langle u^*_\ell : \ell < n_* \rangle$, and let us add, $m_*$ 
such that for every $\alpha \in \uu$: 
$n_\alpha = n_*$, $u_\alpha \cap w_* = u_*$, $u_{\alpha, \ell} \cap w_* = u^*_\ell$, and $|u_\alpha| = m_*$.  
Let 
\begin{equation} 
\label{eq325}
\ma^1_\alpha = \mx_\alpha \mbox{ for each $\alpha \in \uu$.} 
\end{equation} 
Then for each $\alpha$ we have $\ma^1_\alpha \bp \ma^2_\alpha$, by \ref{smooth}(3). 

As $\mfa \in \AP$, the $\kim$-c.c. holds for $\ba_\mfa$, so, given the sequence of elements 
$\bar{\ma}^1 = \langle \ma^1_\alpha : \alpha \in \uu \rangle = \langle \mx_\alpha : \alpha \in \uu \rangle$, fix 
\[ j_1 < \omega, \mbox{ \hspace{3mm}} \uu_1 \in [\uu]^\kappa, \mbox{ \hspace{3mm}} A_1 \in \mci \] 
such that $\oplus$ of \ref{d1:cca} holds. (We will not really use this additional strength in the present proof, but it is natural to add.) 
We aim to show that the pattern transfers for  
\[  j = j_1 + 2, \mbox{ \hspace{3mm}}  \uu_1, ~\mbox{ \hspace{3mm}}A = A_1 \cup \{ \ell : \ell \leq j + m_* + n_* \} \cup \xi^{-1}\{1 \}, \mbox{ \hspace{3mm}} \bar{\ma}^1 \rstr \uu_1 \]
recalling $\xi$ is the level function for $T_\xm$ from (\ref{e:xm}), so adding $\xi^{-1}\{ 1\}$ to $A$ amounts to ensuring 
$n$ is not 
an active level for $T_\xm$. 
Fix, then,  $n$, $u$, $i$, $\ma_*$ such that  $n \in \omega \setminus A$, $u$ is a finite subset of $\uu_1$, and 
$\ma_* \in \ba^+_\mfa$, and together they satisfy 
\begin{equation} \label{e255}
\mbox{ $m_n / (m^\circ_n)^{n^i} < |u| < m_n$ and  } \ba_\mfa \models 0 < \ma_* \leq \bigcap_{\alpha \in u} \mx_\alpha.  
\end{equation}
Note that by definition of $A$, necessarily $n > \max \{ m_*, n_* \}$. 
We would like to find $v \subseteq u$ of an ``appropriate size'' [i.e. satisfying \ref{paircc}(4)(f)(ii)] so that $\ba_\mfb \models \bigcap_{\alpha \in v} \ma^2_\alpha \cap 
\ma_* > 0$. First observe that by \ref{smooth}(4), it would suffice to show that for some $v \subseteq u$ of an appropriate size, 
\begin{equation} \label{e256}
 \ba_\mfb \models \bigcap_{\alpha \in v} (\mx_\alpha \cap \mb^1_{u_\alpha}) \cap \ma_* > 0 
\end{equation} 
and for (\ref{e256}), just as in \ref{10.12A}(\ref{e157}), 
it would suffice to show that for $w(v) = \bigcup_{\alpha \in v} u_\alpha$ and
 some $\ma_{**} \in \ba_\mfa$ with $0 < \ma_{**} \leq \ma_*$, we have  
\begin{equation} \label{e157a}
\mbox{ $\ma_{**} \leq \mb_{w(v)}$. }
\end{equation} 
In search of a suitable $v$, we make the following additional reductions. First, since $\mct_{2,n}$ is finite, 
 we may replace $\ma_*$ by a smaller positive element $\ma_{**} \in \ba_\mfa$ 
 if necessary [\emph{i.e. by decreasing $\ma_*$ to $\ma_{**}$ 
to decide a finite partition of $\ma[\mcp(a_\beta)]$ for each of the finitely many $\alpha \in u$ and $\beta \in u_\alpha$, remembering 
by equation (\ref{5-dot}) that $\ma_* \leq \mx_\alpha \leq \ma[\mcp(a_\beta)]$ for $\alpha \in u$, 
$\beta \in u_\alpha$}] so that for each $\alpha \in u$ and each $\gamma \in u_\alpha$, 
there is one and only one $\eta_\gamma \in 
\mct_{2,n}$ such that 
\begin{equation} 
\ma_{**} \leq \ma[P_{\eta_\gamma}(a_\gamma)]. 
\end{equation}
Informally, for each relevant formula $R(x,a_\gamma)$, $\ma_{**}$ decides 
which branch $a_\gamma$ belongs to at level $n$.
Next, we choose a subset of $u$ on which this sequence of choices is constant. 
That is, for each $\alpha \in u$, let $\langle \gamma(\alpha, \ell) : \ell < m_* \rangle$ list $u_\alpha$ in strictly increasing order, 
necessarily without repetitions. 
Let $\langle \eta_{\gamma(\alpha, \ell)} : \ell < m_* \rangle$ be the corresponding sequence of elements of $\mct_{2,n}$ 
chosen by $\ma_{**}$. 
Recalling \ref{d:indx}$(c)$, $|\mct_{2,n}| = m^\circ_n$, so 
there are at most $(m^\circ_n)^{m_*}$ possible such sequences, and so there are 
$\langle \nu_\ell : \ell < m_* \rangle$ and $v \subseteq u$ such that for every $\alpha \in v$, 
$\langle \eta_{\gamma(\alpha, \ell)} : \ell < m_* \rangle = \langle \nu_\ell : \ell < m_* \rangle$ and 
\begin{equation} |v| \geq |u| / (m^\circ_n)^{m_*} \geq |u| / (m^\circ_n)^n  \geq m_n / (m^\circ_n)^{n^{i}  \cdot n} = 
m_n / (m^\circ_n)^{n^{i+1}} . 
\end{equation}
Such a $v$ therefore has an appropriate size. There is no harm (and there will be a later help) in 
replacing $v$ by a subset if necessary so that we know its size exactly:  
\begin{equation} 
|v| = m_n / (m^\circ_n)^{n^{i+1}} . 
\end{equation}
Fix such a $v$ for the rest of the proof, and let $w = w(v)$. Note that, liberally,  
\begin{equation} \label{size-of-w}
|w| ~ = ~ m_* \cdot \left(m_n / (m^\circ_n)^{n^{i+1}}\right) 
~ \leq ~ m_n.
\end{equation}
It remains to show equation (\ref{e157a}) holds for our $\ma_{**}$ and this $w$.  

When we ask whether $\ma_{**} \leq \mb_w$, we are asking whether the decisions already made by $\ma_{**}$ are 
enough to guarantee consistency of 
\begin{equation}
\label{eq-type} \{ Q_{\rho_*}(x) ~\land~ R(x, a_{\gamma}) \land P_{\eta_\gamma}(a_{\gamma}) : \alpha \in u, \gamma \in u_\alpha \}  
\end{equation}
e.g. for any possible further choice of leaves that the $a_\gamma$'s may belong to. 
For each individual $\alpha \in  u$, we know $\ma_{**} \leq \mb_{u_\alpha}$, so the answer is yes for the smaller set 
\begin{equation}
\label{eq-type2} \{ Q_{\rho_*}(x) ~\land~ R(x, a_{\gamma}) \land P_{\eta_\gamma}(a_{\gamma}) : \gamma \in u_\alpha \}.  
\end{equation}
Recalling the constant sequence $\langle \nu_\ell : \ell < m_* \rangle$ from the definition of $v$, we may rewrite 
equation (\ref{eq-type2}) to say also that the answer is yes for 
\begin{equation}
 \{ {Q_{\rho_*}(x)~ \land ~} R(x, a_{\gamma(\alpha,\ell)}) \land P_{\nu_\ell}(a_{\gamma(\alpha,\ell)}) : \ell < m \}. 
\end{equation}
Thus, for some $\rho \in \mct_{1,n}$ such that $\rho_* \tlf \rho$ 
[\emph{recalling $n > \lgn(\rho_*)$ by the definition of $A$}] we have that  
\begin{equation}  \{ \rho \} \times \{ \nu_\ell : \ell < m_* \} \subseteq \mcr_{n}.  
\end{equation}
[Note that $\alpha$ doesn't really matter in the choice of $\rho$: the axioms of $T_\xm$ imply that 
for $\rho$ from $\mct_{1,n}$ and $\langle \nu_\ell : \ell < m_* \rangle$ a sequence of elements from $\mct_{2,n}$, the 
set $\{ Q_\rho(x) \land R(x,y_\ell) \land P_{\nu_\ell}(y_\ell) : \ell < m_* \}$ is 
consistent if and only if $\{ \rho \} \times \{ \nu_\ell : \ell < m_* \}$ is a complete bipartite graph in $\mcr_n$. 
Fix any $\rho$ which has this property; there may be many to choose from.]  

For each $\gamma \in w$, choose some $\eta^*_\gamma \in \lim(\mct_2)$ such that if $\gamma = \gamma(\alpha, \ell)$ 
then $\nu_\ell \tlf \eta^*_\gamma$ 
[i.e., this leaf extends the choice made for $a_\gamma$ by $\ma_{**}$].  We shall prove  
\begin{equation} \label{find-rho} 
\mbox{ there is $\varrho \in \lim(\mct_1)$ such that $\rho \tlf \varrho$ and $(\varrho, \eta^*_\gamma) \in \mcr_\infty$ 
for every $\gamma \in w$. }
\end{equation}
This will suffice for equation (\ref{e157a}) and so will finish the proof. 

To prove (\ref{find-rho}) we shall choose an increasing sequence $\varrho_{n+t} \in \mct_{1,n+t}$ by induction on $t \geq n$. 
The case of $t = 0$ is immediate: let $\varrho_n = \rho$, and then $(\varrho_n, ~\eta^*_\gamma \rstr_n) \in \mcr_n$ 
for each $\gamma \in w$ by choice of $\rho$ (and definition of $v$). 
For $t = 1$, recall the choice of $n \in \omega \setminus A$ means $\xi(n) = 0$, so there are no new constraints 
(it is a ``lazy level'').  
Let $\varrho_{n+1}$ be any element of $\mct_{1,n+1}$ which extends $\varrho_n$, and  
it will remain true that 
$(\varrho_{n+1}, ~ \eta^*_\gamma \rstr_{n+1}) \in \mcr_{n+1}$ for each $\gamma \in w$.
[Note that the lazy level plays an essential role here: $|w| \leq m_n$, and 
elements of $\mct_{2,n}$ have $m_n$ immediate successors, 
recalling \ref{d:indx}; improvements to the upper-bound calculation (\ref{size-of-w}) won't bring the size 
down to something small at this level.] 
For $t  > 1$, the coast is now clear: whether $\xi(t) = 0$ or $1$, recalling (\ref{size-of-w}), the set $w$ is now 
very comfortably small in the sense of \ref{f:a3}: 
\begin{equation}
\label{e:size}
|w| \leq m_n < m^\circ_{n+1} \leq (m^\circ_{n+t-1})^{(n+t-1)^{(n+t-1)}}. 
\end{equation}
So we may apply Claim \ref{c:helpful}(1), which ensures (using $n+t-1$ here for $k$ there)
that given $\{ \eta^*_\gamma \rstr_{n+t-1}: \gamma \in w \} \subseteq \mct_{2,n+t-1}$ along with equation (\ref{e:size})
and our inductive hypothesis, there is $\varrho_{n+t} \in \mct_{1,n+t}$
extending   
$\varrho_{n+t-1}$ such that 
$(\varrho_{n+t}, ~\eta^*_\gamma \rstr_{n+t}) \in \mcr_{n+t}$ for each $\gamma \in w$.   
So we can carry the induction.  Let $\varrho \in \lim(\mct_1)$ be the unique element such that $\varrho \rstr i = \varrho_i$ for all $i<\omega$. 
This proves (\ref{find-rho}), so proves (\ref{e157a}), and so finishes the proof of the theorem. 
\end{proof}

For our third and final case of realizing types, in the case of the random graph, the theory is simple enough that we 
will construct the refinement directly. The proof is like \cite{MiSh:1009} Theorem 3.2, though 
there are still many differences. The proof is mainly a matter 
of keeping track of equality. 

\setcounter{equation}{0}

\begin{defn} \label{d:rg-types}
For our purposes, ``\emph{$\bar{\mb} = \langle \mb_u : u \in [\theta]^{<\aleph_0} \rangle$ 
is a problem for the theory of the random graph}'' in $\ba$ means that 
$\bar{\mb}$ arises from some partial type $\{ R(x,a_\gamma)^{\trv(\gamma)} : \gamma < \theta, 
\trv: \theta \rightarrow \{ 0, 1 \} \}$ in some enveloping ultrapower,\footnote{Here of course $R(x,y)$ is the edge relation for the 
random graph, unlike elsewhere.} 
and so reflecting this, 
$\bar{\mb}$ is given along with 
$\{ \ma[a_\beta \neq a_\gamma] : \beta, \gamma < \theta \}$ and the sequence of truth values 
$\{ \trv(\beta) : \beta < \theta \}$, and
for each $u \in [\theta]^{<\aleph_0}$,  
\begin{equation}  \label{defn-bu}
\mb_u = \bigcap \{ \ma[a_\beta \neq a_\gamma] : \beta, \gamma \in u, \trv(\beta) \neq \trv(\gamma) \} 
\end{equation}
translating the fact that in the enveloping ultrapower, at a given index $i \in I$,  a subset  
$\{ R(x,a_\gamma[i])^{\trv(\gamma)} : \gamma \in u \}$ of the background type is consistent if for all $\gamma, \beta \in u$,  $\trv(\gamma) \neq \trv(\beta) \implies \ma_\gamma[i] \neq \ma_\beta[i]$. 
In keeping with our earlier notation, we will drop parentheses on singletons, and so write 
$\mb_\gamma \mbox{ for } \mb_{\{\gamma \}}$.  Note that each $\mb_\gamma = 1_\ba$. 
\end{defn}

\begin{lemma}[Realizing random graph types]  \label{claimtrg}   
Suppose $\mfa \in \AP$ and let $\theta \leq \lambda$. 
Suppose $\bar{\mb} = \langle \mb_u : u \in [\theta]^{<\aleph_0} \rangle$  
is a problem for the theory of the random graph, so a sequence of elements of $\ba^+_\mfa$, indeed of $\de_\mfa$. 
Then there is $\mfb \in \AP$ such that 

\begin{enumerate}
\item $\mfa \leq_{\AP} \mfb$, and 
\item there is a solution of $\bar{\mb}$ in $\mfb$, i.e. there are $\mb^1_\alpha \in \de_\mfb$ for $\alpha < \theta$\\  such that 
$\ba_\mfb \models \bigcap \{ \mb^1_\alpha : \alpha \in u \} \leq \mb_u$ for $u \in [\theta]^{<\aleph_0}$. 
\end{enumerate}
\end{lemma}

\begin{proof}
Suppose we are given $\bar{\mb}$ along with the supporting data mentioned in \ref{d:rg-types}. 
We start by defining some auxiliary objects, which aim to keep track of equalities. 
For each $\gamma < \theta$, call 
$\mx \in \ba^+_\mfa$ \emph{collapsed} for $\gamma$ if for some $\beta \leq \gamma$, 
\begin{equation} \label{e:coll}
0 < \mx \leq \ma[a_\gamma = a_\beta] \mbox{ but for all $\delta < \beta$ }, ~\mx \cap \ma[a_\gamma = a_\delta] = 0_{\ba_\mfa}.
\end{equation}  
Note that for every $\gamma < \theta$ and $\ma \in \ba^+_\mfa$ there exists $\mx$ with $0 < \mx \leq \ma$ 
which is collapsed for $\gamma$, because the ordinals are 
well ordered.  
So we may try to choose for each $\gamma < \theta$ a maximal antichain [supporting $\mb_\gamma$, though this is trivial as 
each $\mb_\gamma = 1_{\ba_\mfa}$] consisting of disjoint nonzero elements $\ma_{\gamma, \epsilon}$
each of which is collapsed for $\gamma$, by induction on $\epsilon < \kappa$.
As $\ba_\mfa$ satisfies the $\kappa$-c.c. we stop at some ordinal below $\kappa$. 
Renumbering, we may assume this maximal antichain is indexed by a cardinal $\mu_\gamma < \kappa$.
Write the result as:
\begin{equation}
\label{antichain}
\langle (\ma_{\gamma, \epsilon}, \beta_{\gamma, \epsilon} ) : \epsilon < \mu_\gamma \rangle 
\end{equation}
where $\beta_{\gamma, \epsilon}$ is the minimal $\beta \leq \gamma$ such that $\ma_{\gamma, \epsilon} \leq 
\ma[a_\gamma = a_\beta]$, in the sense of (\ref{e:coll}).  
Without loss of generality, we may assume that 
$\epsilon < \epsilon^\prime$ implies $\beta_{\gamma, \epsilon} \neq \beta_{\gamma, \epsilon^\prime}$: if not, combine all 
elements of the antichain with the same $\beta$ and renumber.

On the relation of these antichains to each other:  fixing $\gamma < \theta$, suppose $\beta = \beta_{\gamma, \epsilon}$. 
Observe that there is one and only one $\delta < \mu_\beta$ such that $\ma_{\gamma, \epsilon} \cap \ma_{\beta, \delta} > 0$. 
[There is at least one such $\delta$ because $\langle \ma_{\beta, \delta} : \delta < \mu_\beta \rangle$ is a maximal antichain.
So 
$0 < \ma_{\gamma, \epsilon} \cap \ma_{\beta, \delta} \leq  \ma[a_\gamma = a_{\beta_{\gamma, \epsilon}}] \cap 
\ma[ a_\beta = a_{\beta_{\beta, \delta}} ]$ but since $\beta_{\gamma, \epsilon}$ is minimal and $\beta_{\beta, \delta} \leq \beta \leq \beta_{\gamma, \epsilon}$, necessarily they are all equal. 
There is only one such $\delta$ by the last sentence of the previous paragraph.]
So necessarily $\ma_{\gamma, \epsilon} \leq \ma_{\beta, \delta}$. 
For $\gamma < \theta$, $\epsilon < \mu_\gamma$ let us write 
$\zeta_{\gamma, \epsilon} < \mu_{\beta_{\gamma, \epsilon}}$ 
for this unique value so that 
$\ma_{\gamma, \epsilon} \leq \ma_{\beta_{\gamma, \epsilon}, \zeta_{\gamma, \epsilon}}$. 
Note moreover that 
\begin{equation}
\label{e-b-g} 
\mbox{ if $\beta_{\gamma_1, \epsilon_1} = \beta_{\gamma_2, \epsilon_2} = \beta$ (for $\gamma_1$ not necessarily equal to $\gamma_2$) 
then $\zeta_{\gamma_1, \epsilon_1} = \zeta_{\gamma_2, \epsilon_2}$ }
\end{equation}
by a similar argument, since there is at most one $\zeta < \mu_\beta$ such that $\beta_{\beta, \zeta} = \beta$.  

To summarize, we may write the elements of (\ref{antichain}) as triples: 
\begin{equation}
\label{eq55}
\langle (\ma_{\gamma, \epsilon}, \beta_{\gamma, \epsilon}, \zeta_{\gamma, \epsilon} ) : \epsilon < \mu_\gamma \rangle. 
\end{equation}
where we can informally 
express this information as saying that on $\ma_{\gamma, \epsilon}$, $a_\gamma$ collapses to $a_\beta$ where 
$\beta = \beta_{\gamma, \epsilon}$, and $\zeta = \zeta_{\gamma, \epsilon}$ tells us $\ma_{\gamma, \epsilon}$ is below the $\zeta$-th
element of the antichain for $\beta$. 

\br
\noindent 
We now define the Boolean algebra $\ba_\mfb$.   Let $\{ \mc_{\gamma, \epsilon} : \gamma < \theta, \epsilon < \mu_\gamma \}$ 
be new elements. 
Let $\ba_\mfb$ be the completion of the Boolean algebra $\ba^0_\mfb$ generated by 
\[ \mathcal{X} = \ba_\mfa \cup \{ \mc_{\gamma, \epsilon} : \gamma < \theta, \epsilon < \mu_\gamma \} \]
freely except:
\begin{enumerate}
\item[(i)] the equations in $\ba_\mfa$
\item[(ii)] $\mc_{\gamma, \epsilon_1} \cap \mc_{\gamma, \epsilon_2} = 0$ when $\epsilon_1 < \epsilon_2 < \mu_\gamma$
\item[(iii)]  $\mc_{\gamma_1, \epsilon_1} \cap \mc_{\gamma_2, \epsilon_2} = 0$ when 
$\epsilon_1 < \mu_{\gamma_1}$, $\epsilon_2 < \mu_{\gamma_2}$,  
 $\trv(\gamma_1) \neq \trv(\gamma_2)$ and $\beta_{\gamma_1, \epsilon_1} = \beta_{\gamma_2, \epsilon_2}$. 
\end{enumerate}
Let us verify that without loss of generality $\ba_\mfa \subseteq \ba_\mfb$. 
Let $h$ be the identity on $\ba_\mfa$ and map $\mc_{\gamma, \epsilon}$ (for $\gamma < \theta$, $\epsilon < \mu_\gamma$) to $0_{\ba_\mfa}$. 
So $h$ is a function from $\mathcal{X}$ onto $\ba_\mfa$ respecting the equations, hence has an extension $\hat{h}$ which is a  homomorphism from $\ba_\mfb$ into $\ba^0_\mfa$ such that $\hat{h} \rstr \ba_\mfa = \operatorname{id}_{\ba_\mfa}$, so we are done. 

Let us verify that  $\ba_\mfa \lessdot \ba_\mfb$.  
Let $\mc \in \ba^+_\mfb$ and we shall find a projection from $\ba_\mfb$ onto $\ba_\mfa$ mapping $\mc$ to a positive element.  
We can replace $\mc$ by any $\mc^\prime \leq \mc$ which is not zero. As the generators are dense in the completion,\footnote{Informally, we can find a positive element below $\mc$ which is the intersection of some $\ma \in \ba^+_\mfa$, 
some number of $\mc_{\gamma, \epsilon}$'s which appear positively, and some number of $\mc_{\gamma, \epsilon}$'s which appear 
negatively, but remembering that each $\{ \mc_{\gamma, \epsilon} : \epsilon < \mu_\gamma \}$ is an antichain, among the elements which appear positively there can be no more than one from each antichain.} 
without loss of generality for some $\ma \in \ba^+_{\mfa}$, finite $u, v \subseteq \theta$, a function $f$ with 
finite domain $u \subseteq \theta$ 
such that $f(\gamma) < \mu_\gamma$ for $\gamma \in u$, and finite $w_\beta \subseteq \mu_\beta$ for $\beta \in v$, we have 
\begin{equation}
\label{eq125a} 0 < \bigcap \{  \mc_{\gamma, f(\gamma)} : \gamma \in u \} \cap \bigcap \{ - \mc_{\beta, \epsilon} : \beta \in v, \epsilon \in w_\beta \} \cap \ma \leq \mc  
\end{equation}
where necessarily $\{ (\gamma, f(\gamma)) : \gamma \in u \} \cap \{ (\beta, \epsilon) : \beta \in v, \epsilon \in w_\beta \} = \emptyset$, and if $\gamma_1 \neq \gamma_2 \in u$ and $\trv(\gamma_1) \neq \trv(\gamma_2)$ then 
 $\beta_{\gamma_1, f(\gamma_1)} \neq  \beta_{\gamma_2, f( \gamma_2)}$.
 If not, $\mc_{\gamma_1, f(\gamma_1)} \cap \mc_{\gamma_2, f(\gamma_2)} = 0$, which contradicts the intersection being positive.]
Define $h : \mathcal{X} \rightarrow \ba_\mfa$ so that 
$h$ is the identity on $\ba_\mfa$, 
$ h(\mc_{\gamma, f(\gamma)}) = \ma $  for $\gamma \in u $,  and
$h(\mc_{\gamma, \epsilon}) = 0_{\ba_\mfa} \mbox{ for }(\gamma, \epsilon) \notin \{ (\gamma, f(\gamma)) : \gamma \in u \}.$ 
By the note after equation (\ref{eq125a}), $h$ respects the equations in (iii). So 
$h$ extends to a homomorphism $\hat{h}$ from $\ba_\mfb$ onto $\ba_\mfa$ which indeed is the 
identity on $\ba_\mfa$ and sends $\mc$ to a positive element.

\br
\noindent
Now we define our multiplicative refinement. First, 
 for $\gamma < \theta$,  $\epsilon < \mu_\gamma$ let 
$\mb^1_{\gamma, \epsilon} = \ma_{\gamma, \epsilon} \cap \mc_{\gamma, \epsilon}$.  
Then for every $\gamma < \theta$ let 
\begin{equation} \mb^1_\gamma = \bigcup \{ \mb^1_{\gamma, \epsilon} : \epsilon < \mu_\gamma \}. 
\end{equation}
We have to check first that 
\begin{equation}
\label{first-step} u \in [\theta]^{<\aleph_0} \implies \bigcap_{\gamma \in u} \mb^1_\gamma \leq \mb_u 
\end{equation}
and second that there is an ultrafilter $\de$ on $\ba_\mfb$ extending $\de_\mfa$ such that 
\begin{equation} 
\label{in-de}
\mb^1_\gamma \in \de ~\mbox{ \hspace{3mm} for each $\gamma < \theta$. }  
\end{equation}
First we check equation (\ref{first-step}). 
It will suffice to consider $u = \{ \gamma_1, \gamma_2 \}$, $\gamma_1 \neq \gamma_2$. 
[If $u = \{ \gamma \}$ remember that $\mb_\gamma = 1_\ba$. If $|u| > 2$ remember that $\mb_u = 
\bigcap_{v \subseteq u, |v| = 2} \mb_v$, since consistency of the corresponding formulas depends only on instances of equality.]  
So it suffices to prove that for $(\epsilon_1, \epsilon_2) \in \mu_{\gamma_1} \times \mu_{\gamma_2}$, 
\begin{equation} \mb^1_{\gamma_1, \epsilon_1} \cap \mb^1_{\gamma_2, \epsilon_2} \leq \mb_{\{ \gamma_1, \gamma_2\} }. 
\end{equation}
There are three cases.\footnote{\emph{Case 1}. We connect to 
both $a_{\gamma_1}$ and $a_{\gamma_2}$ or neither, so there is no conflict.  \emph{Case 2}. The truth values are different but 
they collapse to different elements, so there is no conflict. \emph{Case 3.} 
The truth values are different, they collapse to the same $\beta$ (and so, by (\ref{e-b-g}), are contained in the same piece of the 
antichain for $\beta$). The $\mc$'s in this case do not intersect, so there is no conflict.} 
\emph{Case 1}.
$\trv(\gamma_1) = \trv(\gamma_2)$.  Then $\mb_{\{\gamma_1, \gamma_2\}} = 1_\ba$, so the inequality holds trivially. {\emph{Case 2}.}
$\trv(\gamma_1) \neq \trv(\gamma_2)$, $\beta_{\gamma_1, \epsilon_1} \neq \beta_{\gamma_2, \epsilon_2}$. Then $\ma_{\gamma_1, \epsilon_1} \cap \ma_{\gamma_2, \epsilon_2} \leq \ma[ a_{\gamma_1} \neq a_{\gamma_2} ] $ which suffices.
\emph{Case 3}. 
$\trv(\gamma_1) \neq \trv(\gamma_2)$, $\beta_{\gamma_1, \epsilon_1} = \beta_{\gamma_2, \epsilon_2}$ (call it $\beta$).
Then  
\[ \mb^1_{\gamma_1, \epsilon_1} \cap \mb^1_{\gamma_2, \epsilon_2} \leq \mc_{\gamma_1, \epsilon_1} \cap \mc_{\gamma_2, \epsilon_2} =  0_{\ba} \] 
by clause (iii) in the definition of $\ba_\mfb$.
This completes the verification of (\ref{first-step}).

\br \br
Next we verify equation (\ref{in-de}). 
For this we should show there is an ultrafilter $\de$ on $\ba_\mfb$ extending $\de_\mfa \cup \{ \mb^1_\gamma : \gamma < \theta \}$. 
As in \ref{ext1}, it suffices to prove that given a finite $u \subseteq \theta$ and $\md \in \de_\mfa$, 
\begin{equation} \mb^1_u = \bigcap \{ \mb^1_\gamma : \gamma \in u \} \mbox{  is not disjoint to } \md. 
\end{equation}
As $\mb_u \in \de_\mfa$, 
$\mb_u \cap \md > 0$ so we may assume without loss of generality that $\md \leq \mb_u$. 
Enumerate $u = \langle \gamma_i : i < |u| \rangle$.
Now as each $\langle \ma_{\gamma_i, \epsilon} : \epsilon < \mu_{\gamma_i} \rangle$ is a maximal antichain of $\ba_\mfa$,
we may choose by induction on $i$ a function $f$ with domain $u$ 
such that for each $i$, $f(i) < \mu_{\gamma_i}$ and $\ma_{\gamma_i, f(i)} \cap \bigcap_{j<i} \ma_{\gamma_j, f(j)} \cap 
\md > 0$.  Let $\ma_f$ denote  $\bigcap_{i < |u|} \ma_{\gamma_i, f(i)}$. Since $\md \leq \mb_u$, we know 
$\ma_f \cap \mb_u > 0$, \emph{thus}, in fact, $\ma_f \leq \mb_u$. [Why? By definition, $\mb_u$ only depends on information about 
collisions, which remain constant on $\ma_f$ by construction of our maximal antichains.] 

Let $\md_* = \md \cap \ma_f = \md \cap \ma_f \cap \mb_u > 0$. 
Since $\md_* \leq \ma_f \leq \mb_u$ (and recall $u = \dom(f)$) for any $\gamma_1, \gamma_2 \in u$ we have that $\gamma_1 \neq \gamma_2$ implies either 
$\trv(\gamma_1) = \trv(\gamma_2)$ or $\beta_{\gamma_1, f(\gamma_1)} \neq \beta_{\gamma_2, f(\gamma_2)}$.  
It follows that the equations in the definition of $\ba^0_\mfb$ permit   
\[ \md_* \cap \{ \mc_{\gamma, f(\gamma)} : \gamma \in u \} > 0. \]
Now define $h : \mathcal{X} \rightarrow \ba_\mfa$
to be the identity on $\ba_\mfa$, 
$ h(\mc_{\gamma, f(\gamma)}) = \md_*$  for $\gamma \in u $,  and
$h(\mc_{\gamma, \epsilon}) = 0_{\ba_\mfa} \mbox{ for }(\gamma, \epsilon) \notin \{ (\gamma, f(\gamma)) : \gamma \in u \}$. 
This $h$ extends to a homomorphism $\hat{h}$ from $\ba_\mfb$ onto $\ba_\mfa$ which is the identity on $\ba_\mfa$ and sends $\md$ 
to a positive element.  Thus, a suitable ultrafilter exists; fix one.  
This completes the construction of $(\ba_\mfb, \de_\mfb)$ and the verification that $\langle \mb^1_\gamma : \gamma < \theta \rangle$ 
is a solution for $\bar{\mb}$ there.

\vspace{5mm}

It remains to verify that the pair $(\ba_\mfa, \ba_\mfb)$ satisfies the $\kim$-pattern transfer property. 
Suppose we are given $\langle \ma^2_\alpha : \alpha < \kappa \rangle$ a sequence of positive elements of $\ba_\mfb$. 
Similarly to the earlier proof that $\ba_\mfa \lessdot \ba_\mfb$, 
as the generators are dense in the completion, we may find for each 
$\alpha < \kappa$ an $\ii_\alpha = ( \mx_\alpha, u_\alpha, f_\alpha, v_\alpha, \overline{w}_\alpha = \langle 
w_{\alpha, \beta} : \beta \in v_\alpha \rangle)$ such that 
 $\mx_\alpha \in \ba^+_{\mfa}$, $u_\alpha, v_\alpha$ are finite subsets of $\theta$, $f_\alpha$ is a function with 
finite domain $u_\alpha \subseteq \theta$ 
such that $f_\alpha(\gamma) < \mu_\gamma$ for $\gamma \in u_\alpha$; $w_{\alpha, \beta}$ is a finite subset of $\mu_\beta$ for $\beta \in v_\alpha$; we have that 
 $\{ (\gamma, f_\alpha(\gamma)) : \gamma \in u_\alpha \} \cap \{ (\delta, \epsilon) : \delta \in v_\alpha, \epsilon \in w_{\alpha, \beta} \} = \emptyset$, and if $\gamma_1 \neq \gamma_2 \in u_\alpha$, $\trv(\gamma_1) \neq \trv(\gamma_2)$, then 
 $\beta_{\gamma_1, f_\alpha(\gamma_1)} \neq  \beta_{\gamma_2, f_\alpha( \gamma_2)}$; 
and together they satisfy 
\begin{equation}
\label{eq125} 
\ba_\mfb \models 0 <  \mx_\alpha \cap \bigcap \{  \mc_{\gamma, f_\alpha(\gamma)} : \gamma \in u_\alpha \} 
\cap \bigcap \{ - \mc_{\beta, \epsilon} : \beta \in v_\alpha, \epsilon \in w_{\alpha, \beta} \}  \leq \ma^2_\alpha    
\end{equation}

\noindent Next we will want to use the $\Delta$-system lemma to smooth out collisions. Towards this, define $B_\alpha = \{ \beta_{\gamma, f(\gamma)} : \gamma \in u_\alpha \}$, and define $g_\alpha : B_\alpha \rightarrow \{ 0, 1 \}$ to be the function given by 
$\beta_{\gamma, f(\gamma)} \mapsto \trv(\gamma)$ for $\gamma \in u_\alpha$. [We have to justify why this is a 
function:  if $\gamma_1 \neq \gamma_2 \in u_\alpha$, $\trv(\gamma_1) \neq \trv(\gamma_2)$, then 
 $\beta_{\gamma_1, f_\alpha(\gamma_1)} \neq  \beta_{\gamma_2, f_\alpha( \gamma_2)}$ otherwise the intersection in 
 equation (\ref{eq125}) would be empty.]

Then, by the $\Delta$-system lemma, we may find $\uu \in [\kappa]^\kappa$ along with  
$u_*$, $f_*$, $B_*$, 
$g_*$ satisfying: 
\begin{itemize}
\item $\langle u_\alpha : \alpha \in \uu \rangle$ is a $\Delta$-system with heart $u_*$.
\item for each $\alpha \in u$, $f_\alpha \rstr u_* = f_*$  $($recall $f_\alpha(\gamma)$ can take fewer than $\kappa$ values$)$. 
\item $\langle B_\alpha : \alpha \in \uu \rangle$ is a $\Delta$-system with heart $B_*$.
\item for each $\alpha \in u$, $g_\alpha \rstr B_* = g_*$.  
\end{itemize} 
Let us show that for any finite $u \subseteq \uu$, letting $w = \bigcup_{\alpha \in u} u_\alpha$, we have that 
\begin{equation}
\label{e:pt}
\bigcap_{\gamma \in w} \mc_{\gamma, f_\alpha(\gamma)} = \bigcap \{ \mc_{\gamma, f_\alpha(\gamma)} : \alpha \in u, \gamma \in u_\alpha \} > 0.  
\end{equation}
Why? It suffices to show that if $\alpha_1, \alpha_2 \in u$, $\gamma_1 \in u_{\alpha_1}, \gamma_2 \in u_{\alpha_2}$, 
then $\mc_{\gamma_1, f_{\alpha_1}}(\gamma_1) \cap \mc_{\gamma_2, f_{\alpha_2}}(\gamma_2) > 0$. 
First suppose $\alpha_1 = \alpha_2$; then this follows from equation (\ref{eq125}).  Next, 
suppose $\alpha_1 \neq \alpha_2$ but $\gamma_1 = \gamma_2$. This can happen only if 
$\gamma := \gamma_1 =\gamma_2$ belongs to $u_*$, in which case $f_{\alpha_1}(\gamma) = 
f_*(\gamma) = f_{\alpha_2}(\gamma)$ and 
there is no problem from equation (ii). 
Moreover, necessarily $\beta_{\gamma, f_{\alpha_1}}(\gamma) =  
\beta_{\gamma, f_{\alpha_2}}(\gamma)$, so this value, call it $\beta$, must belong to $B_*$ and 
$g_{\alpha_1}(\beta) = g_*(\beta) = g_{\alpha_2}(\beta)$ thus $\trv(\gamma_1) = \trv(\gamma_2)$ and  
there is no problem from (iii). 
Finally, suppose $\alpha_1 \neq \alpha_2$ and $\gamma_1 \neq \gamma_2$.  Equation (ii) is irrelevant. 
As before, if $\beta_{\gamma_1, f_{\alpha_1}}(\gamma_1) =  
\beta_{\gamma_2, f_{\alpha_2}}(\gamma_2)$, then this value, call it $\beta$, must belong to $B_*$ and 
$g_{\alpha_1}(\beta) = g_*(\beta) = g_{\alpha_2}(\beta)$ thus $\trv(\gamma_1) = \trv(\gamma_2)$ and there is again no problem from 
equation (iii).  This completes the proof of (\ref{e:pt}). 

Just as in \ref{smooth}(4), this tells us that for a finite $u \subseteq \uu$, if $\bigcap_{\alpha \in u} \mx > 0$, then 
\[ 0 <  \bigcap_{\alpha \in u}  \mx_\alpha \cap \bigcap \{  \mc_{\gamma, f_\alpha(\gamma)} : \gamma \in u_\alpha \} 
\cap \bigcap \{ - \mc_{\beta, \epsilon} : \beta \in v_\alpha, \epsilon \in w_{\alpha, \beta} \} . \]
Recall that by \ref{9.14f}, this suffices for pattern transfer, so finishes the proof. 
\end{proof}

\br

Recall that $\kappa$ is a regular uncountable cardinal and $\lambda \geq \kappa$. 

\begin{concl} \label{9.28k} 
There exists a regular ultrafilter $\de$ on $\lambda$ such that: 

\begin{enumerate}

\item $\de$ is $\kappa$-good, i.e. $(\mathbb{N}, <)^\lambda/\de$ is $\kappa$-saturated.\footnote{Equivalently, for 
any model $M$ with $|\tau_M| < \kappa$, $M^\lambda/\de$ is $\kappa$-saturated.} 

\item if $\xm \in \mcm$ then $\de$ is $(\lambda^+, T_\xm)$-good.

\item if $\xn \in \mcn$ then $\de$ is not $(\kappa^+, T_\xn)$-good.
\end{enumerate}
\end{concl}

\begin{proof} 
Let $\langle S_\gamma : \gamma < 2^\lambda \rangle$ partition $2^\lambda$ into sets each of cardinality $2^\lambda$ with 
$\gamma \leq \min(S_\gamma)$. 
We choose by induction on $\alpha \leq 2^\lambda$ not only $\mfb_\alpha$ but also $\bar{f}_\alpha$ such that:

\begin{enumerate}
\item[(1)] $\bar{\mfb} = \langle \mfb_\alpha : \alpha \leq 2^\lambda \rangle$ is a $\leq_{\AP}$-increasing continuous general construction sequence, 
so by definition and our work above this will mean: 
\begin{enumerate}
\item $\mfb_0 = \mfa_*$, 
\item each $\mfb_\alpha \in \AP$, 
\item $\alpha < \beta \implies (\ba_\mfa, \ba_\mfb)$ satisfies the $(\kappa, \mci, \bar{m})$-pattern transfer property, 
\item each $|\ba_{\mfb_\alpha}| \leq 2^\lambda$,   
\item each $\ba_{\mfb_\alpha}$ satisfies the $\kim$-c.c. 
\end{enumerate}
\item[(2)] $\bar{f}_\alpha = \langle f_{\alpha, \gamma} : \gamma \in S_\alpha \rangle$
\item[(3)] $\bar{f}_\alpha$ lists  $\mcf_\alpha = \{ f : f $ is a function from $[\theta]^{<\aleph_0}$ into 
$\de_{\mfb_\alpha}$ which is monotone for some $\theta = \theta_f \leq \lambda$ $\}$. 

\item[(4)] Now if $\alpha = \gamma+1$ and $\gamma \in S_{\beta}$, so  $\beta \leq \gamma$, 
and $\operatorname{range}(f_{\alpha, \gamma}) \subseteq 
\de_{\mfb_\gamma}$, 
and \emph{if} maintaining the restriction in (1) we can choose $\mfb_\alpha$ such that $f_{\alpha, \gamma}$ has 
a multiplicative refinement $g_\gamma$ with range $\subseteq \de_{\mfb_\alpha}$, \emph{then} there is such a refinement [i.e., 
then we do so].  Otherwise, we do nothing.  

\emph{Alternately}, in this step we could say that if $f_{\alpha, \gamma}$ is a possibility pattern for a theory $T = T_{f_{\alpha, \gamma}}$ such that 
one of the following occurs: 
\begin{itemize}
\item $\theta_{f_{\alpha, \gamma}} < \kappa$, 
\item $T_{f_{\alpha, \gamma}}$ is $T_\xm$ for $\xm \in \mcm$, and the possibility pattern in question comes from a positive $R$-type together with a single formula of the form $Q_{\nu_*}(x)$ for some $\nu_* \in \mct_1$ $($or: together with a single formula of the form 
$P_{\eta_*}(y)$ for some $\eta_* \in \mct_2$$)$ 
\item $T_{f_{\alpha, \gamma}}$ is $\trg$ and the possibility pattern comes from a type in positive and negative instances of $R$.
\end{itemize}
then we solve it using \ref{10.12A}, \ref{9.28a}, or \ref{claimtrg}, otherwise we do nothing.  
Note that  \ref{10.12A}, \ref{9.28a}, \ref{claimtrg} show that at the very least, these 
three kinds of types will be handled if we take the first alternative. 
\end{enumerate} 
Having built our Boolean algebra and ultrafilter, 
we choose the data of separation of variables $\de_0, \jj$ to go with this $\ba_\mfa$ and $\de_\mfa$ in the sense of 
\ref{d:built}, which gives us the 
regular ultrafilter desired, recalling Theorem \ref{t:separation}. 
Note that by our analysis of ultrapower types in \ref{concl123}, these $R$-types suffice for saturation for $T_\xm$, 
and by \ref{9.14aa} the fact that our Boolean algebra $\ba_\mfa$ has the $\kim$-c.c.  suffices for non-saturation of $T_\xn$ for $\xn \in \mcn$. 
\end{proof}

\vspace{5mm}
\section{Main results and the big picture} \label{s:main}

In this section we summarize the main consequences of our construction for Keisler's order. 

Recall that we can easily form new countable theories as the ``disjoint union,'' or sum, of up to countably many countable theories. More formally: 

\begin{defn} \label{d:sum}
Recall $T = \sum T_n$ when without loss of generality $\langle \tau(T_n) : n < \omega \rangle$ are pairwise disjoint and have only 
predicates, and we have countably many new unary predicates $\{ P_n : n < \omega \}$ with 
$\{ P_n : n < \omega \}\cap \tau(T_n) = \emptyset$ for each $n$. 
Then $M \models T$ iff $\langle P^M_n  : n<\omega \rangle$ are pairwise disjoint and $( M \rstr P^M_n) \rstr \tau(T_n) \models T_n$, 
and if $R \in \tau(T_n)$ then $R^M \subseteq {^{\operatorname{arity}(R)}(P^M_n)}$. 
\end{defn}

In this notation: 

\begin{cor}
Continuing with the objects of $\ref{9.28k}$, given any $\xm_\ell \in \mcm$ for $\ell < \omega$, 
we have that $T_* = \sum_{\ell < \omega} T_{\xm_\ell}$ is a complete countable simple theory, 
and $\de$ is $(\lambda^+, T_*)$-good. 
\end{cor}

Recall that $u \in [X]^{\leq \aleph_0}$ means that $u$ is an at most countable subset of $X$. Given any fast sequence 
and a family of independent level functions, our construction above gives
a set of continuum many theories which can be thought of as Keisler-independent in the following strong sense.  
Suppose from the basic theories we form all possible ``small composite theories'' (countable, complete theories 
formed in the natural way as disjoint unions of at most countably many theories from our original set). Then our construction has shown that 
even the composite theories interact as freely as possible in the sense of Keisler's order, reflecting only the interaction in their indices, as the 
next theorem makes precise.

\begin{theorem} \label{c920}
We can find $\bar{T}$ such that: 

\begin{enumerate}

\item $\bar{T} = \langle T_u : u \in [2^{\aleph_0}]^{\leq \aleph_0} \rangle$

\item $T_u$ is a complete first order countable simple theory with trivial forking

\item $T_u \tlf T_v$ if and only if $u \subseteq v$, for $u, v \in [2^{\aleph_0}]^{\leq \aleph_0}$. 

\item if $W \subseteq 2^{\aleph_0}$, $\aleph_0 \leq {{\mu}} < {{\mu}}^+ \leq \kappa = \cf(\kappa) \leq \lambda$  then there is a 
regular ultrafilter $\de$ on $\lambda$ such that:

\begin{enumerate}
\item $\de$ is $\kappa$-good 
\item if $u \in [W]^{\leq \aleph_0}$ then $\de$ is $(\lambda^+, T_u)$-good

\item if $u \in [2^{\aleph_0}]^{\leq \aleph_0}$, $u \not\subseteq W$ then $\de$ is not $(\kappa^+, T_u)$-good. 

\end{enumerate}

\end{enumerate}
\end{theorem}

\begin{proof}
Let $\bar{m}$ be a fast sequence, let  $\bar{E}$ and $\Xi$ satisfy the hypotheses of \ref{4.8}, and let 
$ \sM_* = \{ \xm_\alpha = \prm[\bar{m}, \bar{E}, \xi_\alpha] : \alpha < 2^{\aleph_0} \} $ be as defined in \ref{4.8}. 
Let $W$ play the role of $\mcm \subseteq \sM_*$.   Then by our construction there is a regular ultrafilter 
$\de$ on any infinite $\lambda$ as in (4), satisfying also (4)(a), such that $\de$ is good for every $T_\xm$, $\xm \in \mcm$, and 
$\de$ is not good for any $T_\xn$, $\xn \in \sM_* \setminus \mcm$.  Note that any sum $T_u$ of theories 
$\{ T_{\xm_\alpha} : \alpha \in u \}$ is complete and remains simple with trivial forking. Clearly (3) and (4)(b), (c) follow. 
\end{proof}

As immediate consequences, we obtain several ``more quotable'' theorems. 

\begin{theorem}
There exist a perfect set of incomparable theories in Keisler's order, in ZFC, already within the class of countable simple unstable theories whose only 
forking comes from equality. 
\end{theorem}

\begin{proof}
Use any set of countably many pairwise incomparable theories from Theorem \ref{c920} to 
label the nodes of an infinite binary branching tree, and consider the 
set of theories which correspond to disjoint unions of theories along a given branch. 
\end{proof}

\begin{theorem}
In the $($countable, complete$)$ simple non low theories, 
\begin{enumerate}
\item there is a chain of cardinality continuum, and 
\item there is an antichain of cardinality continuum
\end{enumerate}
in Keisler's order.
\end{theorem}

\begin{proof}
Enumerate a countable subset of the theories from Theorem \ref{c920} as $\langle T_q : q \in \mathbb{Q} \rangle$.  
Let $\langle C_\beta = (C^0_\beta, C^1_\beta) : \beta < 2^{\aleph_0} \rangle$ 
enumerate the continuum many cuts of the rationals. 
Let $T_\beta$ be the theory which is the disjoint union of $\{ T_q : q \in C^0_\beta \}$. 
Then clearly $\alpha \leq \beta$ implies $T_\alpha \tlf T_\beta$.  As the rationals 
are dense in the reals, and our family of theories are pairwise 
incomparable, if $\alpha < \beta$ then $T_\alpha \triangleleft T_\beta$.  This gives us a 
chain of cardinality continuum. 

For an antichain of cardinality $\mathfrak{c}$, use all the theories from Theorem \ref{c920}. 
\end{proof}

\begin{rmk}
Even in this frame the theories involved seem simple.  
\end{rmk}

\begin{theorem}
There exist continuum many complete, countable, simple theories whose only forking comes from equality, and which are pairwise incomparable in Keisler's order.   
\end{theorem}

\begin{proof}
Consider all $u$ of size $1$ in Theorem \ref{c920}, i.e., in the notation of that proof, consider $\{ T_{\xm_\alpha} : \alpha < 2^{\aleph_0} \}$.  
\end{proof}

\begin{concl}  \label{c:psubsets}
{Keisler's order is not at all simple, indeed the partial order on the family of subsets of $\mathbb{N}$ is embeddable into it.} 
\end{concl}

\br

In fact, with more work we can upgrade \ref{c:psubsets} to the following a priori much stronger result. 
We include the proof in the next section. 

\begin{theorem} \label{thm-mcp}
$\mcp(\omega)/\operatorname{fin}$ embeds into Keisler's order. 
\end{theorem}

\begin{proof} 
See page \pageref{p-omega}, below.
\end{proof}

The theorem 
whose proof concludes the next section explains: 

\begin{theorem} \label{thm-ideals}
There is a family of parameters $\{ \xm[A] : A \subseteq \omega \}$  such that each $T_{\xm[A]}$ is countable, complete, simple, and low, 
and the following are equivalent for any $\lambda \geq 2^{\aleph_0}$ and any set 
$\mcx \subseteq \mcp(\omega)$:
\begin{enumerate}
\item There exists a regular ultrafilter $\de$ on $\lambda$ such that $\mcx = \{ A \subseteq \omega : \de $ is $(\lambda^+, T_{\xm[A]})$-good $\}$. 
\item $\mcx \supseteq [\omega]^{<\aleph_0}$ is an ideal. 
\end{enumerate}
\end{theorem}

\begin{proof}
See page \pageref{proof-ideals}, below. 
\end{proof}

\begin{disc} \emph{In light of our main theorems, it would be natural to consider strengthenings of Keisler's order, for instance, by restricting 
consideration to ultrafilters on $\lambda$ which are $\lambda$-good [but of course not necessarily $\lambda^+$-good].  However, our construction can be carried out even under such a definition, as 
we now spell out.}
\end{disc}

\begin{defn}
Let the \emph{strong Keisler order} mean that $T_1 \tlf^s T_2$ if and only if for every infinite $\lambda$, for every regular ultrafilter on $\lambda$ 
\emph{which is 
$\lambda$-good}, for every $M_1 \models T_1$, and every $M_2 \models T_2$, if 
$(M_2)^\lambda/\de$ is $\lambda^+$-good, then $(M_1)^\lambda/\de$ is $\lambda^+$-good. 
\end{defn}

\begin{theorem} \label{t:11-12}
There exists a family $\{ \mct_{\xm_\alpha} : \alpha < 2^{\aleph_0} \}$ of continuum many complete countable simple theories, with forking coming only from equality, such that for any $u,v \subseteq 2^{\aleph_0}$ and $u,v$ and most countable, letting $T_u, T_v$ denote the ``disjoint unions'' of 
$\{ T_{\xm_\alpha} : \alpha \in u \}$, $\{ T_{\xm_\beta} : \beta \in v \}$ respectively, we have that $T_u \tlf^s T_v$  in the strong Keisler order if 
and only if $u \subseteq v$. 
\end{theorem}

\begin{proof}
This is a simplified version of Theorem \ref{c920} in the case ${{\mu}} < {{\mu}}^+ = \kappa = \lambda$. 
\end{proof}

\begin{rmk} 
Note that Theorem $\ref{t:11-12}$ deals with the case where we quantify over all $\lambda$. We could ask the analogous question 
$\tlf^s_\lambda$, restricting to ultrafilters on $\lambda$. The proof of $\ref{t:11-12}$ shows that $\ref{t:11-12}$ remains true for 
$\tlf^s_\lambda$ when $\lambda$ is any successor. 
If $\lambda$ is a regular limit, the construction should likewise go through using $\lambda = \kappa$ and Definition $\ref{newba}$.
If $\lambda$ is a singular limit cardinal, then any $\lambda$-good ultrafilter is $\lambda^+$-good. 
[This is a consequence of the fact that $Th(\mathbb{Q}, <)$ is maximal in Keisler's order, see 
\cite{Sh:c} IV.2.6. p. 337.]
Hence there is only one class under $\tlf^s_\lambda$. 
\end{rmk}

Since we know that Keisler's order reduces to the study of $\vp$-types, it was asked: 
\begin{old-prob}
Determine whether there exists a function  
associating to each complete countable theory $T$ a formula $\vp_T$ of the language of $T$ such that for any 
infinite $\lambda$, any regular ultrafilter $\de$ on $\lambda$, and any model $M \models T$, the ultrapower  
$M^\lambda/\de$ is $\lambda^+$-saturated if and only if it is $\lambda^+$-saturated for $\vp_T$-types.   
\end{old-prob}
Call such a function \emph{an assignment of formulas to theories witnessing Keisler's order}. 
The results of the present paper show no formula assignment can exist:  

\begin{cor}
There cannot exist an assignment of formulas to theories witnessing Keisler's order. 
\end{cor}

\begin{proof}
Continuing in the notation of Theorem \ref{c920}, let $\langle \xm_n : n < \omega \rangle$ be any countable sequence of 
distinct elements of $\mcm$ and let $T = \sum T_{\xm_n}$ be their disjoint union.  Following \ref{d:sum}, we may assume 
that without loss of generality $\{ \tau(T_{\xm_n}): n < \omega \}$ are pairwise disjoint and have only predicates. 
Assume for a contradiction that there exists a formula $\vp_T$ such that for any 
infinite $\lambda$, any regular ultrafilter $\de$ on $\lambda$, and any model $M \models T$, we would have 
$M^\lambda/\de$ is $\lambda^+$-saturated if and only if it is $\lambda^+$-saturated for $\vp_T$-types.   Fix $M$.  
The formula $\vp_T$, being finite, is a formula of the language of 
$\bigcup \{ \tau_{T_{\xm_n}} : n \in v \}$ for some finite $v \subseteq \omega$. 
By Theorem \ref{c920}, there is a regular ultrafilter $\de$ on any uncountable $\lambda$ which is $\lambda^+$-good 
for $\{ T_{\xm_n} : n \in v \}$ and not $\lambda^+$-good for $\{ T_{\xm_\ell} : \ell \in \omega \setminus v \}$.  
So in $M^\lambda/\de$ we realize all $\vp_T$-types over sets of size $\lambda$,  but omit some other $\psi$-type over some set of 
size $\lambda$. 
\end{proof}

\begin{disc}
However, it may be interesting to give a model-theoretic characterization of the theories $T$ for which such a $\vp_T$ exists. 
This includes all stable theories $($choose any formula with the finite cover property if one exists, if not choose 
any formula \cite{Sh:c} VI.5.2$)$, and all theories with $SOP_2$ \cite{MiSh:998}. 
\end{disc}

\vspace{5mm}

\section{Embedding $\mcp(\omega)/\operatorname{fin}$} \label{s:p-omega}

In this section we prove Theorem \ref{thm-mcp}: Keisler's order embeds $\mcp(\omega)/\operatorname{fin}$, and 
Theorem \ref{thm-ideals}.  

\begin{conv}
For this section, fix some fast sequence $\bar{m}$, and a sequence of graphs $\bar{E}$ which is good for it, just as in 
$\ref{d:32}$ above.   
\end{conv} 

\begin{conv} \label{c:A-B}
In this section, for each $A \subseteq \omega$, let $\xm[A]$ be the parameter $\prm[\bar{m}, \bar{E}, \xi]$ where 
$\xi: \omega \rightarrow \{ 0, 1\}$ satisfies $\xi^{-1}\{1\} = A$.  Let $T_{\xm[A]}$ be its associated theory. 
\end{conv}
Our proof will show that $T_{\xm[A]} \tlf T_{\xm[B]}$ if and only if $A \subseteq^* B$.  The proof will also show this for 
the interpretability order $\tlf^*$, not only Keisler's order $\tlf$.  [A recent introduction to $\tlf^*$ is in \cite{MiSh:1124}, \S 1. 
The equivalence relation induced by $\tlf^*$ is 
strictly finer than that induced by $\tlf$, see \cite{MiSh:1124} 2.12.]

It will be useful to make explicit a property of parameters which held in our main case above.
\footnote{In Definition \ref{d:forgetful}, whether $(\eta^\smallfrown \langle \ell_1 \rangle, \rho^\smallfrown \langle \ell_2 \rangle)$ is an edge 
in $\mcr$ depends only on the fact that $(\eta, \rho) \in \mcr$ and on $\ell_1, \ell_2$. 
There is no ``memory'' coming from the path in the tree which affects the graph used at this stage (just the level matters).} 
Note that \ref{d:forgetful}(1)(c) says that we can write, for each level $k$, a graph $\mcs_k$ such that the connections in $\mcr$ depend only 
on these graphs.  

\begin{defn} \emph{ }
\label{d:forgetful}
\begin{enumerate}
\item We say $\xm$ is forgetful when
\begin{itemize}
\item[(a)] $\xm$ is self-dual for transparency\footnote{so we don't really need $m^1$ and $m^2$ in item (b).}
\item[(b)] for each $\ii = 1,2$ and $k$ there is 
$m^\ii_k$ such that for all $\eta \in \mct_{\ii, k}$,\\  $(\forall j) ( \eta^\smallfrown {\langle j \rangle} \in \mct_{\ii, k+1} \iff j < m^\ii_k )$ 
\item[(c)] for each $k$ there is $S_k \subseteq m^1_k \times m^2_k$ such that
if $(\eta, \nu) \in \mct_{1,k} \times \mct_{2,k}$ then $(\eta, \nu) \in \mcr^\xm_k$ if and only if 
$(\forall \ell < k) ( \eta(\ell), \nu(\ell) ) \in S_\ell$. 
\end{itemize}

\item So if $\xm$ is forgetful and $k <\omega$, then $\xn = \xm \rstr_{ \geq k }$ satisfies: 

$m^\ii_n[\xn] = m^\ii_{n+k}[\xm]$, and this does not depend on $\ii$

$\mct_{\ii,n}[\xn] = \prod_{i<n} m^\ii_{k+i}[\xm]$, and 

$S_n [\xn] = S_{k+n}[\xm]$, which are each symmetric. 
\end{enumerate}
\end{defn}
\noindent
Informally, all nodes at level $k$ have the same number of immediate successors, and when extending edges to the next level $($which recall only happens between elements whose 
restrictions are connected at all earlier 
levels$)$ only the last value matters:

\begin{rmk} \label{rmk:last}
$\ref{d:forgetful}(1)(b)$ implies that if we are given $\eta \in \mct_{1,k}$ and $\rho \in \mct_{2,k}$
and $\eta^\prime \in \mct_{1,k}$ and $\rho^\prime \in \mct_{2,k}$, and $(\eta, \rho) \in \mcr$ and 
$(\eta^\prime, \rho^\prime) \in \mcr$, then for any $i < m^1_{k+1}$ and $j < m^2_{k+1}$, 
\[ (\eta^\smallfrown \langle i \rangle, \rho^\smallfrown \langle j \rangle) \in \mcr 
\iff
({\eta^\prime }~^\smallfrown \langle i \rangle, {\rho^\prime} ~^\smallfrown \langle j \rangle) \in \mcr. \]
\end{rmk}

So the parameter depends in the natural way on a sequence of graphs.  The parameters used in the main arguments 
above satisfied this definition.  The following easy consequence will also be useful. 

\begin{claim} \label{d:forget}
Suppose $\xm$ is a forgetful parameter, $\eta_* = {\eta_0}^\smallfrown \langle i \rangle ^\smallfrown \eta_\infty \in \lim(\mct^\xm_1)$, 
and $\rho_* = {\rho_0}^\smallfrown \langle j \rangle ^\smallfrown \rho_\infty \in \lim(\mct^\xm_2)$, where $\lgn(\eta_0) = \lgn(\rho_0) < \omega$. 
Suppose that for every $s < \omega$, $(\eta_* \rstr s, ~ \rho_* \rstr s) \in \mcr^\xm$. 
If we replace $i, j$ by $i^\prime, j^\prime$ respectively so that $({\eta_0}^\smallfrown \langle i^\prime \rangle, 
~{\rho_0}^\smallfrown \langle j^\prime \rangle) \in \mcr^\xm$, then it remains true that 
for every $s < \omega$, 
\[ (~ ({\eta_0}^\smallfrown \langle i^\prime \rangle ^\smallfrown \eta_\infty)~ \rstr s, ~ 
({\rho_0}^\smallfrown \langle j^\prime \rangle ^\smallfrown \rho_\infty) \rstr s) \in \mcr^\xm. \] 
\end{claim}

\begin{proof}
By definition of forgetful. 
\end{proof}

Next, let $\xm$ be any forgetful parameter.  We define a useful class of restricted models, starting with the simplest case.

\begin{defn} \label{d:defn1}
For any $\xm$ and $M \models T^0_\xm$, $k < \omega$, and $(\eta, \rho) \in \mct^\xm_{1,k} \times \mct^\xm_{2,k}$, 

\begin{enumerate}
\item[(A)] we define $\xn = \xm \rstr (\eta, \rho)$ by: 

\begin{enumerate}
\item $\mct^\xn_1 = \{ \nu : \eta^\smallfrown \nu \in \mct^\xm_1 \}$
\item $\mct^\xn_2 = \{ \nu : \rho^\smallfrown \nu \in \mct^\xm_2 \}$
\item $\mcr^\xn_k = \{ (\nu_1, \nu_2) : (\eta^\smallfrown \nu_1, \rho^\smallfrown \nu_2) \in \mcr^\xm_k \}$
\end{enumerate}

\item[(B)] we define $N = N[\eta, \rho, M]$ as the following $\tau(T_{\xn} )$-model:
\begin{enumerate}
\item the universe is $Q^M_\eta \cup P^M_\rho$.
\item$Q^N_\nu = Q^M_{\eta^\smallfrown \nu}$ for ${\eta^\smallfrown \nu} \in \mct^\xm_1$
\item $P^N_\nu = P^M_{\rho^\smallfrown \nu}$ for ${\rho^\smallfrown \nu} \in \mct^\xm_2$
\item $\mcr^N = \mcr^M \rstr (Q^M_\eta \times P^M_\rho)$.
\end{enumerate}
\end{enumerate}
\end{defn}

\begin{claim}  \label{12.14}
Assume $\xn = \xm \rstr (\eta, \rho)$ where $(\eta, \rho) \in \mcr^\xm_k$, and $N = N[\eta, \rho, M]$.
Then:
\begin{enumerate}
\item If $M \models T^0_\xm$, then $N \models T^0_\xn$.
\item If $M \models T_\xm$, then:
\begin{enumerate}
\item $N \models T_\xn$, and 
\item if $A \subseteq N$ and 
$\bar{a}, \bar{b} \in {^n N}$ then $\tp(\bar{a}, A, M) = \tp(\bar{b}, A, M)$ if and only if
$\tp(\bar{a}, A, N) = \tp(\bar{b}, A, N)$. 
\end{enumerate}
\item If $M$ is a $\kappa$-saturated model of $T_\xm$, then $N$ is a $\kappa$-saturated model of $T_\xn$. 
\end{enumerate}
\end{claim}

\begin{proof}
(1) is by the definitions. (2) (a) is easy. (2)(b) is by the elimination of quantifiers. 
(3) follows from (2)(b). 
(Or see \ref{12.17A} below.)
\end{proof}

\begin{claim}
Assume $\xm$ is forgetful and $k < \omega$. 
\begin{enumerate}
\item[(Claim)] If $(\eta, \rho) \in \mcr_{\xm, k}$ and $(\eta^\prime, \rho^\prime) \in \mcr_{\xm,k}$ then 
$\xm \rstr (\eta, \rho) = \xm \rstr (\eta^\prime, \rho^\prime)$.

\item[(Defn)] Let $\xn = \xm \rstr_{\geq k}$ be $\xm \rstr (\eta, \rho)$ for any (some) $(\eta, \rho) \in \mcr_{\xm, k}$. 
\end{enumerate}
\end{claim}

\begin{proof} 
By the definition of forgetful.
\end{proof}

Now we generalize \ref{d:defn1}, but there is a subtlety. 
Suppose we are given $M \models T_\xn$ (or just $M \models T^0_\xn$) 
and $k < \omega$, and non-empty $\Lambda_1 \subseteq \mct_{1,k}$, $\Lambda_2 \subseteq \mct_{2,k}$ such that 
$\Lambda_1 \times \Lambda_2 \subseteq \mcr$ (i.e. a complete bipartite graph in the template at level $k$). 
We define 
$N = N[\Lambda_1, \Lambda_2, M]$ \underline{not} by looking at the multi-rooted 
subtrees of $\mct_1$, $\mct_2$ which start in $\Lambda_1$ or $\Lambda_2$ respectively, and then restricting $M$ to predicates coming from 
these multi-rooted subtrees in the natural sense. \underline{Rather}, we form a new model ``by gluing'' (or ``by stacking on top of'') all the subtrees with 
roots in $\Lambda_1$ to make the rooted tree on the left, and stacking up all the subtrees with roots in $\Lambda_2$ to make the rooted tree on the right.
So $Q^N_{\langle \rangle}$ is the union  
$\bigcup \{ Q^M_{\eta} : \eta \in \Lambda_1 \}$ and 
$P^N_{\langle \rangle}$ is the union  
$\bigcup \{ P^M_{\rho} : \rho \in \Lambda_2 \}$.  In general, $Q^N_\nu$ is the union of $\bigcup \{ Q^M_{\eta^\smallfrown \nu} : 
\eta \in \Lambda_1 \}$, and $P^N_\nu$ is the union of $\bigcup \{ P^M_{\rho^\smallfrown \nu} : 
\rho \in \Lambda_2 \}$.  The edges  in the model stay the same: $R^N = R^M \rstr Q^N_{\langle \rangle} \times P^N_{\langle \rangle}$
[note: roman $R$ in the model, not $\mcr$ in the template; for more, see \ref{12.17A}].   
This works well because $\xm$ is forgetful, so the subtrees sit exactly on top of each other, with the same branching and the 
same $\mcr$-information.   

We will be most interested in cases of the form $N = N[\{ \eta \}, \Lambda, M]$ where $\eta$ is a singleton and 
$\Lambda = \{ \rho \in \mct_2 : (\eta, \rho) \in \mcr \}$ is the set of all neighbors of $\eta$, or the parallel 
$N = N[\Lambda, \{ \rho \}, M]$, 
though we give the general definition. 

\begin{defn} \label{defn-x}
Assume $\xm$ is forgetful, $M \models T_\xm$, $k <\omega$, $\emptyset \neq \Lambda_\ell \subseteq \mct^\xm_{\ell, k}$ for $\ell = 1, 2$ 
and $\Lambda_1 \times \Lambda_2 \subseteq \mcr^\xm_k$.  

We define $N = N[\Lambda_1, \Lambda_2, M]$ as the following 
$\tau(\xm \rstr_{\geq k})$-model:

\begin{itemize}
\item universe: $\bigcup \{ Q^M_\eta \cup P^M_\rho : (\eta, \rho) \in \Lambda_1 \times \Lambda_2 \}$
\item $Q^N_\nu = \bigcup \{ Q^M_{\eta^\smallfrown \nu} : \eta \in \Lambda_1 \}$, so in particular, $\mcq^N = Q^N_{\langle \rangle}$. 
\item $P^N_\nu = \bigcup \{ P^M_{\eta^\smallfrown \nu} : \eta \in \Lambda_1 \}$, so in particular, $\mcp^N = P^N_{\langle \rangle}$. 
\item $R^N = R^M \rstr \mcq^N \times \mcp^N$.
\end{itemize}
\end{defn}

Note that \ref{defn-x} gives an interpretation of $N$ in $M$. 
Since we will mainly use the case where either $\Lambda_1$ or $\Lambda_2$ has cardinality $1$, we spell out this special case. 

\begin{obs} \label{2.16B}  Let $\xm$ be a forgetful parameter. 
For $M \models T_\xm$,
\begin{enumerate}
\item if $\Lambda_1 = \{\eta \}$ and $\Lambda_2 \subseteq \{ \rho \in \mct_{2,k} : (\eta, \rho) \in \mcr \}$, then 
our model has 
\begin{itemize}
\item universe: $Q^M_\eta \cup \bigcup \{P^M_\rho : \rho \in \Lambda_2 \}$
\item $Q^N_\nu = Q^M_{\eta^\smallfrown \nu}$. 
\item $P^N_\nu = \bigcup \{ P^M_{\rho^\smallfrown \nu} : \rho \in \Lambda \}$. 
\item $R^N = R^M \rstr \mcq^N \times \mcp^N$, i.e. we retain all existing edges between elements which make it into the domain of the new model. 
\end{itemize}
\item We have the parallel in the case where  $\Lambda_2 = \{ \rho \}$ and $\Lambda_1 \subseteq \{ \eta \in \mct_{1,k} : (\eta, \rho) \in \mcr \}$. 
\end{enumerate}
\end{obs}

\begin{conv}
We may write $N[\eta, \Lambda, M]$ instead of $N[\{\eta\}, \Lambda, M]$ and likewise we may write 
$N[\Lambda, \rho, M]$ instead of $N[\Lambda, \{ \rho \}, M]$. 
\end{conv}

Comparing Definition \ref{d:defn1} and Definition \ref{defn-x}, 
the reader will notice that in \ref{defn-x} we don't define a new parameter $\xn = \xn \rstr (\Lambda_1, \Lambda_2)$. This is because, by forgetfulness,
we already have an appropriate parameter,   
namely: 
in $\ref{2.16B}$, 
the structure $N = N[\Lambda_1, \Lambda_2, M]$ is a model of $T^0_{\xn}$ where $\xn = \xm \rstr_{\geq k}$. 

\begin{claim} \label{12.17A}
Assume $\xm$ is forgetful.   Assume $(\Lambda_1, \Lambda_2)$ are as in $\ref{defn-x}$,  
$N = N[\Lambda_1, \Lambda_2, M]$ and $\xn = \xm \rstr_{\geq k}$.  
Then: 
\begin{enumerate}
\item If $M \models T^0_\xm$, then $N \models T^0_\xn$.
\item If $M \models T_\xm$, then:
\begin{enumerate}
\item $N \models T_\xn$, and 
\item if $\bar{a}, \bar{b} \in {^n N}$ and $\tp(\bar{a}, \emptyset, M) = \tp(\bar{b}, \emptyset, M)$, then 
$\tp(\bar{a}, \emptyset, N) = \tp(\bar{b}, \emptyset, N)$. 
\end{enumerate}
\item If $M$ is a $\kappa$-saturated model of $T_\xm$, then $N$ is a $\kappa$-saturated model of $T_\xn$. 
\item If $M$ is a $\kappa$-special model of $T_\xm$, then $N$ is a $\kappa$-special model of $T_\xn$. 
\end{enumerate}
\end{claim}

\begin{proof}
Easy, but as this is central we elaborate.  

(1) There are two parts to check: the unary predicates match the structure of the tree, and 
the edges $R$ indeed reflect the instructions of $\mcr^\xn$. Since $\Lambda_1 \times \Lambda_2 \subseteq \mcr^\xm$, 
and $\xm$ is forgetful, $\xn = \xm \rstr_{\geq k}$ is well defined. Fix for a moment $\nu_1 \in \mct^\xn_{1, \ell}$, $\nu_2 \in \mct^{\xn}_{2,\ell}$. 
Then whenever $(\nu_1, \nu_2) \notin \mcr^\xn$, 
$T^0_\xn$ implies $(\forall x \in Q_{\nu_1})(\forall y \in P_{\nu_2})( \neg R(x,y))$.   So let us check: in the model $N$, 
\[ Q^N_{\nu_1} = \bigcup  \{ Q^M_{\eta^\smallfrown \nu_1} : \eta \in \Lambda_1 \} \]
\[ P^N_{\nu_2} = \bigcup  \{ P^M_{\eta^\smallfrown \nu_1} : \eta \in \Lambda_2 \}. \]
Suppose $(\nu_1, \nu_2) \notin \mcr_{\xn,\ell}$. Then (as $\xm \rstr_{\geq k}$) is well defined) for any choice of $(\eta, \rho) \in \Lambda_1 \times \Lambda_2$, 
$(\eta^\smallfrown \nu_1, \rho^\smallfrown \nu_2) \notin \mcr_{\xm, k+\ell}$. Since for all $a \in Q^N_{\nu_1}$, and all 
$b \in P^N_{\nu_2}$, there are $\eta \in \Lambda_1$ and $\rho \in \Lambda_2$ such that 
$a \in Q^M_{\eta^\smallfrown \nu_1}$ and $b \in P^M_{\rho^\smallfrown \nu_2}$, 
$M \models \neg R(a,b)$, so as $N$ changes no edges on its domain, also $N \models \neg R(a,b)$. 
So $(\forall x \in Q_{\nu_1})(\forall y \in P_{\nu_2})( \neg R(x,y))$ holds in $N$.   Thus we verify the universal axioms of 
$T^0_\xn$.

\br
(2), (3) are immediate from the fact that \ref{defn-x} gives an interpretation of $N$ in $M$. 
 
One could prove (2) directly:  
since $T_\xn$ eliminates quantifiers, it suffices to check truth of relevant formulas as in the quantifier elimination argument \ref{k17}. 

For (3), we also sketch a direct proof as it will help for (4). 
Recall $M$ is $\kappa$-saturated.  Convention: given a leaf, say $\rho \in \lim(\mct^n_2)$, let us write $P^M_\rho$ to abbreviate 
$\bigcap \{ P^M_{\rho \rstr n} : n < \omega \}$, and the parallel for $\lim(\mct^n_1)$.  Sometimes we will repeat this for emphasis. 
As $T_\xn$ has elimination of quantifiers, it suffices to verify that:

\begin{quotation}
\noindent $(\star)_1$ if $\nu \in \lim(\mct^\xn_1)$, $X = \{ \rho \in \lim(\mct^\xn_2) : \bigwedge_m (\nu \rstr m, \rho \rstr m) \in \mcr^\xn_m \}$, and 
$A, B \subseteq \bigcup \{ P^N_\rho : \rho \in X \}$ are disjoint and $|A| + |B| < \kappa$, 
then for some $c \in Q^N_\nu = \bigcap_n Q^N_{\nu \rstr n}$ we have that 
$(c,a) \in R^N$ for all $a \in A$, and $(c,b) \notin R^N$ for all $b \in B$.
\end{quotation}
as well as $(\star)_2$ the dual reversing the trees.
In the case where $A$, $B$ are finite, this is the content of redoing the elimination of quantifiers as mentioned in (2). 

As we are using $\Lambda_1, \Lambda_2$, symmetry holds and it suffices to prove $(\star)_1$. 
In particular it suffices to prove that for some $\eta \in \Lambda_1$ there is 
$c \in \bigcap \{ Q^M_{\eta^\smallfrown \nu \rstr n } : n < \omega \}$ such that 
$(c,a) \in R^M$ for all $a \in A$, and $(c,b) \notin R^M$ for all $b \in B$.
Thus we move the problem to the model $M$. 

In $M$, the sets $A$ and $B$ clearly remain disjoint. Let us locate the $a$'s and $b$'s. 
Since our $N$ was a model of $T_\xn$, for each $a \in A$, there is $\nu_a \in \lim(\mct^\xn_2)$ such that $a \in P^N_{\nu_a}$, and likewise 
for each $b \in B$, there is $\nu_b \in \lim(\mct^\xn_2)$ such that $b \in P^N_{\nu_b}$. 
Therefore in $M$, for each $a \in A$ there is $\rho_a \in \Lambda_2$ 
such that $a \in P^M_{{\rho_a}^\smallfrown \nu_a}$, and likewise 
for each $b \in B$ there is $\rho_b \in \Lambda_2$ 
such that $b \in P^M_{{\rho_b}^\smallfrown \nu_b}$, though the positive part is what is most important.
Choose any $\eta \in \Lambda_1$.   
Since $\Lambda_1 \times \Lambda_2 \subseteq \mcr^\xm$, and by definition of $X$, we conclude that 
$\bigwedge_n ( \eta^\smallfrown \nu \rstr n, ~ {\rho_a}^\smallfrown \nu_a \rstr n) \in \mcr^\xm_n$. 
So the type $p(x)$ saying that $\{ R(x,a) : a \in A \} \cup \{ \neg R(x,b) : b \in B \} \cup \{ P_{\eta^\smallfrown \nu \rstr n} (x) : n < \omega \}$
is indeed a consistent partial type in $M$, and therefore 
there is such a $c \in |M|$ satisfying 
$(c,a) \in R^M$ for all $a \in A$, and $(c,b) \notin R^M$ for all $b \in B$, by saturation of $M$. Moreover, by definition of $p(x)$, 
it must be that $c \in P^M_{\eta}$, thus $c \in |N|$ by construction. This completes the proof. 

(4) Similarly to (3).  
\end{proof}

\begin{claim}
Let $\de$ be a regular ultrafilter on $I$, $|I| = \lambda$. Suppose $\xm$ is a forgetful parameter $($thus self-dual$)$ and $k < \omega$. 
Suppose $\xn = \xm \rstr_{\geq k}$.  Then the following are equivalent:
\begin{enumerate}
\item $(M_\xm)^I/\de$ is $\lambda^+$-saturated, for some, equivalently every, $M_\xm \models T_\xm$.
\item $(M_\xn)^I/\de$ is $\lambda^+$-saturated, for some, equivalently every, $M_\xn \models T_\xn$.
\end{enumerate}
In other words, $T_\xm$ and $T_\xn$ are equivalent in Keisler's order.\footnote{We pedantically avoid using the letter $N$ for a model in this proof 
to avoid confusion with the operation $N[\Lambda_1, \Lambda_2, M]$.} 
\end{claim}

\begin{proof}
First note: for each given $\eta, \Lambda$, clearly we can interpret $N[\eta, \Lambda, M]$ in the model $M$. 
So as ultrapowers commute with reducts,  
$(N[\eta, \Lambda, M])^I/\de$ is canonically isomorphic to $N[\eta, \Lambda, M^I/\de]$. The parallel holds for $N[\Lambda, \eta, M]$, though since 
$\xm$ is forgetful therefore symmetric, it suffices to consider the first case.  

(1) implies (2): 
As the ultrafilter is regular, it won't matter which model of $T_\xm$ or $T_\xn$ we choose (see Keisler \cite{keisler} 2.1a), so we may as well 
choose $\eta \in \mct^\xm_{1,k}$, $\rho \in \mct^\xm_{2,k}$ and $M_\xn = N[\eta, \rho, M_\xm]$.
Now apply \ref{12.17A} and the fact that ultrapowers commute with reducts.

(2) implies (1): 
Suppose (2). 
Then $\de$ is good for the theory of the random graph, because the theory of $M_\xn$ will be unstable and the 
theory of the random graph is minimum in Keisler's order among the unstable theories (see \cite{mm4} \S 5).  

Let $M^*_\xm = (M_\xm)^I/\de$. 
By Conclusion \ref{concl123}, it suffices to show that every partial type of $M^*_\xm$ of the form 
\[ r(x) = { \{ Q_{\nu}(x) \} } \cup \{ R(x,a) : a \in A \} \]
is realized, where $\nu \in \mct_{1}$, $A \subseteq M^*_\xm$ and $|A| \leq \lambda$.
So $\nu$ has finite length, without loss of generality $\lgn(\nu) \geq k$, though of course $>k$ is allowed.  
Let $\eta = \nu \rstr k$. 

Since $r$ is consistent, we may choose $\eta_*$ such that:
\begin{itemize}
\item[(a)] $\eta ^\smallfrown \eta_* \in \lim(\mct^\xm_1)$, and also $\nu \tlf \eta^\smallfrown \eta_*$. 
[That is, we extend $\nu$ to a branch, but we write it as its restriction to $k$ plus the continuation.]
\item[(b)] and moreover the larger type 
\[   r_*(x) = \{ Q_{\eta^\smallfrown \eta_* \rstr \ell}(x) : \ell < \omega \} \cup \{ R(x,a) : a \in A \}  \]
is still consistent. 
\end{itemize}
Moreover, for each $a \in A$, 
we may choose $\nu_a \in \mct_{2,k}$ and
$\rho_a$ with ${\nu_a}^\smallfrown \rho_a \in \lim(\mct^\xn_2)$ such that for each $ a \in A$,
$M^*_\xm \models \{ P_{{\nu_a}^\smallfrown{\rho_a} \rstr \ell}(a) : \ell < \omega \}$.
[Again we identify the branch of each $a \in A$, and write it as its restriction to $k$ plus the continuation.]

By the fact that $r_*$ is a [partial] type, note that
\begin{enumerate}
\item[(c)] if we write $\Lambda = \{ \rho \in \mct^\xm_{2,k} : (\eta, \rho) \in \mcr^\xm \}$, then 
$\nu_a \in \Lambda$ for each $a \in A$.

\item[(d)] moreover, for each $a \in A$ and each $\ell < \omega$, 
\[ (\eta^\smallfrown \eta_* \rstr \ell, ~ {\nu_a}^\smallfrown \rho_a \rstr \ell) \in \mcr^\xm. \]
\end{enumerate}

Let $M_\xn = N[\eta, \Lambda, M_\xm]$. Then, recalling the beginning of the proof, we have that $M^*_\xn = N[\eta, \Lambda, M^*_\xm] = ( ~N[\eta, \Lambda, M_\xm]~)^I/\de$. 

By (c), we have that $A \subseteq M^*_\xn$ [i.e. $A \subseteq \dom(M^*_\xn)$]. 

By (d) and the definition of forgetful, we have that $(\eta_* \rstr \ell, \rho_a \rstr \ell) \in \mcr^\xn$ for each $a \in A$ and $\ell < \omega$.

Thus $q_*(x) = \{ Q_{\eta_* \rstr \ell}(x) : \ell < \omega \} \cup \{ R(x,a) : a \in A \}$ is a partial type in $M^*_\xn$. By our hypothesis 
(2), $M^*_\xn$ is $\lambda^+$-saturated, so $q_*$ is realized, say by $b$. By construction, $\dom(M^*_\xn) \subseteq \dom(M^*_\xm)$, so 
$b \in M^*_\xm$, indeed $M^*_\xm \models Q_\eta(b)$.  
Moreover, since $\nu$ (from the definition of $r(x)$ above) is an initial segment of $\eta^\smallfrown \eta_*$, 
$M^*_\xm \models Q_\nu(b)$. This shows that $b$ realizes $r(x)$ in $M^*_\xm$, which completes the proof. 
\end{proof}

The proofs just given show:

\begin{concl}
Let $\xm$ be a forgetful parameter.
\begin{enumerate}
\item 
Suppose $M = T_\xm$. 
$M$ is $\lambda^+$-saturated if and only if for every pair $(\Lambda_1, \Lambda_2)$ such that 
$\Lambda_1 \times \Lambda_2 \subseteq \mcr^\xm$ and either $|\Lambda_1| = 1$ or $|\Lambda_2| = 1$, the model 
$N = N[\Lambda_1, \Lambda_2, M]$ is $\lambda^+$-saturated. 
\item Let $k < \omega$ and let $\xn = \xm \rstr_{\geq k}$. Then $T_\xm$, $T_\xn$ are equivalent in Keisler's order. 
\end{enumerate}
\end{concl}

Recall that in the context of Convention \ref{c:A-B} p. \pageref{c:A-B}, our goal in this section is to compare theories of the form 
$T_{\xm[A]}$ and $T_{\xm[B]}$.  The above shows that without loss of generality $A \subseteq B$, so Claim \ref{modfinite}, which we now state, will follow from \ref{cor62} below.  

\begin{claim} \label{modfinite}
Suppose $A \subseteq^* B$. Then $T_{\xm[A]} \tlf T_{\xm[B]}$. 
\end{claim}

\begin{proof} This will follow from \ref{cor62} below. \end{proof}

Recall that $\mcq = Q_{\langle \rangle}$, $\mcp = P_{\langle \rangle}$. Write $\mct^A_\ii$ for $\mct^{\xm[A]}_\ii$, $\mct^B_{\ii}$ for 
$\mct^{\xm[B]}_\ii$ and $\ii = 1, 2$, 
and $\mcr^A = \mcr^{\xm[A]}$, $\mcr^B = \mcr^{\xm[B]}$. 

\begin{obs}
As $A \subseteq B$, we have that:
\begin{itemize}
\item $\mct^A_\ii = \mct^B_\ii$ for $\ii=1,2$.
\item $\mcr^B \subseteq \mcr^A$.
\end{itemize}
\end{obs}

\begin{conv} Notation: 
\begin{enumerate}
\item 
If $a \in \mcp^{M}$, saying that ``$\rho$ is the leaf of $a$'' means that $\rho$ is the unique element of $\lim(\mct^M_2)$ such that 
$M \models P_{\rho \rstr \ell}(a)$ for $\ell < \omega$, and similarly for $a \in \mcq^M$ and $\eta \in \lim(\mct^M_1)$. 
\item Write $(\eta, \rho) \in \mcr^A_\infty$ to mean that $\eta \in \lim(\mct^A_1)$, $\rho \in \lim(\mct^A_2)$ and for all $\ell < \omega$,
$(\eta \rstr \ell, \rho \rstr \ell) \in \mcr^A$. 
\end{enumerate} 
\end{conv}

\begin{obs} \label{obs51}
Suppose $n<\omega$, $Y \subseteq \mcp^{M_A}$ is finite, $\eta \in \mct^A_1$. 
 For $a \in Y $ let $\rho_a \in \lim(\mct^A_2)$ be the leaf of $a$. 
 Let $M_A \models T_{\xm[A]}$. 
The following are equivalent: 
\begin{enumerate}
\item $M_A \models (\exists x)\bigwedge_{a \in Y} ( ~Q_\eta(x) \land R(x,a)~) $
\item there is $\nu \in \lim(\mct^A_1)$, $\eta \tlf \nu$ such that 
$(\nu, \rho_a) \in \mcr^A_\infty$ for each $ a\in Y$. 
\end{enumerate}
\end{obs}

\begin{obs}[Discussion/Observation]
\emph{Suppose we are in a model $M_A \models T_A$, $\eta \in \mct^A_1$ ~[so $\lgn(\eta) < \omega$], and we have a finite set of formulas 
[\emph{not necessarily a partial type}] of the form} 
\[  \{ Q_\eta(x) \} \cup \{ R(x,a) : a \in X \} \]
\emph{which we may write for uniformity as a sequence, possibly with repetitions:}
\[ p(x) = \langle Q_\eta(x) \land R(x,a) : a \in X \rangle. \]
\emph{For each $a \in X$, let $\rho_a \in \lim(\mct^A_2)$ be the leaf of $a$. 
Then the ``pattern'' of this set of formulas in $M_A$ is captured by the data of: for which $\sigma \subseteq X$ is it the case that 
there exists $\eta_* \in \lim(\mct^A_1)$ such that (i) $\eta \tlf \eta_*$ and (ii) $(\eta_*, \rho_a) \in \mcr^A_\infty$ for all $a \in X$.}  
\end{obs} 

\begin{disc}
One obstruction to the natural try at $A \subseteq B$ implies $T_A \tlf T_B$ is: Consider $M_A \models T_A$, 
$M_B \models T_B$, $\de$ a regular ultrafilter on $I$, $M^*_A = (M_A)^I/\de$, $M^*_B = (M_B)^I/\de$. Now we might like to show that given a certain type of $M^*_A$, we can copy it to a type of $M^*_B$ and apply saturation of the second to solve the first. A natural way to copy is to 
assign finitely many formulas of $p$ to each index $t \in I$, using the regularizing family, and then copy this finite pattern from 
$M_A$ to $M_B$. The natural way 
to copy patterns is to use the same leaves, but even though 
$\mct^A_\ii = \mct^B_\ii$, the same leaves may give a different pattern in $M_B$ as $\mcr^A \supsetneq \mcr^B$. 

Instead, we show that given $M_A$ and a finite pattern there, we can first modify the leaves involved a little without changing the 
pattern (informally, we make sure there is no `spreading out' on $k \in \omega \setminus A$), and second, that any such modified 
set of leaves behaves as desired in $\mcr^B$. 
This  
allows the plan to go through. 
\end{disc}

\begin{claim} \label{claim52}
Recalling that $\xm[A]$ and $\xm[B]$ are forgetful, 
suppose we are given $\langle Q_\eta(x) \land R(x,a) : a \in X \rangle$ in $M_A$ where:
\begin{itemize}
\item $\eta \in \mct^A_1$
\item $X$ finite, 
\item for each $a \in X$, $M_A \models (\exists x)( ~Q_\eta(x) \land R(x,a) ~)$ 
\item \emph{thus} we are given $\langle \rho_a : a \in X \rangle$ where $\rho_a \in \lim(\mct^A_2)$ is the leaf of $a$. 
\end{itemize}
Let $\langle \rho^*_a : a \in X \rangle \in {^{|X|} \lim(\mct^B_2)}$ satisfy the following two conditions:
\begin{enumerate}
\item[(a)] for $k \in A$ and $a \in X$, $\rho^*_a(k) = \rho_a(k)$
\item[(b)] for $k \in \omega \setminus A$,  for all $a \in X$, $\rho^*_a(k) = 0$
\end{enumerate}
Then for some $\nu \in \mct^B_1$ with $\lgn(\nu) = \lgn(\eta)$,  
for every $\sigma \subseteq X$, the following are equivalent:
\begin{enumerate}
\item there exists $\eta_\sigma \in \lim(\mct^A_1)$, $\eta \tlf \eta_\sigma$ 
and $(\eta_\sigma, \rho_a) \in \mcr^A_\infty$ for every $a \in \sigma$
\item there exists $\nu_\sigma \in \lim(\mct^B_1)$, $\nu \tlf \nu_\sigma$ 
and $(\nu_\sigma, \rho^*_a) \in \mcr^B_\infty$ for every $a \in \sigma$. 
\end{enumerate}
\emph{Informally,} ``the pattern remains unchanged and transfers to $M_B$''.  
\end{claim}

\begin{rmk} \label{r:52}
Continuing in the notation of Claim \ref{claim52}, it will follow from that Claim and from Observation \ref{obs51} 
that if we choose distinct elements $b_a \in |M_B|$ for $a \in X$ so that for each $a \in X$, $\rho^*_a$ is the leaf of $b_a$ in $M_B$, 
then for any $\sigma \subseteq X$, 
the following are equivalent:
\begin{enumerate}
\item[$(\star)_1$] $M_A \models (\exists x) \bigwedge_{a \in \sigma} (~ Q_\eta(x) \land R(x,a) ~)$. 
\item[$(\star)_2$] $M_B \models (\exists x) \bigwedge_{a \in \sigma} (~ Q_\nu(x) \land R(x,b_a) ~)$.
\end{enumerate}
\end{rmk}

\vspace{1mm}
\begin{proof}[Proof of Claim \ref{claim52}] 
Note: the construction in the proof will give a $\nu$ which depends only on $\eta$. 
In particular:  
for each $k < \lgn(\eta)$, define $\nu(k)$ by:
\begin{itemize}
\item if $k \in A$, $\nu(k) = \eta(k)$ 
\item if $k \in \omega \setminus A$, $\nu(k) = \min \{ i < m_k : (i, 0) \in E^k \}$ which exists by the definition of ``small'' in the construction of the graphs $\bar{E}$. 
\end{itemize}
Note that if $\lgn(\eta) \subseteq A$, then this gives $\nu = \eta$. 

\textbf{(1) implies (2).}  
By (1) there is such an $\eta_\sigma$.  Define $\nu_\sigma \in \lim(\mct^B_1)$ [recalling $\mct^B_1 = \mct^A_1$] 
as follows:
for each $k <\omega$, define $\nu_\sigma(k)$ by:
\begin{itemize}
\item[(i)] if $k \in A$, $\nu_\sigma(k) = \eta_\sigma(k)$ 
\item[(ii)] if $k \in \omega \setminus A$, $\nu_\sigma(k) = \min \{ i < m_k : (i, 0) \in E^k \}$ which exists by the definition of ``small'' in the construction of the graphs $\bar{E}$. 
\end{itemize}
Then recall $\eta \tlf \eta_\sigma$ hence $\nu \tlf \nu_\sigma$.  For each $a \in \sigma$, we will prove by induction on $\ell < \omega$ that 
\[ (\star)_\ell  \mbox{ \hspace{10mm} }
 (\eta_\sigma \rstr \ell, ~ \rho_a \rstr \ell) \in \mcr^A \implies (\nu_\sigma \rstr \ell, ~ \rho^*_a \rstr \ell) \in \mcr^B. \]
For $\ell = 0$ this is trivially true as $(\langle \rangle, \langle \rangle) \in \mcr^A \cap \mcr^B$. 
Suppose $(\star)_\ell$ holds.
If $\ell \in A$, then $\nu_\sigma(\ell) = \eta_\sigma(\ell)$ and $\rho_a(\ell) = \rho^*_a(\ell)$, so by 
inductive hypothesis and forgetfulness, $(\star)_{\ell+1}$ is immediate. 
If $\ell \notin A$, then by inductive hypothesis and $(ii)$, 
\[ ( ~ (\nu_\sigma \rstr \ell)^\smallfrown \langle \nu_\sigma(\ell) \rangle, ~ (\rho^*_a \rstr \ell)^\smallfrown \langle 0 \rangle) \in \mcr^B \]
i.e. $(\nu_\sigma \rstr \ell+1, \rho^*_a \rstr \ell+1) \in \mcr^B$, 
so $(\star)_{\ell+1}$ holds.  

This proves $(\star)_{\ell+1}$ and so proves this direction. 

\br

\noindent\textbf{(2) implies (1).}
Suppose there is such a $\nu_\sigma$. Write $\nu_\sigma = \nu^\smallfrown \nu_\infty$. Let $\eta_\sigma = \eta^\smallfrown \nu_\infty$. 
So clearly $\eta \tlf \eta_\sigma$. 
For each $a \in \sigma$, we will prove by induction on $\ell < \omega$ that 
\[ \oplus_\ell  
\mbox{ \hspace{10mm} } (\nu_\sigma \rstr \ell, ~ \rho^*_a \rstr \ell) \in \mcr^B \implies (\eta_\sigma \rstr \ell, ~ \rho_a \rstr \ell) \in \mcr^A. \]
For $\ell \leq \lgn(\eta)$ this is trivially true as it follows from the assumptions of the Claim (third bullet point) that 
$(\eta, \rho_a \rstr \lgn(\eta)) \in \mcr^A$ for each $a \in X$, so also for each $a \in \sigma$. 
Suppose $\oplus_\ell$ holds and $\ell \geq \lgn(\eta)$.  
If $\ell \in A$, then $\eta_\sigma(\ell) = \nu_\sigma(\ell)$ and 
$\rho_a(\ell) = \rho^*_a(\ell)$, so $\oplus_{\ell+1}$ follows by inductive hypothesis and forgetfulness. 
If $\ell \in \omega \setminus A$, then $\ell$ is a lazy level for $A$, so $\oplus_{\ell+1}$ is immediate from $\oplus_\ell$. 
This proves $\oplus_{\ell+1}$, which finishes this direction and so finishes the proof. 
\end{proof}

\br

\begin{lemma} \label{l342}
Let $\de$ be a regular ultrafilter on $I$, $|I| = \lambda$. 
Let $M^*_A = (M_A)^I/\de$.  Let $M^*_B = (M_B)^I/\de$. 
Suppose $p(x) = \{ Q_\eta(x) \} \cup \{ R(x,a) : a \in X \}$ is a partial type of $M_A$, with $\eta \in \mct^A_1$ and $|X| \leq \lambda$.
There is a partial type 
$q(x) = \{ Q_\eta(x) \} \cup \{ R(x,b_a) : a \in X \}$ 
of $M^*_B$, such that if $q(x)$ is realized in 
$M^*_B$, then $p(x)$ is realized in $M^*_A$.
\end{lemma}

\begin{proof} 
Let $\{ X_a : a \in X \} \subseteq \de$ be a regularizing family for $\de$. 
Enumerate $p$ as $\langle Q_\eta(x) \land R(x,a) : a \in X \rangle$, \emph{so here $X$ may be infinite}. 
Let $f: [X]^{<\aleph_0} \rightarrow \de$ be the map given by 
\[ u \mapsto \{ t \in I : M_A \models (\exists x)\bigwedge_{a \in u}
(Q_\eta(x) \land R(x,a)) \} \cap \bigcap_{a \in u} X_a. \]
In particular, for each $t \in I$, $X_t = \{ a \in X : t \in f( \{ a \}) \}$ 
is finite.

We define the parameters for $q$ (what will be the corresponding type in $M^*_B$) coordinatewise. 
For each $t \in I$,  apply Claim \ref{claim52} to the following sequence 
of formulas of $M_A$: 
\[ p_t(x) = \langle Q_\eta(x) \land R(x,a[t]) : a \in X_t  \rangle. \] 
Let $\nu_t$ and $\langle \rho^*_{t,a} : a \in X_t \rangle$    
be as returned by that Claim.  
Observe that we have $\nu_t$ defined for every $t \in I$, but $\{ \nu_t : t \in I \}$ is finite, so for some $J \in \de$, $\langle \nu_t : t \in J \rangle$ 
is constant. 
[Alternately, by the first sentence of the proof of Claim \ref{claim52}, $\nu_t$ depends only on $\eta$, so in fact we obtain the same 
$\nu_t$ for every $t \in I$.]
So we call it $\nu$.

For $a \in X_t$, choose $b_{t,a}$ to be pairwise distinct elements of 
$P^{M_B}_{\rho^*_{t,a}}$.  
For $a \in X \setminus X_t$, choose $b_{t,a}$ to be any element of 
$P^{M_B}_\emptyset$.  Recalling \ref{r:52}, we have that for any $\sigma \subseteq X$,
the following are equivalent: 
\begin{enumerate}
\item $M_A \models (\exists x) \bigwedge_{a \in \sigma} (~ Q_\eta(x) \land R(x,a[t]) ~)$
\item $M_B \models (\exists x) \bigwedge_{a \in \sigma} (~ Q_\nu(x) \land R(x,b_{t,a}) ~)$
\end{enumerate}

For each $a \in X$, define $b_a = \langle b_{t,a} : t \in I \rangle/\de$, 
so $b_a \in M^*_B$. Consider 
\[ q(x) = \langle Q_\eta(x) \land R(x,b_a) : a \in X \rangle. \]
Then, since $q$ has the 
same `pattern' as $p$ at each index $t \in I$ under the function $f$,  and $p$ is a partial type in $M^*_A$, 
we have that $q(x)$ is also a partial type of $M^*_B$, and moreover, if $q$ is realized in 
$M^*_B$ then also $p$ is realized in $M^*_A$.  
\end{proof}

\begin{cor} \label{cor62}
If $A \subseteq B$, then $T_A \tlf T_B$. 
\end{cor}

We include a slightly more general result for $\tlf^*$, the interpretability order. 
For the reader interested in Keisler's order, this is not essential to our main arguments. 

\begin{lemma} \label{tlf-lemma}
Suppose $\xm_1, \xm_2$ are forgetful and $\xm_1 \rstr_{\geq k} = \xn = \xm_2 \rstr_{\geq k}$. Then 
$T_{\xm_1}$, $T_\xn$, $T_{\xm_2}$ are $\tlf^*$-equivalent, thus $\tlf$-equivalent.  
\end{lemma}

\begin{proof}  
Let $M_1, M, M_2$ be models of $T_{\xm_1}$, $T_\xn$, $T_{\xm_2}$ respectively such that all three are special of cardinality 
$\mu = \beth_\delta$, $\delta$ limit. Let $\chi$ be such that $M_1, M, M_2 \in \mch(\chi)$. Let 
$\ba_*$ be the model $(\mch(\chi), \in, M_1, M, M_2)$ or an expansion of it. 
Note first that for $\Lambda_1 \times \Lambda_2 \subseteq \mcr^{\xm_1}_k$, 
$M$ is isomorphic to 
$N[\Lambda_1, \Lambda_2, M_1]$. [Why? $M_1$ is a special model of $T_{\xm_1}$, so by 
Claim \ref{12.17A}(4), so is $N[\Lambda_1, \Lambda_2, M_1] \models T_\xn$. By assumption $M \models T_\xn$ is special. 
So by the uniqueness of special models, they must be isomorphic.]
The parallel holds replacing $M_1$ by $M_2$. 
These isomorphisms are recorded by $\ba_*$.  Hence they remain in any $\ba \equiv \ba_*$. That is, 
\[ \ba_* \models ~ M \cong N[\Lambda_1, \Lambda_2, M_\ii] \]
for $\ii = 1, 2$, and $\Lambda_1, \Lambda_2$ as above. Moreover, for any $\ba \equiv \ba_*$, 
\[ \ba \models ~ \mbox{``} M \cong N[\Lambda_1, \Lambda_2, M_\ii] \mbox{''} \]
(as $M, M_1, M_2$ are definable elements). Note also that 
\[ \left(    N[\Lambda_1, \Lambda_2, M]    \right)^\ba  = N[\Lambda_1, \Lambda_2, M^\ba_\ii], ~ \ii = 1, 2 \]
by absoluteness. (Pedantically, if $\ba$ is not well-founded, the set $\{ P_\eta : \ba \models P_\eta \in \tau(M_\ii) \}$ may have 
non-standard elements, but no harm.)
Thus, for $\ii = 1, 2$, 
\[ N[\Lambda_1, \Lambda_2, M^\ba_\ii] \mbox{ is a special model of $T_\xn$ of cardinality $\mu$ isomorphic to $M^\ba$} \] 
and so by transitivity, as both are isomorphic to $M^\ba$, 
\[ N[\Lambda_1, \Lambda_2, M^\ba_1] \cong N[\Lambda_1, \Lambda_2, M^\ba_2]. \]
So $Th(\ba_*)$ witnesses that $T_{\xm_1}, T_\xn, T_{\xm_2}$ are $\tlf^*$-equivalent. Why? $M^\ba_\ii$ is $\kappa$-saturated 
if and only if all the models of the form $N[\Lambda_1, \Lambda_2, M^\ba_\ii]$ for $\Lambda_1 \times \Lambda_2 \subseteq \mcr^{\xm_\ii}_k$ 
are $\kappa$-saturated, if and only if $M^\ba$ is $\kappa$-saturated, which suffices. 
\end{proof}

\begin{claim} \label{lemma5}
Suppose $B \setminus A$ is infinite. Then $\neg (T_{\xm[B]} \tlf T_{\xm[A]})$. 
\end{claim}

\begin{proof}
Just as in the main argument above: as long as $B \setminus A$ is infinite, 
we have enough ``independence'' between the parameters for this direction. 
That is,  
consider the chain condition Definition 8.2 with the cosmetic difference that we write 
$A,B$ instead of using level functions: really, we could define $\xi_{\xm[A]}$ to be $1$ if $n \in A$ and $0$ otherwise, 
and $\xi_{\xm[B]}$ to be $1$ if $n \in B$ and $0$ otherwise.  Then if $\mci$ is the ideal generated by $\{ A \} \cup [\omega]^{<\aleph_0}$, 
clearly $B \neq \emptyset \mod \mci$ so we may quote \ref{9.14aa}.  Thus,  
it remains possible to construct an ultrafilter which is good for $T_{\xm[A]}$ while preserving the fact that it is not 
good for $T_{\xm[B]}$. 
\end{proof}

So we arrive at:

\begin{theorem}
We can find $T_A$ for $A \subseteq \omega$ such that $T_{A} \tlf T_{B}$ if and only if $A \subseteq B$ mod finite. 
\end{theorem}

\begin{proof} 
We use the family $\{ T_{\xm[A]} : A \subseteq \omega \}$ defined above. 
If $A \subseteq B$ mod finite, apply \ref{modfinite}.
On the other hand, if $B \setminus A$ is infinite, apply Lemma \ref{lemma5}. 

Thus for $A, B \subseteq \omega$, 
\[ T_{\xm[A]} \tlf T_{\xm[B]} \mbox{ if and only if } A \subseteq^* B. \]
This is what we hoped to prove. 
\end{proof}

\noindent In fact we have shown more: 

\begin{concl} \label{p-omega}
Suppose we consider the family $\{ \xm[A] : A \subseteq \omega \}$ of parameters defined at the beginning of this section. 
Then:
\begin{enumerate}
\item $A \subseteq^* B \implies T_{\xm[A]} \tlf T_{\xm[B]}$, indeed $T_{\xm[A]} \tlf^* T_{\xm[B]}$
\item $| B \setminus A | = \aleph_0 \implies \neg (T_{\xm[B]} \tlf T_{\xm[A]})$, thus a fortiori 
$\neg (T_{\xm[A]} \tlf^* T_{\xm[B]})$.
\end{enumerate}
\end{concl}

\noindent This completes the proof of Theorem \ref{thm-mcp}.  

\vspace{5mm}

To motivate the second theorem of this section, 
remember that when we were partitioning $\sM_*$ into $\mcm$ and $\mcn$ in the previous section, 
we were essentially choosing a partition of independent subsets of $\omega$. We may ask about which partitions are possible of 
\emph{all} families of subsets of $\omega$. 
The next theorem answers this question: it is  
 \ref{thm-ideals} above, which we restate for convenience.

\setcounter{equation}{0}

\begin{thm-nn}[Theorem \ref{thm-ideals}]  
There is a family of parameters $\{ \xm[A] : A \subseteq \omega \}$  such that each $T_{\xm[A]}$ is countable, complete, simple, and low, 
and the following are equivalent for any $\lambda \geq 2^{\aleph_0}$ and any set 
$\mcx \subseteq \mcp(\omega)$:
\begin{enumerate}
\item There exists a regular ultrafilter $\de$ on $\lambda$ such that $\mcx = \{ A \subseteq \omega : \de $ is $(\lambda^+, T_{\xm[A]})$-good $\}$. 
\item $\mcx \supseteq [\omega]^{<\aleph_0}$ is an ideal. 
\end{enumerate}
\end{thm-nn}

\begin{proof}[Proof of Theorem \ref{thm-ideals}] \label{proof-ideals} Fix $\mcx \subseteq \mcp(\omega)$, and fix $\lambda \geq 2^{\aleph_0}$. 

(2) $\rightarrow$ (1):  Immediate from the earlier construction: simply choose the ideal $\mci$ in the chain condition \ref{d1:cca} to be our $\mcx$.

(1) $\rightarrow$ (2): 
If $A \in \mcx$ and $B \subseteq^* A$, then 
$\de$ is $(\lambda^+, T_{\xm[B]})$-good by \ref{modfinite} above. 
So to show that $\mcx$ is an ideal which extends the finite sets, 
it suffices to show that if $A \in \mcx$ and $B \in \mcx$, then $A \cup B$ is in $\mcx$. 
In other words, we shall fix $A, B \subseteq \omega$, and assume that $\de$ is a regular ultrafilter on $|I|$, $|I| = \lambda \geq 2^{\aleph_0}$ which is $(\lambda^+, T_{\xm[A]})$-good and $(\lambda^+, T_{\xm[B]})$-good, and we shall show that $\de$ is 
$(\lambda^+, T_{\xm[A \cup B]})$-good.\footnote{Informally, if $\de$ can handle types coming from trees where the levels in $A$ are active, and trees where the levels in $B$ are active, then it can 
handle types coming from trees where the levels in $A \cup B$ are active.}

Since $T_{\xm[A]}$ and $T_{\xm[B]}$ are both simple unstable, it follows that $\de$ is good for the theory of the random graph. 
Choose $M_A$, $M_B$, $M_{A \cup B}$ to be $\aleph_1$-saturated models of $T_{\xm[A]}$, $T_{\xm[B]}$, and $T_{\xm[A \cup B]}$ respectively. 
Let $M^*_A, M^*_B, M^*_{A \cup B}$ be the respective ultrapowers using $\de$. As usual, to show that $M^*_{A \cup B}$ is $\lambda^+$-saturated, it 
suffices to prove that all partial types of the form 
\[ r(x)= \{ Q_\eta(x) \cup \{ R(x,c) : c \in C \} \]
are realized, where $\eta \in \mct^{\xm[A \cup B]}_1$ and $|C| \leq \lambda$. 
Fix such an $r$. 
Without loss of generality, we will assume $|C| = \lambda$. 

Let $\{ Y_c : c \in C \} \subseteq \de$ be a regularizing family for $\de$. \footnote{i.e. any element of the family belongs to $\de$, 
but the intersection of any infinitely many elements of this family is empty -- exists by definition of regular ultrafilter.}
Let 
\[ d: [C]^{<\aleph_0} \rightarrow \de \]
be the map given by: 
\[ u \mapsto \{ t \in I : M_{A \cup B} \models \exists x \bigwedge_{c \in u} (Q_\eta(x) \land R(x,c[t]))  ~~\} \cap \bigcap_{c \in u} Y_c. \]
Note that $d$ is monotonic ($u \subseteq v$ implies $d(u) \supseteq d(v)$), and 
 for each $t \in I$, the set $C_t = \{ c \in C : t \in d(\{c\}) \}$ is finite. 

For each $t$ and $c \in C$, there is a leaf $\rho_{c,t} \in \lim(\mct^{\xm[A]}_2) = 
\lim(\mct^{\xm[B]}_2)$ such that $M \models P_{\rho_{c,t} \rstr \ell}~(c[t])$ for all $\ell <\omega$. 
For each $t \in I$ and each $c \in C$ choose\footnote{One may follow these instructions for $t \in I$ and $c \in C_t$, and otherwise 
choose arbitrarily.} $a_{c,t}$ to be any element of $M_A$ such that $M_A \models P_{\rho_{c,t} \rstr \ell}~(a_{c,t})$ for all $\ell < \omega$ 
and choose $b_{c,t}$ to be any element of $M_B$ such that $M_B \models P_{\rho_{c,t} \rstr \ell}~(b_{c,t})$  for all $\ell < \omega$. 
Let $a_c = \langle a_{c,t} : t \in I \rangle/\de \in M^*_A$, and let $b_c = \langle b_{c,t} : t \in I \rangle/\de \in M^*/B$. 
[Without loss of generality, $a_c[t] = a_{c,t}$ and $b_c[t] = b_{c,t}$.] 
Consider 
\[ r_A = \{ Q_\eta(x) \} \cup \{ R(x,a_c) : c \in C \} \]
and consider 
\[ r_B = \{ Q_\eta(x) \} \cup \{ R(x,b_c) : c \in C \}. \]
Observe that $r_A(x)$ is a partial type in $M^*_A$ since $\mcr^{\xm[A]} \supseteq \mcr^{\xm[A \cup B]}$, and likewise 
$r_B(x)$ is a partial type in $M^*_B$ since 
$\mcr^{\xm[B]} \supseteq \mcr^{\xm[A \cup B]}$.  By our assumption on $\de$, $r_A$ is realized, say by $a_* \in M^*_A$, and 
$r_B$ is also realized, say by $b_* \in M^*_B$. 
Let us define $d_* : [C]^{<\aleph_0} \rightarrow \de$ to be the refinement of $d$ given by: 
\[ d_*(u) = d(u) \cap \{ t \in I :  M_A \models R(a_*[t], a_{c,t}) \} \cap \{ t \in I : M_B \models R(b_*[t], b_{c,t}) \}. \]
Thus, for each $t \in I$, we may define $C^*_t = \{ c \in C : t \in d_*(\{c\}) \}$.  It follows from the definition of $d_*$ that 
for each $t \in I$, $C^*_t \subseteq C_t$, \emph{thus} $C^*_t$ is finite; moreover, for each $t \in I$, 
\begin{equation}
\label{e:wa} M_A \models (\exists x)\left(Q_\eta(x) \land \bigwedge_{c \in C^*_t} R(x,a_{c,t})\right) 
\end{equation}
[in particular $a_*[t]$ is such a witness] and likewise
\begin{equation}
\label{e:wb} M_B \models (\exists x)\left(Q_\eta(x) \land \bigwedge_{c \in C^*_t} R(x,b_{c,t})\right) 
\end{equation} [in particular $b_*[t]$ is such a witness].
Fix $t \in I$. We now aim to prove: 
\begin{equation}
\label{e:AB} M_{A \cup B} \models (\exists x)\left(Q_\eta(x) \land \bigwedge_{c \in C^*_t} R(x,c[t]))\right). 
\end{equation} 
Recall that 
\begin{enumerate}
\item[(i)] $\mct^{\xm[A \cup B]}_1 = \mct^{\xm[A]}_1 = \mct^{\xm[B]}_1$ 
\item[(ii)] $\mcr^{\xm[A \cup B]}$ is contained in each of $\mcr^{\xm[A]}$ and $\mcr^{\xm[B]}$.  
\item[(iii)] for each $\ell < \omega$, for some, equivalently every, choice of\footnote{Note there is always some such choice, by the extension axioms.}
\\ $(\eta, \rho) \in \mcr^{\xm[A \cup B]}_\ell$,   $(\eta_A, \rho_A) \in \mcr^{\xm[A]}_\ell$,  $(\eta_B, \rho_B) \in \mcr^{\xm[B]}_\ell$
\\ we have that:
if $\ell \in A$, then for all $i,j < m_\ell$, 
\[ (\eta^\smallfrown \langle i \rangle, \rho^\smallfrown \langle j \rangle) \in \mcr^{\xm[A \cup B]} \mbox{ if and only if }
({\eta_A}^\smallfrown \langle i \rangle, {\rho_A}^\smallfrown \langle j \rangle) \in \mcr^{\xm[A]} \]
and if $\ell \in B$, then for all $i,j < m_\ell$, 
\[ (\eta^\smallfrown \langle i \rangle, \rho^\smallfrown \langle j \rangle) \in \mcr^{\xm[A \cup B]} \mbox{ if and only if }
({\eta_B}^\smallfrown \langle i \rangle, {\rho_B}^\smallfrown \langle j \rangle) \in \mcr^{\xm[B]} \]
and otherwise, if $\ell \in \omega \setminus A \cup B$, for all $i,j < m_\ell$, 
\[ (\eta^\smallfrown \langle i \rangle, \rho^\smallfrown \langle j \rangle) \in \mcr^{\xm[A \cup B]}. \]
[Of course if $\ell \in A \cap B$ then the two relevant conditions hold simultaneously.] 
\end{enumerate}
Recall that for each $c \in C^*_t$ we had defined its leaf $\rho_{c,t}$. 
It would suffice to prove that there is $\eta_* \in \lim(\mct^{\xm[A \cup B]}_1)$ so that 
\begin{equation}
\label{e:eta-star} (\eta_* \rstr s, \rho_{c,t} \rstr s) \in \mcr^{\xm[A \cup B]} \mbox{ for all $c \in C^*_t$ and all $s < \omega$}. 
\end{equation}

Let $\eta_a$ be the leaf of $a_*$, i.e. the unique element of $\lim(\mct^{\xm[A]}_1)$ such that $M_A \models Q_{\eta_a \rstr \ell}(a_*[t])$ for all 
$\ell < \omega$, and let $\eta_b$ be the leaf of $b_*[t]$, i.e. the unique element of $\lim(\mct^{\xm[B]}_1)$ such that 
$M_B \models Q_{\eta_b \rstr \ell} (b_*[t])$ for all $\ell < \omega$. So necessarily $\eta \tlf \eta_a$ and $\eta \tlf \eta_b$.
Suppose first that $\eta_a = \eta_b$. Let $\eta_* \in \lim(\mct^{\xm[A\cup B]}_1)$ 
be given by $\eta_* = \eta_a = \eta_b$. This $\eta_*$ satisfies (\ref{e:eta-star}), 
as is easy to verify by inductively applying (iii) above.  

If not,  suppose that there is some $\ii < \omega$ minimal for the property that\footnote{notice that $\eta_a, \eta_b$ satisfy these conditions 
for $\ii=0$}
we have $\eta^\ii_a \in \lim(\mct^{\xm[A]}_1)$, $\eta^\ii_b \in \lim(\mct^{\xm[B]}_1)$ such that:
\begin{enumerate}
\item $\eta^\ii_a \rstr \ii = \eta^\ii_b \rstr \ii$. 
\item $(\eta^\ii_a \rstr s, ~\rho_{c,t} \rstr s) \in \mcr^{\xm[A]}$ for each $c \in C^*_t$
\item $(\eta^\ii_b \rstr s, ~\rho_{c,t} \rstr s) \in \mcr^{\xm[B]}$ for each $c \in C^*_t$
\end{enumerate}
and let us prove that we can define $\eta^{\ii+1}_a$, $\eta^{\ii+1}_b$ so that 
$\eta^\ii_a \rstr \ii \tlf \eta^{\ii+1}_a$, $\eta^\ii_b \rstr \ii \tlf \eta^{\ii+1}_b$, and properties (1),(2),(3) hold with $\ii+1$ in place of $\ii$.
(By continuing this process one therefore eventually obtains two equivalent sequences.)

Write $\eta_a = {\eta_{a,0}}^\smallfrown \langle i_a \rangle^\smallfrown \eta_{a,\infty}$, and 
$\eta_b = {\eta_{b,0}}^\smallfrown \langle i_b \rangle^\smallfrown \eta_{b,\infty}$, where $\lgn(\eta_{a,0}) = \lgn(\eta_{b,0}) = \ii$. 
By definition, $\eta_{a,0} = \eta_{b,0}$ and for every $c \in C^*_t$, 
\[ (\eta_{a,0}, \rho_{c,t}) \in \mcr^{\xm[A]} \cap \mcr^{\xm[B]}. \]
There are three cases.
\begin{enumerate}
\item[(Case 1)] $\ii \notin B$.
In this case, $\ii$ is not an active level for $B$, so we define $\eta^{\ii+1}_a = \eta^\ii_a$, and define $\eta^{\ii+1}_b = 
{\eta_{b,0}}^\smallfrown \langle i_a \rangle^\smallfrown \eta_{b,\infty}$, i.e. replace $i_b$ by $i_a$. 
[Since we defined $\xm[A]$, $\xm[B]$ using the same background sequence of graphs, it doesn't matter whether $\ii \in A$ or not, 
recalling Remark \ref{rmk:last}. 
So applying 
Claim \ref{d:forget}, we conclude that 
$(\eta^{\ii+1}_b \rstr s, \rho_{c,t} \rstr s) \in \mcr^{\xm[B]}$ for each $c \in C^*_t$.
\item[(Case 2)] $\ii \notin A$.
In this case, since $\ii$ is not an active level for $A$, define $\eta^{\ii+1}_b = \eta_b$, and define $\eta^{\ii+1}_a = 
{\eta_{a,0}}^\smallfrown \langle i_b \rangle^\smallfrown \eta_{a,\infty}$, i.e. replace $i_a$ by $i_b$, and again use 
Remark \ref{rmk:last} (if necessary) and Claim \ref{d:forget}.
\item[(Case 3)] $\ii \in A \cap B$.  If $i_a = i_b$, define $\eta^{\ii+1}_a = \eta^\ii_a$ and 
$\eta^{\ii+1}_b = \eta^\ii_b$. If $i_a \neq i_b$, then since we 
defined $\xm[A]$, $\xm[B]$ using the same background sequence of graphs, by Claim \ref{d:forget} we may 
without loss of generality use $i_a$: that is, define $\eta^{\ii+1}_a = \eta_a$, and define $\eta^{\ii+1}_b = 
{\eta_{b,0}}^\smallfrown \langle i_a \rangle^\smallfrown \eta_{b,\infty}$.
\end{enumerate}
\noindent In this way we eventually construct two equal sequences, so $\eta_*$ is well defined, so (\ref{e:eta-star}) is satisfied, and as this was sufficient 
to prove (\ref{e:AB}), we are done. 
\end{proof}

\section{Further discussion and open questions} \label{s:open}

In the late sixties when Keisler's order was defined, it was natural to conjecture that it 
had a small finite number of classes (see the introduction to \cite{MiSh:1050}). Though it was quickly understood 
that the order might give an interesting calibration of theories (see \cite{keisler} and also \cite{morley}),  it long remained reasonable to believe that 
the order's power to give model theoretic information would be tied to its simplicity.   
We are now at a surprising mathematical juncture, where the order has become very complicated, but without losing its 
tight connection to and calibration of model-theoretic structure.  
To communicate some of our excitement, we include a broad list of questions. 

\subsubsection*{$A$. Saturated models of simple theories.} 
Determining Keisler's order on the stable theories required developing the stability theory 
to prove a characterization of the 
saturated models of stable theories (see \cite{Sh:c} Theorem III.3.10 and \cite{MiSh:1030} Question 10.4): essentially, that for a model of a complete countable 
stable theory to be $\lambda^+$-saturated it suffices that it is $\aleph_1$-saturated and that every maximal indiscernible 
set is large.  [The theorem is stronger: $\aleph_1$-saturated is really $\kappa(T)$-saturated and the theory need not be countable;
for us, regular ultrapowers of models of countable 
theories are $\aleph_1$-saturated, so this statement suffices.]  What, in simple unstable theories, are the right analogues of maximal indiscernible sets?

\begin{prob} \label{prob5}
In light of the results of this paper, formulate a plausible conjecture of a characterization of $\lambda^+$-saturated models of simple theories. 
\end{prob}

\subsubsection*{$B$. Variants of the construction}

\begin{disc}
\emph{Our main construction fixes $\bar{m}$, $\bar{E}$, and $\Xi$; varying these inputs one would have different theories.}  
\end{disc}

\begin{disc}
\emph{We have written the present construction for a single fast sequence $\bar{m}$ and a family of independent level functions. This was 
a decision to make the structure of the ideal clearer, among other things. But we might also have written the construction, without level 
functions, simply for 
continuum many sequences growing at very different rates $($which is, in some sense, what the level functions were formally coding$)$. 
Looking from this second point of view may give a different perspective on how growth rates of finite families affect model theoretic structure.}
\end{disc}

\subsubsection*{$C$. Interactions with forking.}

\begin{qst}
For every low simple $T$, is there a ``very simple'' $T$ equivalent to it, for example, a theory which is simple rank one? 
\end{qst}

\begin{qst} \label{q:non-low}
Can the continuum many incomparable classes be reproduced, in ZFC, within the simple non low theories?  Within the non simple theories? 
\end{qst}

\begin{disc}  For $\ref{q:non-low}$, it may be reasonable to consider set theoretic hypotheses such as a measurable or supercompact 
cardinal, recalling \cite{MiSh:1030}. 
We may also ask the parallel question for $\tlf^*$. 
\end{disc}

\br

For computability theorists, a natural question may be:

\begin{qst}
Is the structure of the Turing degrees embeddable into Keisler's order? 
\end{qst}

\subsubsection*{$E$. Questions about ultrafilter construction.}

An important part of the argument above is constructing ultrafilters, and it may be fruitful to 
further investigate methods from iterated forcing.

\begin{disc} \label{d:low}
\emph{Recall that one way of measuring ``size'' of a regularizing family $\{ X_\alpha : \alpha < \lambda \}$ in an ultrafilter $\de$ on $\lambda$ 
is to look at the sequence of integers $\{ n_t : t \in \lambda \}$ where $n_t = \{ \alpha < \lambda : t \in X_\alpha \}$.  Say that a regularizing family is 
\emph{below} a nonstandard integer if its size is. Flexible ultrafilters  
are those having a regularizing family below any nonstandard integer \cite{mm5}. 
Each of the ultrafilters we build here, by virtue of its connection to certain integer sequences, has a certain amount of flexibility appropriate to those 
sequences.  There remain very interesting open questions about the extent to which flexibility $($which is equivalent to ``OK''$)$ 
may be separated from goodness, such as 
Dow's 1985 question, for references and recent work see e.g. Problems 3.5 and 3.6 of \cite{MiSh:1069}.  It may be interesting to investigate whether the new family of 
filters built here, of apparently intermediate flexibility, sheds light on the landscape around these questions, as our methods suggest  
further ways of engineering the relation of ``sizes'' of filters and of sequences.}
\end{disc}

\subsubsection*{$F$. The minimal simple class, and the maximal class.}

Recall that the theory $\trg$ of the random graph is minimum among the unstable theories in Keisler's order.  
There is a set-theoretic characterization of its class (i.e., there is a necessary and sufficient condition for regular ultrafilters 
to be good for $\trg$), but to date there is no model-theoretic characterization, indeed no model theoretic characterization 
of \emph{any} unstable equivalence class. 
A natural place to begin is: 

\begin{prob} \label{p13} 
Give a model theoretic characterization of the class of the theory of the random graph in Keisler's order. 
\end{prob}

\noindent Any reasonable list of open problems on Keisler's order should recall the parallel of \ref{p13},  
one of the major questions on the table. (See \cite{MiSh:998}.) 

\begin{prob}
Give a model theoretic characterization of the maximal class in Keisler's order, and under $\tlf^*$ without instances of GCH. 
\end{prob}

\subsubsection*{$G$. Variants of Keisler's order.} 

It is a very interesting and natural question to consider what less fine variants of Keisler's order may show about the structure of simple theories, 
and whether such variants may be found whose number of equivalence classes is finite. 
We plan to say more about this in future papers. 

\subsubsection*{$H$. Building blocks of simple theories.}

\begin{disc}
\emph{These theories we have built appear quite different from the theories witnessing the infinite descending chain in Keisler's order, which were sums of certain 
generic $n$-free $k$-hypergraphs, studied originally by Hrushovski \cite{h:letter} (for \cite{MiSh:1050}, we used the case $n=k+1$).  
Such theories may be thought of as encoding ``pure amalgamation problems.''  }
\end{disc}

{Indeed, we originally built the precursor to the present theories 
in \cite{MiSh:1140} to witness Keisler-incomparability with the $T_{k+1,k}$'s. 
The role of the new theory in \cite{MiSh:1140}, this precursor of the $T_\xm$'s, 
was in some sense to replace a certain canonical non-low theory in the known, non-ZFC 
incomparability arguments \cite{ulrich}, \cite{MiSh:F1530}.  We might describe these theories 
as containing enough of a finite approximation to forking to 
retain incomparability, but without actually forking.  
We verified in $\ref{k11}$ above that the theory of \cite{MiSh:1140} fits in the present framework, though the background Boolean algebras in the 
two papers are quite different.  
What does this picture tell us about the building blocks of complexity in simple theories? What interesting non-trivial interactions may occur within simplicity between the weak avatars of forking $($the uniform inconsistency along various finite quotient sets$)$ in the $T_\xm$'s, and the 
inconsistency arising from amalgamation in the $T_{n,k}$'s?}

\subsubsection*{$I$. Hypergraph regularity.}

Not unrelated to Problem \ref{prob5}, the above discussion of graphs and hypergraphs suggests the following 
speculative question. The theories we have built in the present 
paper are really fundamentally graphs (layered across predicates).  The key relation is a binary relation, and the key underlying densities 
are densities of bipartite graphs.  Is uniform incomparability across a family necessarily a graph (binary) phenomenon? 
Recalling that the hypergraph analogues of 
Szemer\'edi phenomena are known \cite{gowers}, \cite{rnssk}, \cite{tao} we may ask:\footnote{Added in revision: Part of this question, constructing theories which are hypergraph analogues of 
the theories $T_\xm$, is now addressed in \cite{MiSh:1206}.}
 
\begin{qst} \label{q5}
Is there a true ``hypergraph analogue'' of our construction?  For instance, can one construct a family of simple theories whose only forking comes from equality, by analogy to what we have done here, 
which reflect in some fundamental way the densities of certain families of finite 3-uniform hypergraphs, and which themselves 
form a higher layer of uniform incomparability phenomena in Keisler's order which is not explained by their restrictions to graphs?\end{qst}

Understanding in either direction may significantly change our understanding of dividing lines in simple theories.

\end{document}